\documentclass[oneside,a4paper]{amsart}

\newif\ifpdflatex
\ifx\pdfoutput\undefined
  \pdflatexfalse
\else
  \pdfoutput=1
  \pdflatextrue
\fi

\ifpdflatex
  \usepackage[pdftex]{epsfig}
  \usepackage[pdftex,colorlinks=true,%
              pdftitle={Stability of peak solutions of a non-linear
                        transport equation on the circle},%
              pdfauthor={Edith Geigant and Michael Stoll}]{hyperref}
  \newcommand{\Gr}[2]{\psfig{file=#1.pdf,width=#2}}
\else
  \usepackage{epsfig}
  \newcommand{\Gr}[2]{\psfig{file=#1.eps,width=#2}}
\fi

\usepackage{amssymb}
\usepackage{comment}

\newcommand{\Z}{{\mathbb Z}}
\newcommand{\R}{{\mathbb R}}

\newcommand{\D}{{\mathcal D}}
\newcommand{\C}{{\mathcal C}}

\newcommand{\Int}{\int_{S^1}}

\newcommand{\half}{\frac{1}{2}}
\newcommand{\thalf}{\tfrac{1}{2}}

\newcommand{\dfdt}{\frac{\partial\! f}{\partial t}}
\newcommand{\eps}{\varepsilon}
\newcommand{\supp}{\operatorname{supp}}
\newcommand{\dist}{\operatorname{dist}}
\newcommand{\sign}{\operatorname{sign}}
\newcommand{\Vol}{\operatorname{Vol}}

\renewcommand{\Re}{\operatorname{Re}}

\newtheorem{theorem}{Theorem}[section]
\newtheorem{proposition}[theorem]{Proposition}
\newtheorem{corollary}[theorem]{Corollary}
\newtheorem{lemma}[theorem]{Lemma}

\theoremstyle{remark}

\newtheorem{remarks}[theorem]{Remarks}

\theoremstyle{definition}

\newtheorem{definition}[theorem]{Definition}
\newtheorem{example}[theorem]{Example}

\begin{document}

\title[Stability of peak solutions]%
      {Stability of peak solutions \\ of a non-linear transport equation on the circle}
\author{Edith Geigant}
\address{Bayreuth}
\email{egeigant@gmx.de}
\author{Michael Stoll}
\address{Mathematisches Institut, Universit\"at Bayreuth, 95440 Bayreuth, Germany}
\email{Michael.Stoll@uni-bayreuth.de}
\date{August 27, 2011}


\maketitle

\section{Introduction}

In this paper we analyze the pattern forming ability and pattern stability
for a one-dimensional non-linear transport-diffusion equation on the circle.
The distinguishing feature of this equation is the non-local turning velocity
that is determined by interactions between particles in various orientations ---
velocity is given by a convolution term of an interaction rate \( V \) with the
distribution function. In its general form, it also includes a diffusion term.

In Section~\ref{S:basic} we establish some basic facts on the transport-diffusion
equation, like conservation of mass and symmetries, correspondence between solutions
of higher periodicity for \( V \) and general solutions for its `rolled-up' version
\( V_n \). We also discuss the corresponding
equation on the real line and its relation with the equation on the circle.
Linearization near the constant stationary state provides conditions on the
interaction rate \( V \) and on the smallness of the diffusion coefficient such
that non-constant stationary states exist.
These conditions are related to some conditions of Primi et~al.~\cite{Primi},
whose work deals with the same transport-diffusion equation.
But Primi et~al.\ succeeded in proving existence of \( n \)-peaks like steady states
(for \( n \ge 1 \)) using {\em weaker} conditions on \( V \) (so their result is
stronger) using a method completely independent of linearization.
Finally, we show that the constant stationary state is globally stable if the
diffusion coefficient is large enough or the interaction rate small enough.
For this proof we use the Fourier transformed version of the transport-diffusion equation.

The method used by Primi et~al.~\cite{Primi} to construct stationary solutions
allows no conclusions on the stability of the pattern.
Mogilner et~al.~\cite{Mogilner2} argued that without diffusion a single peak is stable
if the interaction rate is everywhere attracting (this is a positivity condition
on \( V \)).
They used a discrete setting with peaks of equal masses.
In Section~\ref{stability-of-position} we generalize that method to peaks with
different masses and analyze the stability of \( n \ge 1 \) equidistant peaks
of equal mass by a linearization argument.

The transport-diffusion equation is closely related to a non-linear
integro-differential equation on \( S^1 \), in which the non-linearity comes from
interactions between particles, and in which particles jump instantaneously from
one orientation to another.
We will comment on these relations in the discussion in Section~\ref{discussion},
which concludes this paper.
For the limit of exact alignment in this jump process the stability of a single peak
has been analyzed by linearization near the peak, see Geigant~\cite{Geigant2000}.
In Section~\ref{small-perturbations} we use linearization of the transport equation
(without diffusion) near \( n \)-peak solutions to see whether they are stable or not.
It turns out that a single peak is stable if \( V \) is everywhere attracting and that
two opposite peaks are stable {\em up to changes in masses} if \( V \) unites
attracting and repelling features.
Most interestingly, \( n \ge 3 \) equidistant peaks of equal masses are in general
{\em not} stable.
Under appropriate conditions on \( V \) the number of peaks is invariant, but
positions and masses of the peaks may change slightly.
A technical difficulty is that solutions are invariant with respect to translations,
therefore `stability' always means stability {\em up to translations}.

The integro-differential equation with exact alignment has also been studied by
Kang et~al.~\cite{Kang}.
They find that solutions converge in the sense of distributions to peak solutions
if the support of the starting function consists of disjoint, small intervals
and if the interaction is attracting.
In Section~\ref{compact-support} we prove that the same holds true for the transport
equation, but for the transport equation the assumptions on \( V \) are more local.

An important fact in all proofs of convergence to peaks is invariance of the first
moment or barycenter. However, the first moment is only invariant for the linearized
equations or if the initial function has sufficiently small support.
Example~\ref{Ex:moment} in Section~\ref{stability-of-position}
shows that the first moment (if defined in a naive way) is in general {\em not}
invariant. This is caused by the discontinuity of the function \( \theta \)
given by identifying \( S^1 \) minus a point with $\left]-\thalf,\thalf\right[$.
The same is true for the related integro-differential equation.%
\footnote{Therefore we think that equation~(26) in Kang et~al.~\cite{Kang} and
conclusions based on~(26) are not correct. The authors have informed
us that an erratum is in preparation.}
In Section~\ref{S:bary} we show how barycenters can be defined locally and
when they are invariant.

Section~\ref{numerical-algorithms} presents an algorithm for fast computation
of solutions for the transport-diffusion equation.
It is based on the Fourier transform of the transport-diffusion equation.
The same method has been used by Geigant and Stoll~\cite{GeigantStoll}
for the integro-differential equation.
A second algorithm, namely an iteration scheme, is used to find stationary solutions
of the transport-diffusion equation.
It is essentially the scheme for which Primi et~al.~\cite{Primi} prove convergence
for suitable turning rates \( V\!\), and indeed, in our computations we observe
that their assumptions on \( V \) are necessary for convergence.

In Section~\ref{examples} we present a number of examples obtained using our
numerical algorithms.
The first example shows the simultaneous bifurcation of first and second mode;
near that bifurcation point a stable mixed-mode solution and a backward bifurcation
exist.
The second example shows a stable two-peaks like solution where the two peaks are
not opposite.
In the third example {\em stable} one-peak and two-peaks like solutions exist
at the same parameter values.
This contradicts a conjecture of Primi et~al.~\cite{Primi} that a certain feature
of \( V \) (basically its shape and the sign of the primitive) allows conclusions
on the number of developing peaks, namely one versus two.
In the last example we show that pattern formation may occur even if \( V \)
is nowhere attracting.

%================================================================================

\section{The non-linear transport equation with diffusion} \label{S:basic}

Let \( S^1 = \R / \Z \) be a circle of length~1.
If we denote by $p : \R \to S^1 = \R/\Z$ the canonical projection, then we have associated maps $p^*$ from functions on~$S^1$ to functions on~$\R$, where $p^*(f) = f \circ p$ is the associated 1-periodic function on~$\R$, and $p_*$ from (sufficiently fast decaying) functions on~$\R$ to functions on~$S^1$, where
\[ p_*(g)(\theta) = \sum_{x \in \R, p(x) = \theta} g(x) \,. \]

\begin{definition} \label{D:interval}
  A \emph{closed interval}~$I$ on~$S^1$ is a closed connected subset that is not all
  of~$S^1$. Then $I = p(I')$ for some closed interval $I' = [a,b] \subset \R$
  (such that $b-a < 1$), and we write $I = [p(a),p(b)]$ and call $\alpha = p(a)$ the
  \emph{lower end} and $\beta = p(b)$ the \emph{upper end} of~$I$.
  If $h$ is a function on~$S^1$, we write
  \[ \int_a^b h(\theta)\,d\theta
       = \int_\alpha^\beta h(\theta)\,d\theta
       = \int_{[\alpha,\beta]} h(\theta)\,d\theta
       = \int_a^b p^*(h)(x)\,dx .
  \]
  If $\theta, \psi \in I$, we write $\theta-\psi \in \R$ for the difference
  $\theta'-\psi'$ where $\theta', \psi' \in I'$ are such that
  $p(\theta') = \theta$, $p(\psi') = \psi$.
\end{definition}

If $V : S^1 \to \R$ is a function and $I = \left]a,b\right[ \subset \R$ is an interval
such that $p(I) \neq S^1$, we will (for simplicity) say that
`$V > 0$ on~$\left]a,b\right[$' if $V > 0$ on $p(I)$ (equivalently, $p^*(V) > 0$
on~$I$); similarly for half-open or closed intervals. In the same way, we
write $V(a)$ for $V(p(a))$ if $a \in \R$.

%================================================================================

\subsection{The equation on the circle} \label{S:eqncircle} \strut

We want to model a process that describes the change of orientation of filaments
over time. The orientation is given by an `angle' $\theta \in S^1$.
The density of filaments at time \( t \ge 0 \) with orientation \( \theta \in S^1 \)
is given by \( f(t,\theta) \).
The filaments turn continuously where the velocity of turning is determined by interactions with other filaments on the circle.
At the same time there is random reorientation.
This kind of dynamics is described by the following transport equation with diffusion.
\begin{equation} \label{tp}
  \dfdt (t,\theta)
    = D \frac{\partial^2 f}{\partial \theta^2}(t,\theta)
        + \frac{\partial}{\partial \theta}
             \bigl((V*f(t,\cdot)) \cdot f(t,\cdot)\bigr)(\theta) ,
\end{equation}
where \( D \ge 0 \) is the diffusion coefficient and
\( (V*f)(\theta) = \Int V(\theta-\psi) f(\psi) \,d\psi \) is the convolution
of~\( V \) with~\( f \) and gives the negative velocity of turning of filaments with orientation~\( \theta \).

We assume that the interaction function \( V : S^1 \to \R \) is {\em odd}, because interactions with filaments on opposite sides of \( \theta \) must have similar consequences.
In particular, \( V(0) = 0 \), i.e.,~there is no repulsion or attraction of filaments with the same orientation, and \( V(\half) = 0 \), i.e.,~there is no interaction with filaments of opposite orientation.
The sign of \( V(\theta) \) is important.
If \( V(\theta) > 0 \) for some interaction angle \( 0 < \theta < \half \) then the two filaments move towards each other, we call this `attracting'; if on the other hand \( V(\theta) < 0 \) for some \( 0 < \theta < \half \) then the distance between the filaments becomes greater, they are `repelling each other'.
For odd \( V \in \C^{\infty}(S^1) \) Primi et~al.~\cite{Primi} prove a-priori estimates by which unique existence of smooth solutions of equation~\eqref{tp} can be shown.

The following easy statement will be useful.

\begin{lemma} \label{L:symmint}
  Let $V, f \in \C(S^1)$ with $V$ odd. Then
  \[ \int_{S^1} (V * f)(\theta) f(\theta)\,d\theta = 0 . \]
\end{lemma}

\begin{proof}
  We have
  \[ \int_{S^1} (V * f)(\theta) f(\theta)\,d\theta
       = \int_{S^1} \int_{S^1} V(\theta-\psi) f(\psi) \,d\psi\,f(\theta)\,d\theta
       = \int_{S^1} \int_{S^1} V(\theta-\psi) f(\psi) f(\theta) \,d\psi\,d\theta .
  \]
  If we swap $\psi$ and~$\theta$ in the last integral, it changes sign (since
  $V$ is odd); therefore it must be zero.
\end{proof}

The following proposition states some basic facts on equation~\eqref{tp}.

\begin{proposition} \label{invariance}
  Equation~\eqref{tp} preserves mass, non-negativity, axial symmetry with respect
  to any axis and periodicity of initial functions.
  Moreover, the solution space is invariant under the group~$O(2)$ of translations
  and reflections on~$S^1$.
\end{proposition}

\begin{proof}
  Preservation of mass and positivity are shown by Primi et~al.~\cite{Primi}.
  The remaining statements follow from the observation that the operator
  on the right hand side of equation~\eqref{tp} is $O(2)$-equivariant
  (for the reflections in~$O(2)$, this uses that $V$ is odd).
\end{proof}

Since the partial differential equation~\eqref{tp} lives on~$S^1$, the equation
turns into a discrete system of ODEs when it is Fourier transformed.
We denote by
\[ f_k = \int_{S^1} f(\theta) e^{-2 \pi i k \theta} \,d\theta
   \quad \text{for $k \in \Z$}
\]
the \( k \)-th Fourier coefficient of \( f : S^1 \to \R \).
Since \( f \) is real, \( f_k = \bar{f}_{-k} \);
if \( f \) is even or odd, then \( f_k \in \R \) or \( f_k \in i \R \), respectively.
For differentiable functions one has
\( (f')_k = 2 \pi i k f_k \),
the Fourier coefficients of a convolution are
\( (V*f)_k = V_k f_k \),
and the Fourier transform of a product is the convolution of the Fourier series,
\( (f \cdot g)_k = \sum_{l \in \Z} f_l g_{k-l} \).

Hence the Fourier transform of the transport-diffusion equation~\eqref{tp} is
\begin{align}
  \dot{f}_k(t)
    &= -(2 \pi k)^2 D f_k + 2 \pi i k \sum_{l \in Z} V_l f_l f_{k-l} \nonumber \\
    &= c_k \,f_k + 4 \pi k \, \sum_{l \in \Z, l \neq 0,k} v_l f_l f_{k-l}
         \quad \text{for $k \in \Z$,}
      \label{fou_k}
\end{align}
where the eigenvalues \( c_k \in \R \) of the system and the \( v_k \in \R \) are defined as
\begin{equation}
\label{ev}
  c_k = -(2 \pi k)^2 D + 4 \pi k f_0 v_k
  \quad \mbox{and} \quad
  v_k = \int_{0}^{\half} V(\theta) \sin(2 \pi k \theta) \,d\theta = \tfrac{i}{2} V_k .
\end{equation}
Mass conservation is reflected by the equation \( \dot{f}_0 = 0 \).
Using $f_{-k} = \bar{f}_k$, the equations with $k < 0$ are redundant.
We have \( c_k = c_{-k} = \bar{c}_k \in \R \) because $v_k = -v_{-k}$.
By scaling \( D \) and \( V \) one may assume that the mass is~1,
i.e.,~\( f_0 = \Int f(t,\theta) \,d\theta = 1 \) for all \( t \ge 0 \),
which we will do from now on.

%================================================================================

\subsection{Stationary solutions} \strut

In this subsection, we collect some results on stationary solutions of
equation~\eqref{tp}. We begin with some estimates on the shape of a
stationary solution.

\begin{proposition}[Estimates for stationary solutions] \label{estimates}
  Any stationary solution of equation~\eqref{tp} with $D > 0$ satisfies the following
  ordinary differential equation on \( S^1 \).
  \begin{equation} \label{stateq}
    D \frac{d f}{d \theta} = -(V*f) \cdot f .
  \end{equation}
  Let $C = \max V/D$ and assume that $f \ge 0$ is a stationary solution
  of~\eqref{tp} with mass~1. Then for $\theta_1, \theta_2 \in S^1$ we have
  \[ f(\theta_2) \le f(\theta_1) e^{C |\theta_1 - \theta_2|}
     \quad\text{and}\quad
     f(\theta_2) \ge f(\theta_1) e^{-C |\theta_1 - \theta_2|} .
  \]
  In particular,
  \[ \frac{\max f}{\min f} \le e^{C/2} \]
  and therefore
  \( \quad   \max f \le e^{C/2} \min f \le e^{C/2}
     \quad\text{and}\quad \min f \ge e^{-C/2} \max f \ge e^{-C/2} .
  \)

  In any maximum \( \theta_{\max} \) of a stationary solution we have
  \[ \frac{d^2 f}{d \theta^2}(\theta_{\max}) \ge - \frac{\max V'}{D} f(\theta_{\max}) ;
  \]
  in any minimum \( \theta_{\min} \) of a stationary solution we have
  \[ \frac{d^2 f}{d \theta^2}(\theta_{\min}) \le - \frac{\min V'}{D} f(\theta_{\min}) ;
  \]
\end{proposition}

If $D = 0$ in equation~\eqref{tp}, then $f$ is a stationary solution if and
only if $(V*f) \cdot f = 0$.

These results can be interpreted as follows.

i) If the diffusion coefficient \( D \) is large compared to \( V\!\), then any stationary solution is near the constant solution (or only the constant solution exists).

ii) The inequalities for \( \max f \) have a large right hand side when \( D \)
becomes small.
That leads us to expect that with decreasing \( D \) solutions may become large and maxima may be sharp peaks (the curvature is large).

iii) If the minimum of \( f \) is small, then it is wide (the curvature is small).

\begin{proof}
  (See also Primi et~al.(\cite{Primi}) for the derivation of the ODE.)
  Let
  \( D \frac{d^2 f}{d \theta^2} + \frac{d}{d \theta} ((V*f) \, f) = 0 \).
  Then
  \( D \frac{d f}{d \theta} + (V*f) f = c \),
  where \( c \) is some arbitrary constant of integration.
  But the integral over \( S^1 \) of the first and second terms is zero
  (see Lemma~\ref{L:symmint} for the second term), so $c = 0$.

  If \( f \) is nonnegative, then
  \[ \frac{d f}{d \theta}(\theta)
      \le \frac{1}{D} \max V \int f(\psi) \, d\psi \, f(\theta)
      = C f(\theta) ;
  \]
  Therefore by integrating from $\theta_1$ to~$\theta_2$, we get
  \[ f(\theta_2) \le f(\theta_1) e^{C |\theta_2 - \theta_1|} ; \]
  the other inequality follows by symmetry. We can then take $\theta_1 = \theta_{\min}$
  and $\theta_2 = \theta_{\max}$ to be points where $f$ attains its minimum and
  maximum, respectively.

  Now, let \( \theta_{\max} \) be a maximum of $f$
  Because \( f'(\theta_{\max}) = 0 \) we have
  \[ \frac{d}{d \theta} ((V*f) \cdot f)(\theta_{\max})
       = (V'*f)(\theta_{\max}) f(\theta_{\max})
           + (V*f)(\theta_{\max}) f'(\theta_{\max})
       \le (\max V') f(\theta_{\max}) .
  \]
  Therefore,
  \[ 0 = D f''(\theta_{\max})
          + \partial_\theta ((V*f) \cdot f)(\theta_{\max})
       \le D f''(\theta_{\max}) + (\max V') f(\theta_{\max}) .
  \]
  If \( \theta_{\min} \) is a minimum of $f$,
  then the estimate in \( \theta_{\min} \) can be proved in the same way.
\end{proof}

Primi et al.~\cite{Primi} use the ODE~\eqref{stateq} to set up an iterative
procedure for approximating stationary solutions. See Section~\ref{SubS:iteration}
below.

We state a simple consequence of equation~\eqref{stateq}.

\begin{proposition} \label{P:perstat}
  Assume $D > 0$ in equation~\eqref{tp}. If $V$ is $\tfrac{1}{n}$-periodic
  and odd, then any stationary solution of~\eqref{tp} must also be
  $\tfrac{1}{n}$-periodic.
\end{proposition}

\begin{proof}
  $V(\theta + \tfrac{1}{n}) = V(\theta)$ for all $\theta \in S^1$ implies
  \[ (V*f)(\theta + \tfrac{1}{n})
       = \int_{S^1} V(\theta + \tfrac{1}{n} - \psi) f(\psi) \,d\psi
       = \int_{S^1} V(\theta - \psi) f(\psi) \,d\psi
       = (V*f)(\theta)
  \]
  for functions~$f$ on~$S^1$. If $f$ is a stationary solution of~\eqref{tp},
  then it is a solution of the ODE~\eqref{stateq}. Since $D > 0$, it follows
  that $f'/f = -\frac{1}{D} (V*f)$ is $\tfrac{1}{n}$-periodic. This implies
  that $f(\theta + \tfrac{1}{n}) = \gamma f(\theta)$ with some constant~$\gamma$,
  and since $f > 0$, we must have $\gamma > 0$. Since obviously $\gamma^n = 1$,
  we have $\gamma = 1$, and $f$ is $\tfrac{1}{n}$-periodic.
\end{proof}

%================================================================================

\subsection{Solutions with higher periodicity} \strut

Let \( n \ge 1 \) and \( V : S^1 \to \R \) be odd and continuous.
We will be interested in \( \tfrac{1}{n} \)-periodic solutions of equation~\eqref{tp}.
To understand these, the following functions \( V_n \) and~$\tilde{V}_n$
will be useful.
\[ %\begin{equation} \label{V_n}
  V_n(\theta) = \sum_{j=0}^{n-1} V\Bigl(\theta - \frac{j}{n}\Bigr), \qquad
  \tilde{V}_n(\theta) = \sum_{j=0}^{n-1} V\Bigl(\frac{\theta - j}{n}\Bigr) .
\] %\end{equation}
\( V_n \) and $\tilde{V}_n$ are continuous and odd; $V_n$ is $\tfrac{1}{n}$-periodic.
In particular,
\[ V_n\Bigl(\frac{k}{2 n}\Bigr) = 0 \quad \text{for $0 \le k < 2 n$.} \]

The following result shows that instead of considering $\frac{1}{n}$-periodic
solutions of equation~\eqref{tp}, we can consider solutions of~\eqref{tp}
without higher periodicity, when we modify the parameters $D$ and~$V$ accordingly.

\begin{proposition} \label{1/n-periodic}
  Let \( n \ge 1 \) and \( V \) be odd.
  Then there is a 1-to-1 correspondence between \( \tfrac{1}{n} \)-periodic solutions
  \( f \) of equation~\eqref{tp} and solutions \( \tilde{f} \) of
  \begin{equation} \label{E:per}
    \frac{\partial\!\tilde{f}}{\partial t}(t,\theta)
       = n^2 D \frac{\partial^2\!\tilde{f}}{\partial\theta^2}(t,\theta)
          + \frac{\partial}{\partial\theta}
              \bigl((\tilde{V}_n * \tilde{f}(t,\cdot)) \tilde{f}(t,\cdot)\bigl)(\theta),
  \end{equation}
  namely via \( \tilde{f}(t,n\theta) = f(t,\theta) \).
\end{proposition}

\begin{proof}
  Let \( f(t,\cdot) \) be a \( \tfrac{1}{n} \)-periodic solution of~\eqref{tp}.
  Define $\tilde{f}(t,\theta) = f(t,\tfrac{\theta}{n})$. Then $\tilde{f}$
  has mass~1, and we have
  \begin{align*}
    (V*f(t,\cdot))\Bigl(\frac{\theta}{n}\Bigr)
      &= \sum_{j=0}^{n-1} \int_{\frac{j}{n}}^{\frac{j+1}{n}}
                          V\Bigl(\frac{\theta}{n}-\psi\Bigr) f(t,\psi) \,d\psi
       = \sum_{j=0}^{n-1} \int_{0}^{\frac{1}{n}}
              V\Bigl(\frac{\theta}{n}-\Bigl(\psi+\frac{j}{n}\Bigr)\Bigr)
              f\Bigl(t,\psi+\frac{j}{n}\Bigr) \,d\psi \\
      &= \int_{0}^{\frac{1}{n}}
                 V_n\Bigl(\frac{\theta}{n}-\psi\Bigr) f(t,\psi) \,d\psi
       = \frac{1}{n} \int_{0}^1 V_n\Bigl(\frac{\theta-\psi}{n}\Bigr)
                                f\Bigl(t,\frac{\psi}{n}\Bigr) \,d\psi
       = \frac{1}{n} \bigl(\tilde{V}_n * \tilde{f}(t,\cdot)\bigr)(\theta) .
  \end{align*}
  Also,
  \( \frac{\partial\!\tilde{f}}{\partial \theta}(t,\theta)
      = \tfrac{1}{n} \frac{\partial\!f}{\partial \theta}(t,\frac{\theta}{n})
  \).
  Therefore,
  \begin{align*}
    \frac{\partial\!\tilde{f}}{\partial t}(t,\theta)
       = \dfdt\Bigl(t,\frac{\theta}{n}\Bigr)
      &= D \frac{\partial^2\!f}{\partial \theta^2} \Bigl(t,\frac{\theta}{n}\Bigr)
           + \frac{\partial}{\partial\theta} \bigl(V*f(t,\cdot) f(t,\cdot)\bigr)
              \Bigl(\frac{\theta}{n}\Bigr) \\
      &= n^2 D \frac{\partial^2\!\tilde{f}}{\partial \theta^2}(t,\theta)
           + n \frac{\partial}{\partial \theta}
                \Bigl(\frac{1}{n} (\tilde{V}_n*\tilde{f}(t,\cdot))
                        \tilde{f}(t,\cdot)\Bigr)(\theta) \\
      &= n^2 D \frac{\partial^2\!\tilde{f}}{\partial \theta^2}(t,\theta)
           + \frac{\partial}{\partial \theta}
              \bigl((\tilde{V}_n*\tilde{f}(t,\cdot)) \tilde{f}(t,\cdot)\bigr)(\theta) .
  \end{align*}
  The converse can be shown in the same way.
\end{proof}

This shows in particular that diffusion acts more strongly on solutions of
higher periodicity. In fact, we have $\max \tilde{V}_n \le n \max V\!$, so
the quotient $C = \max V/D$ in Proposition~\ref{estimates} will be multiplied
by a number $\le \tfrac{1}{n}$.

In terms of the Fourier transformed system~\eqref{fou_k}, we have
$f_k = 0$ for $n \nmid k$, $\tilde{f}_k = f_{nk}$, and
$(\tilde{V}_n)_k = n V_{nk}$. So we only look at the equations with $k$ a
multiple of~$n$ and replace $nk$ by~$k$ to obtain the system corresponding
to equation~\eqref{E:per}.

%================================================================================

\subsection{The equation on the real line} \strut

A similar PDE can also be considered with $\R$ instead of~$S^1$ as the spatial domain,
\begin{equation} \label{tpR}
  \frac{\partial g}{\partial t}(t,x)
    = D \frac{\partial^2 g}{\partial x^2}(t,x)
       + \frac{\partial}{\partial x}\bigl((W*g(t,\cdot)) \cdot g(t,\cdot)\bigr)(x) \,.
\end{equation}
Here $W : \R \to \R$ is odd and $g(t, \cdot)$ is assumed to decay sufficiently fast, so that the convolution $W*g(t,\cdot)$ is defined. Note that the convolution is here given by an integral over all of~$\R$.
There is the following relation between equations \eqref{tp} and~\eqref{tpR}.

\begin{proposition} \label{P:2eqns}
  Let $V : S^1 \to \R$ be odd and define $W = p^*(V)$ (which is just $V$ considered
  as a 1-periodic function on~$\R$). Let $g : \left[0,T\right[ \times \R \to \R$
  be a solution of equation~\eqref{tpR}.
  Then $(t,\theta) \mapsto f(t,\theta) = p_*(g(t,\cdot))(\theta)$ is a solution of
  equation~\eqref{tp}.
\end{proposition}

\begin{proof}
  We have
  \begin{align*}
    \frac{\partial\! f}{\partial t}(t,\theta)
      &= \sum_{x: p(x)=\theta} \frac{\partial g}{\partial t}(t, x) \\
      &= \sum_{x: p(x)=\theta}
          \Bigl(D \frac{\partial^2 g}{\partial x^2}(t,x)
                 + \frac{\partial}{\partial x}\bigl((W*g(t,\cdot))(x) \cdot g(t,x)\bigr)
               \Bigr) \\
      &= D \frac{\partial^2\! f}{\partial \theta^2}(t, \theta)
           + \frac{\partial}{\partial \theta}
               \Bigl(\sum_{x: p(x)=\theta}
                        \int_{-\infty}^{\infty} W(x-y) g(t,y)\,dy \cdot g(t,x)\Bigr) \\
      &= D \frac{\partial^2\! f}{\partial \theta^2}(t, \theta)
           + \frac{\partial}{\partial \theta}
               \Bigl(\int_{S^1} V(\theta-\psi) f(t,\psi)\,d\psi
                      \cdot \sum_{x: p(x)=\theta} g(t,x) \Bigr) \\
      &= D \frac{\partial^2\! f}{\partial \theta^2}(t, \theta)
           + \frac{\partial}{\partial \theta}
                \bigl((V*f(t,\cdot)) f(t,\cdot)\bigr)(\theta) \,.
  \end{align*}
  Here we use that
  \[ \int_{-\infty}^{\infty} p^*(V)(x-y) h(y)\,dy
      = \int_{S^1} V(p(x)-\psi) p_*(h)(\psi)\,d\psi \,. \qedhere
  \]
\end{proof}

The advantage of equation~\eqref{tpR} over~\eqref{tp} is that it is easily shown
to not only preserve mass, but also the first moment (or, equivalently, the
barycenter) of~$g(t,\cdot)$, whereas the notion of `first moment' usually does
not even make sense on~$S^1$.

\begin{proposition} \label{P:tpRbasic}
  Let $g$ be a solution of~\eqref{tpR}.
  \begin{enumerate}
    \item For any $a \in \R$, $(t,x) \mapsto g(t,x-a)$ is again a solution
          of~\eqref{tpR}.
    \item $(t,x) \mapsto g(t,-x)$ is again a solution of~\eqref{tpR}.
    \item $\displaystyle\int_{-\infty}^\infty g(t,x)\,dx$ is constant.
    \item $\displaystyle\int_{-\infty}^\infty x g(t,x)\,dx$ is constant.
  \end{enumerate}
\end{proposition}

\begin{proof}
  The first statement is clear (the operator on the right hand side is equivariant
  with respect to translations). Since $W$ is assumed to be odd, the right hand
  side is also equivariant with respect to $x \mapsto -x$, which implies the
  second statement. For the third statement, we compute
  \[ \frac{d}{dt} \int_{-\infty}^\infty g(t,x)\,dx
      = \int_{-\infty}^\infty
           \frac{\partial}{\partial x}
            \Bigl(D \frac{\partial g}{\partial x}(t,x)
                   + (W * g(t,\cdot))(x) g(t,x) \Bigr) \,dx \\
      = 0 \,,
  \]
  using the decay properties of~$g$. For the last statement, we have
  \begin{align*}
    \frac{d}{dt} \int_{-\infty}^\infty x g(t,x)\,dx
      &= \int_{-\infty}^\infty x
           \frac{\partial}{\partial x}
             \Bigl(D \frac{\partial g}{\partial x}(t,x)
                    + (W * g(t,\cdot))(x) g(t,x)\Bigr) \,dx \\
      &= - \int_{-\infty}^\infty
             \Bigl( D \frac{\partial g}{\partial x}(t,x)
                    + (W * g(t,\cdot))(x) g(t,x)\Bigr) \,dx \\
      &= - \int_{-\infty}^\infty \int_{-\infty}^\infty W(x-y) g(t,y) g(t,x) \,dy\,dx \\
      &= 0 \,,
  \end{align*}
  since $W$ is odd, compare Lemma~\ref{L:symmint}.
\end{proof}

Later, we will consider the case without diffusion (so with $D = 0$) in particular.
In this situation, compact support is preserved.

\begin{lemma} \label{L:suppbd}
  Assume that $W$ is bounded and that $D = 0$ in~\eqref{tpR}.
  Let $g \ge 0$ be a solution.
  If $g(0,\cdot)$ has support contained in~$[a,b]$, then the support of~$g(t,\cdot)$ is
  contained in~$[a-Ct,b+Ct]$ for all $t > 0$,
  where $C = \|W\|_\infty \|g(0,\cdot)\|_1$.
\end{lemma}

\begin{proof}
  Let $g_+(t,x) = g(t,x+b+Ct)$. We show that
  $\supp g_+(t,\cdot) \subset \left]-\infty, 0\right]$. This implies that
  $\supp g(t,\cdot) \subset \left]-\infty, b+Ct\right]$. The argument for the
  lower bound is similar.

  We have
  \begin{align*}
    \frac{\partial g_+}{\partial t}(t,x)
      &= \partial_t g(t, x+b+Ct) + C \partial_x g(t, x+b+Ct) \\
      &= \partial_x \bigl((W * g(t,\cdot) + C) \cdot g(t, \cdot)\bigr) (x+b+Ct) \\
      &= \partial_x \bigl((W + \|W\|_\infty)*g(t,\cdot) \cdot g(t, \cdot)\bigr)
            (x+b+Ct) \\
      &= \partial_x \bigl((W + \|W\|_\infty)*g_+(t,\cdot) \cdot g_+(t, \cdot)\bigr)(x)
  \end{align*}
  Now
  \begin{align*}
    \frac{d}{dt} \int_{0}^\infty g_+(t,x)\,dx
      &= \int_{0}^\infty
           \partial_x \bigl((W + \|W\|_\infty)*g_+(t,\cdot) \cdot g_+(t, \cdot)\bigr)
            (x)\,dx \\
      &= -\bigl((W + \|W\|_\infty)*g_+(t,\cdot) \cdot g_+(t, \cdot)\bigr)(0) \\
      &= -\int_{-\infty}^\infty
           (W(-y) + \|W\|_\infty) g_+(t,y)\,dy \cdot g_+(t,0) \\
      &\le 0 \,,
  \end{align*}
  since $g_+ \ge 0$ and $W + \|W\|_\infty \ge 0$. On the other hand,
  \[ \int_{0}^\infty g_+(0,x)\,dx = 0 \quad\text{and}\quad
     \int_{0}^\infty g_+(t,x)\,dx \ge 0 \,,
  \]
  so we must have $\displaystyle\int_{0}^\infty g_+(t,x)\,dx = 0$ for all~$t$.
\end{proof}

If $W(x)$ is positive for positive~$x$, we can say more.

\begin{proposition} \label{P:peak-on-R}
  Assume that $W$ is continuously differentiable with $W'(0) > 0$, that
  $W(x) > 0$ for $x > 0$ and that $D = 0$ in~\eqref{tpR}.
  Let $g \ge 0$ be a solution such that $\supp g(0,\cdot) \subset [a,b]$.
  Then $g(t,\cdot)$ has support contained in~$[a,b]$ for all $t > 0$, and it
  converges to a delta distribution $m \delta_c$ with
  $m = \|g(0,\cdot)\|_1$ and $mc = \int x g(0,x)\,dx$, in the sense that
  \[ \lim_{t\to\infty} \int_{-\infty}^\infty h(x) g(t,x)\,dx = m h(c) \]
  for all twice continuously differentiable functions $h : \R \to \R$.
\end{proposition}

\begin{proof}
  Without loss of generality, $m = 1$ and $c = 0$.
  We first prove the statement on the support of~$g(t,\cdot)$.
  In a similar way as above in the proof of Lemma~\ref{L:suppbd}, we see that
  \[
     \frac{d}{dt} \int_b^\infty g(t,x)\,dx
       = \int_b^\infty
           \partial_x \bigl((W*g(t,\cdot)) \cdot g(t,\cdot)\bigr)(x)\,dx \\
       = -(W*f(t,\cdot))(b) g(t,b)
  \]
  So if $g(t,x) > 0$ for some $t > 0$ and $x > b$, we must have
  $g(\tau, b) > 0$ and $(W*g(\tau,\cdot))(b) < 0$ for some $0 < \tau < t$.
  Let $t_0$ be the infimum of $\tau > 0$ such that $g(\tau, b) > 0$.
  Then $g(t_0,x) = 0$ for $x \ge b$, and it follows that $(W*g(t_0,\cdot))(b) > 0$.
  By continuity, we will have $(W*g(\tau,\cdot))(b) > 0$ and $g(\tau, b) > 0$
  for all sufficiently small $\tau > t_0$, so that the derivative above cannot
  be positive. So $g(t,x) > 0$ for some $t > 0$ and $x > b$ is not possible.
  This shows that $\supp g(t,\cdot) \subset \left]-\infty,b\right]$ for all~$t > 0$.
  The argument for the lower bound is similar.
  
  We now consider the second moment of~$g(t, \cdot)$. Let
  $M(t) = \int_{-\infty}^\infty x^2 g(t,x)\,dx$.
  There is $\mu > 0$ such that $xW(x) \ge \mu x^2$ for all $|x| \le b-a$. Then
  \begin{align*}
    \frac{d}{dt} M(t)
      &= \int_{-\infty}^\infty
           x^2 \partial_x \bigl((W*g(t,\cdot)) \cdot g(t, \cdot)\bigr)(x)\,dx \\
      &= -2 \int_{-\infty}^\infty x\, (W*g(t,\cdot))(x) g(t,x)\,dx \\
      &= -2 \int_{-\infty}^\infty \int_{-\infty}^\infty
              x W(x-y) g(t,y) g(t,x) \,dy\,dx \\
      &= -\int_{-\infty}^\infty \int_{-\infty}^\infty
              (x-y) W(x-y) g(t,y) g(t,x) \,dy\,dx \\
      &\le -\mu \int_a^b \int_a^b (x-y)^2 g(t,y) g(t,x) \,dy\,dx \\
      &= -2 \mu \int_a^b x^2 g(t,x)\,dx = -2 \mu M(t) \,.
  \end{align*}
  This implies that $M(t) \le e^{-2\mu} M(0)$; in particular, $M(t) \to 0$
  as $t \to \infty$.

  Now let $h \in \C^2(\R)$. We can write
  $h(x) = h(0) + h'(0) x + h_2(x)$ where $h_2(0) = h_2'(0) = 0$. This implies
  that there is some $C > 0$ such that $|h_2(x)| \le C x^2$ for $x \in [a,b]$.
  We then have
  \begin{align*}
    \Bigl|\int_{-\infty}^\infty h(x) g(t,x)\,dx - h(0)\Bigr|
      &= \Bigl|\int_{-\infty}^\infty (h(0) + h'(0) x + h_2(x)) g(t,x)\,dx - h(0)\Bigr|
         \\
      &= \Bigl|\int_a^b h_2(x) g(t,x)\,dx\Bigr|
      \le \int_a^b |h_2(x)| g(t,x)\,dx
      \le C M(t) \,,
  \end{align*}
  and this tends to zero as $t \to \infty$.
\end{proof}

%================================================================================

\subsection{Local masses and barycenters} \label{S:bary} \strut

For equation~\eqref{tp}, the mass $\int_{S^1} f(t,\theta)\,d\theta$ is still
an invariant, but there is no reasonable definition of a `first moment'.
(For this, one would need a function $F : S^1 \to \R$ that satisfies
$F(\theta + a) = F(\theta) + a$ for all $\theta \in S^1$ and $a \in \R$.
Such a function obviously does not exist.) However, we can define a localized
version of a first moment.

\begin{definition} \label{D:moment}
  Let $f : S^1 \to \R$ be continuous and nonnegative,
  and let $I \subset S^1$ be a closed interval.
  Let $I' = [a,b] \subset \R$ be an interval such that $p(I') = I$.
  We define the \emph{local mass}
  \[ m(I, f) = \int_I f(\theta)\,d\theta = \int_a^b p^*(f)(x)\,dx \in \R \]
  and, if $m(I,f) > 0$, the \emph{local barycenter}
  \[ M(I, f) = p\Bigl(\frac{1}{m(I, f)} \int_a^b x p^*(f)(x)\,dx\Bigr) \in I \,. \]
\end{definition}

The local barycenter does not depend on the choice of~$I'$ --- any other choice
has the form $I' + k$ with $k \in \Z$, and then we find that
\[ \int_{a+k}^{b+k} x p^*(f)(x)\,dx
     = \int_a^b (x+k) p^*(f)(x)\,dx
     = \int_a^b x p^*(f)(x)\,dx + k m(I, f) \,,
\]
so that the expression under $p(\cdot)$ changes by an integer, and the result
is unchanged.

\begin{lemma} \label{L:local}
  Let $f \ge 0$ be a solution of equation~\eqref{tp}, and let $I \subset S^1$ be a
  closed interval. Then the local mass $m(I, f(t,\cdot))$ is time-invariant,
  provided there is no flow across the
  boundary of~$I$: if $\alpha, \beta \in S^1$ are the endpoints of~$I$, then
  we require
  \[ (V*f(t,\cdot))(\alpha) f(t,\alpha) = (V*f(t,\cdot))(\beta) f(t,\beta) = 0 \]
  for all $t > 0$.

  If in addition, there is no interaction with parts of~$f$ outside of~$I$,
  meaning that
  \[ V(\theta-\psi) f(t,\theta) f(t,\psi) = 0
       \qquad\text{if $\theta \in I$ and $\psi \notin I$,}
  \]
  then the local barycenter $M(I, f(t,\cdot))$ is also time-invariant.
\end{lemma}

\begin{proof}
  We have
  \begin{align*}
    \frac{d}{dt} m(I, f(t,\cdot))
      &= \int_I \partial_\theta\bigl((V*f(t,\cdot)) f(t,\cdot)\bigr)(\theta)\,d\theta
         \\
      &= (V*f(t,\cdot))(\beta) f(t,\beta) - (V*f(t,\cdot))(\alpha) f(t,\alpha) = 0
  \end{align*}
  and, writing $f$ for~$\tilde{f}$ and $m(I, f) = m(I, f(t, \cdot))$ for simplicity,
  \begin{align*}
    m(I, f) &\frac{d}{dt} M(I, f(t,\cdot)) \\
      &= \int_a^b x \partial_x\bigl((V*f(t,\cdot)) f(t,\cdot)\bigr)(x)\,dx \\
      &= b (V*f(t,\cdot))(b) f(t,b) - a (V*f(t,\cdot))(a) f(t,a)
          - \int_a^b (V*f(t,\cdot))(x) f(t,x)\,dx \\
      &= -\int_I \int_{S^1} V(\theta-\psi) f(t,\psi) f(t,\theta)\,d\psi\,d\theta \\
      &= -\int_I \int_I V(\theta-\psi) f(t,\psi) f(t,\theta)\,d\psi\,d\theta \\
      &= 0 \,,
  \end{align*}
  because $V$ is odd, compare Lemma~\ref{L:symmint}.
\end{proof}

%================================================================================

\subsection{Invariance of support} \strut

We consider equation~\eqref{tp} without diffusion on the circle
(but what we say here is also valid for the equation on the real line).

\begin{lemma} \label{L:invsupp}
  Let $A = I_1 \cup \ldots \cup I_n \subset S^1$ be
  a disjoint union of closed intervals in~$S^1$; write $I_j = [\alpha_j,\beta_j]$.
  Assume that for every continuous function $h : S^1 \to \R_+$, we have the
  implication
  \[ h|_A = 0 \;\Longrightarrow\;
     (V*h)(\alpha_j) > 0 \quad\text{and}\quad (V*h)(\beta_j) < 0
     \quad\text{for all $1 \le j \le n$.}
  \]
  Let $f : \left[0,\infty\right[ \times S^1 \to \R_+$ be a solution of
  equation~\eqref{tp} with $D = 0$ such that $f(0,\cdot)|_A = 0$.
  Then $f(t,\cdot)|_A = 0$ for all $t \ge 0$.

  Equivalently, if $\supp f(0,\cdot) \subset \overline{S^1 \setminus A}$,
  then $\supp f(t,\cdot) \subset \overline{S^1 \setminus A}$ for all~$t \ge 0$.
\end{lemma}

\begin{proof}
  For the given solution~$f$, let $\Phi : \left[0,\infty\right[ \times S^1 \to S^1$
  be the flow associated to $-(V*f)$,
  \[ \Phi(0,\theta) = \theta \quad\text{and}\quad
     \frac{\partial\Phi}{\partial t}(t,\theta)
       = -(V * f(t,\cdot))(\Phi(t,\theta)) \,.
  \]
  Then it is readily checked that for all $\alpha, \beta \in S^1$, the integral
  \[ M(\alpha,\beta) = \int_{\Phi(t,\alpha)}^{\Phi(t,\beta)} f(t,\theta)\,d\theta \]
  is independent of~$t$.
%   (Proof:
%   \begin{align*}
%     \frac{d}{dt} \int_{\Phi(t,\alpha)}^{\Phi(t,\beta)} f(t,\theta)\,d\theta
%       &= f(t,\Phi(t,\beta)) \frac{\partial\Phi}{\partial t}(t, \beta)
%           - f(t,\Phi(t,\alpha)) \frac{\partial\Phi}{\partial t}(t, \alpha) \\
%       &\qquad {} + \int_{\Phi(t,\alpha)}^{\Phi(t,\beta)}
%               \frac{\partial}{\partial \theta}
%                 \bigl((V*f(t,\cdot)) f(t,\cdot)\bigr)(\theta)\,d\theta \\
%       &= f(t,\Phi(t,\beta)) \frac{\partial\Phi}{\partial t}(t, \beta)
%           - f(t,\Phi(t,\alpha)) \frac{\partial\Phi}{\partial t}(t, \alpha) \\
%       &\qquad {} + (V*f(t,\cdot))(\Phi(t,\beta)) f(t,\Phi(t,\beta))
%           - (V*f(t,\cdot))(\Phi(t,\alpha)) f(t,\Phi(t,\alpha)) \\
%       &= 0
%   \end{align*}
%   according to the definition of~$\Phi$.)
%
  Write $\alpha_j(t) = \Phi(t,\alpha_j)$, $\beta_j(t) = \Phi(t,\beta_j)$ and
  \[ A(t) = \bigcup_{j=1}^n [\alpha_j(t), \beta_j(t)] \,, \]
  then we have
  \[ \int_{A(t)} f(t,\theta)\,d\theta = 0 \]
  for all $t \ge 0$. Now assume that $f(t,\cdot)|_A$ is not identically zero
  for some $t > 0$. Then we must have that $A \not\subset A(t)$.
  Let $t_0$ be the infimum of all $t > 0$ such that $A \not\subset A(t)$.
  Then for some $1 \le j \le n$, we must have $\alpha_j(t_0) = \alpha_j$
  and $\frac{d\alpha_j}{dt}(t_0) \ge 0$, or $\beta_j(t_0) = \beta_j$ and
  $\frac{d\beta_j}{dt}(t_0) \le 0$. But in the first case
  \[ \frac{d\alpha_j}{dt}(t_0)
       = \frac{\partial\Phi}{\partial t}(t_0, \alpha_j)
       = -(V*f(t_0,\cdot))(\Phi(t_0,\alpha_j))
       = -(V*f(t_0,\cdot))(\alpha_j)
       < 0 \,,
  \]
  since $f(t_0,\cdot)|_A = 0$, and similarly in the second case
  $\frac{d\beta_j}{dt}(t_0) > 0$, leading to a contradiction.
\end{proof}

%%%%%%%%%%%%%%%%%%%%%%%%%%%%%%%%%%%%%%%%%%%%%%%%%%%%%%%%%%%%%%%%%%%%%%%%%%%%

\section{Stability of the constant solution}

\begin{theorem}[Local stability of the constant solution] \label{thm:stability}
  The constant function \( f(\theta) = 1 \) is a stationary solution
  of equation~\eqref{tp}. The eigenvalues of the linearization around~$f$ are
  \[ c_k = 4 \pi k (-\pi D k + v_k) \quad \text{for \( k \in \Z \),} \]
  where \( v_k = \int_0^{\half} V(\theta) \sin(2\pi k \theta) \,d\theta \).
  Hence, the constant stationary solution is locally stable if \( c_k < 0 \)
  for all \( k > 0 \).

  More precisely, assume that $\sum_{l>0} l |v_l| < \infty$
  (this is for example the case when $V$ is~$\C^1$ and piecewise~$\C^2$); then
  \( \rho = \sqrt{\sup_{l > 0} \frac{l^2 (l^2-1)}{6} v_l^2} < \infty \).
  Define
  \[ \|f\|^2 = \sum_{k \ge 1} \frac{1}{k} |f_k|^2 . \]
  Any initial function $f(0,\cdot)$ such that
  \[ \|f(0,\cdot)\| < \frac{-\max_{k \ge 1} k (v_k - k \pi D)}{\rho}
                    = \frac{\min_{k \ge 1} (-c_k)}{4 \pi \rho} \]
  converges to the constant~$1$ in the sense that
  \[ \|f(t,\cdot) - 1\|_\infty \to 0 \quad\text{and}\quad
     \|\partial^n_\theta f(t,\cdot)\|_\infty \to 0 \quad\text{for all $n \ge 1$}
  \]
  as $t \to \infty$.
\end{theorem}

\begin{remarks} \strut
  \vspace{-2ex}
  \begin{enumerate}
  \renewcommand{\theenumi}{\roman{enumi}}
  \renewcommand{\labelenumi}{\theenumi)}
    \item The diffusion term has a stabilizing effect on the constant stationary
          solution.
    \item Since by assumption, $l^2 v_l$ is bounded, $c_k$ will be negative
          for $k \gg 0$. Therefore, higher modes tend to be linearly stable.
    \item Because periodicity is preserved, instability of the \( k \)-th mode,
          i.e.,~\( c_k > 0 \), implies that something interesting happens in the
          subspace of \( \frac{1}{k} \)-periodic solutions. In all examples known
          to us there exist non-constant \( \frac{1}{k} \)-periodic stationary
          solutions.
	  However, we cannot exclude the possibility of {\em time-periodic}
	  \( \frac{1}{k} \)-periodic solutions or chaos. But see
	  Corollary~\ref{CorTimePeriodic} below, which shows that time-periodic
	  solutions are impossible when there is no diffusion.
    \item Since
	  \[ \|f\|^2 = \sum_{k \ge 1} \frac{1}{k} |f_k|^2
                     \le \frac{1}{2} \sum_{k \in \Z \setminus \{0\}} |f_k|^2
                     = \frac{1}{2} \|f - 1\|_2^2
                     \le \frac{1}{2} \|f - 1\|_{\infty}^2 , \]
	  the condition on \( f(0,\cdot) \) is satisfied when
	  \[ \|f(0,\cdot) - 1\|_\infty
	        < \frac{\min_{k \ge 1} (-c_k)}{2 \sqrt{2} \pi \rho} . \]
  \end{enumerate}
\end{remarks}

\begin{proof}
  Formulas \eqref{fou_k} and~\eqref{ev} show that the given \( c_k \) are the
  eigenvalues of the linearization (which is obtained by disregarding quadratic
  terms on the right hand side of the differential equation).

  Recall that we assume \( f_0 = 1 \).
  We scale time by a factor \( 4 \pi \) and set \( \delta = \pi D \) in the Fourier
  transformed system~\eqref{fou_k} to get
  \begin{equation} \label{rescaled}
    \dot{f}_k = k \Bigl((-\delta k + v_k) f_k
                 + \sum_{l \in \Z \setminus \{0,k\}} v_l f_l f_{k-l}\Bigr) .
  \end{equation}
  Because $f_{-k} = \bar{f}_k$ ($f$ is real) and $v_{-k} = -v_k \in \R$, we see that
  \begin{align*}
    \frac{1}{2}\,\frac{d}{dt} \sum_{k \ge 1} \frac{1}{k} |f_k|^2
      &= \sum_{k \ge 1} \frac{1}{k} \Re(\bar{f}_k \dot{f}_k) \\
      &= \sum_{k \ge 1} (v_k - k \delta) |f_k|^2
	  + \Re \Bigl( \sum_{k \ge 1} \sum_{l>k} v_l f_l \bar{f}_k \bar{f}_{l-k} \Bigr)
	\\
      & \qquad\quad {}
	  + \Re \Bigl( \sum_{k \ge 1} \sum_{1 \le l < k} v_l f_l \bar{f}_k f_{k-l}
			- \sum_{k \ge 1} \sum_{l \ge 1} v_l \bar{f}_l \bar{f}_k f_{k+l}
		\Bigr) \\
      &= \sum_{k \ge 1} (v_k - k \delta) |f_k|^2
	  + \Re \Bigl( \sum_{k,m \ge 1} v_{k+m} f_{k+m} \bar{f}_k \bar{f}_m \Bigr)
  \end{align*}
  Note that, setting $k \leftarrow k+l$, we have
  \[ \Re \Bigl( \sum_{k \ge 1} \sum_{1 \le l < k} v_l f_l \bar{f}_k f_{k-l} \Bigr)
      = \Re \Bigl( \sum_{k,l \ge 1} v_l f_l f_k \bar{f}_{k+l} \Bigr)
      = \Re \Bigl( \sum_{k,l \ge 1} v_l \bar{f}_l \bar{f}_k f_{k+l} \Bigr),
  \]
  justifying the last equality above. In the remaining sum, we have set
  $l \leftarrow k+m$. We can estimate the latter as follows.
  \begin{align}
  \Bigl|\Re \Bigl( & \sum_{k,m \ge 1} v_{k+m} f_{k+m} \bar{f}_k \bar{f}_m \Bigr)\Bigr|^2
    \nonumber \\
    &\le \Bigl| \sum_{k,m \ge 1} v_{k+m} f_{k+m} \bar{f}_k \bar{f}_m \Bigr|^2
      \nonumber \\
    &= \Bigl| \sum_{k,m \ge 1} \sqrt{km(k+m)} v_{k+m}
                    \frac{f_{k+m}}{\sqrt{k+m}}\,\frac{\bar{f}_k}{\sqrt{k}}\,
                    \frac{\bar{f}_m}{\sqrt{m}} \Bigr|^2 \nonumber \\
    &\le \Bigl( \sum_{k \ge 1} \frac{1}{k} |f_k|^2 \Bigr)
         \Bigl( \sum_{k \ge 1}
                 \Bigl| \sum_{m \ge 1}
                           \sqrt{km(k+m)} v_{k+m}
                              \frac{f_{k+m}}{\sqrt{k+m}}\,\frac{\bar{f}_m}{\sqrt{m}}
                 \Bigr|^2 \Bigr) \nonumber \\
    &\le \Bigl( \sum_{k \ge 1} \frac{1}{k} |f_k|^2 \Bigr)
         \Bigl( \sum_{k \ge 1}
                 \Bigl( \sum_{m \ge 1} \frac{1}{m} |f_m|^2 \Bigr)
                 \Bigl( \sum_{m \ge 1} km(k+m) v_{k+m}^2 \frac{1}{k+m} |f_{k+m}|^2 
                 \Bigr)
         \Bigr) \nonumber \\
    &= \Bigl( \sum_{k \ge 1} \frac{1}{k} |f_k|^2 \Bigr)^2
       \Bigl( \sum_{l \ge 1} \Bigl( \sum_{k = 0}^l k(l-k) \Bigr) l v_l^2
              \, \frac{1}{l} |f_l|^2 \Bigr) \nonumber \\
    &\le \Bigl( \sum_{k \ge 1} \frac{1}{k} |f_k|^2 \Bigr)^2
         \Bigl( \sum_{l \ge 1} \frac{l^2 (l^2-1)}{6} v_l^2
              \, \frac{1}{l} |f_l|^2 \Bigr) \label{E:bound} \\
    &\le  \sup_{l > 0} \left( \frac{l^2 (l^2-1)}{6} v_l^2 \right) 
           \, \Bigl( \sum_{k \ge 1} \frac{1}{k} |f_k|^2 \Bigr)^3
      = \rho^2 \Bigl( \sum_{k \ge 1} \frac{1}{k} |f_k|^2 \Bigr)^3 . \nonumber
  \end{align}
  In terms of \( \|f\|^2 = \sum_{k \ge 1} \frac{1}{k} |f_k|^2 \), this means
  \begin{equation} \label{diffineq1}
    \frac{1}{2}\,\frac{d}{dt} \|f\|^2
      \le \bigl(\max_{k \ge 1} k(v_k - k \delta) + \rho \|f\|\bigr) \|f\|^2 .
  \end{equation}
  Let $f(0,\cdot)$ satisfy the given condition and set
  \[ c = -\bigl(\max_{k \ge 1} k(v_k - k \delta) + \rho \|f(0,\cdot)\|\bigr) > 0 . \]
  It follows that
  \[ \|f(t)\|^2 \le \|f(0,\cdot)\|^2 e^{-2 c t} \]
  for $t \ge 0$. Since $\frac{1}{k} |f_k|^2 \le \|f\|^2$, this implies that
  \[ %\begin{equation} \label{ineq-fk}
    |f_k(t)| \le \sqrt{k} \|f(0,\cdot)\| e^{-c t}
       \quad\text{for $k \ge 1$ and $t \ge 0$.} 
  \] %\end{equation}

  We need a lemma.
  \begin{lemma} \label{LemmaThm1}
    Assume that $|f_k(t)| \le C k^\alpha e^{-c t}$ for all $k \ge 1$ and all $t \ge 0$,
    where $C > 0$ and $\alpha \le \half$ are constants. Then for any $t_0 > 0$,
    there is a constant $C' > 0$ (depending on~$t_0$) such that
    $|f_k(t)| \le C' k^{\alpha-2} e^{-c t}$ for all $k \ge 1$ and all $t \ge t_0$.
  \end{lemma}

  \begin{proof}[Proof of Lemma~\ref{LemmaThm1}]
    The quadratic part of the right hand side of the differential 
    equation~\eqref{rescaled} for~$f_k$ is
    \begin{align*}
      R_k &= \sum_{0 < l < k} v_l f_l f_{k-l}
	       + \sum_{l > k} v_l f_l \bar{f}_{l-k}
	       - \sum_{l > 0} v_l \bar{f}_l f_{k+l}
    \end{align*}
    We estimate $R_k$:
    \begin{align*}
      |R_k(t)|
        &\le C^2 e^{-2ct} \Bigl(\sum_{l > 0, l \neq k} |v_l|\,l^\alpha\,|l-k|^\alpha
                              + \sum_{l > 0} |v_l|\,l^\alpha\,|l+k|^\alpha\Bigr) \\
        &\le C^2 k^\alpha e^{-2ct} \sum_{l>0} C_1 l |v_l|
         \le C_2 k^\alpha e^{-2ct} \le C_2 k^\alpha e^{-ct} .
    \end{align*}
    Here we use that $\sum_{l > 0} l |v_l| < \infty$ and that
    $l^\alpha |l \pm k|^\alpha \le C_1 k^\alpha l$ for some constant~$C_1$
    only depending on~$\alpha$.
    Write $c'_k = k(v_k - k\delta) = c_k/(4\pi) < 0$. Then we have
    \[ \dot{f}_k - c'_k f_k = R_k \]
    and therefore
    \[ f_k(t) = e^{c'_k t} f_k(0) + \int_0^t e^{c'_k (t-\tau)} R_k(\tau)\,d\tau . \]
    The integral is bounded by
    \[ C_2 k^\alpha \Bigl|\frac{e^{-ct} - e^{c'_k t}}{c + c'_k}\Bigr|
        \le C_3 k^{\alpha-2} e^{-ct}
    \]
    for some constant~$C_3$ (note that $-c'_k \gg k^2$ and that
    $|c + c'_k| = |c'_k| - c > \rho \|f(0,\cdot)\|^2 > 0$).
    This gives
    \[ |f_k(t)| \le \left(k^{2-\alpha} e^{(c'_k+c)t} |f_k(0)| + C_3\right)
                       k^{\alpha-2} e^{-ct} .
    \]
    Since $c'_k + c \le -\operatorname{const.} k^2$, the first summand in
    brackets is bounded uniformly in~$k > 0$ for $t \ge t_0 > 0$.
    This finishes the proof of the lemma.
  \end{proof}

  Repeated application of the lemma then shows that, given $N > 0$ and
  $t_0 > 0$, there is a constant $C_N > 0$ such that
  \[ |f_k(t)| \le C_N k^{-N} e^{-ct}
       \quad \text{for all $k \ge 1$ and all $t \ge t_0$.}
  \]
  This implies
  \[ \|f(\cdot, t) - 1\|_\infty
	\le 2 \sum_{k \ge 1} |f_k(t)| = O\bigl(e^{-c t}\bigr) .
  \]
  Similarly, for any $n \ge 1$, we obtain
  \[ \|\partial_\theta^n f(\cdot, t)\|_\infty
       \le 2 (2\pi)^n \sum_{k \ge 1} k^n |f_k(t)|  = O\bigl(e^{-c t}\bigr) . \qedhere
  \]
\end{proof}

\begin{corollary}[Conditions for pattern formation] \label{cor_stab} \strut
  \renewcommand{\theenumi}{\emph{\alph{enumi}}}
  \renewcommand{\labelenumi}{\theenumi)}
  \begin{enumerate}
    \item Instability of the first mode: 
	  If \( V \) is positive on \( \left]0,\half\right[ \), 
          i.e.~all filaments attract
	  each other, then the first mode is unstable for sufficiently small
	  diffusion coefficient \( D \): \( c_1 > 0 \) for \( D \ll 1 \).
    \item Instability of the second mode: 
	  Assume that there exists \( \theta_0 \in \left]0,\half\right[ \) such that 
	  \( V(\theta) > 0 \) on \( \left]0,\theta_0\right[ \), \( V(\theta_0) = 0 \)
	  and
	  \( V(\theta) < 0 \) on \( \left]\theta_0,\half\right[ \).
	  Moreover let (*)
	  \( V(\theta) \ge V(\half-\theta) \) 
	  for \( \theta \in \left]\min(\theta_0,\half-\theta_0),\frac{1}{4}\right[ \).
	  Then the second mode is unstable for sufficiently small diffusion coefficient
	  \( D \), i.e.~\( c_2 > 0 \) for \( D \ll 1 \).
  \end{enumerate}
\end{corollary}

\begin{proof}
  Statement~a) follows from $v_1 > 0$ and $c_1 = 4\pi(v_1 - \pi D)$;
  statement~b) follows from Proposition~\ref{1/n-periodic} and part~a)
  because \( \tilde{V}_2 > 0 \) on \( \left]0,\half\right[ \).
\end{proof}

The second condition~(*) for instability of the second mode can be interpreted
in the following way.
If there are already two peaks forming then attraction towards the nearer peak
must be stronger than towards the second peak.
Using exactly the assumptions of~b), it has been shown by Primi et~al.~\cite{Primi}
that equation~\eqref{tp} has a \( \half \)-periodic stationary solution with
two equally large and very high maxima if the diffusion coefficient~$D$ is small enough.

\begin{theorem}[Global stability of the constant solution]
\label{Thm:globalstab}
  Assume that 
  \[ \rho' = \sum_{k > 0} \frac{k (k^2 - 1)}{6} v_k^2 < \infty \]
  (this is the case when $V$ is twice continuously differentiable, for
  example). If
  \[ v_k < \pi D \, k  - \frac{\rho'}{k} \qquad\text{for all $k \ge 1$,} \]
  then every nonnegative initial function $f(0,\cdot) \in \C(S^1)$ of mass~$1$
  converges to the constant function~$1$ --- there is some $c > 0$ such that
  \[ \bigl\|\partial_\theta^n (f(t,\cdot) - 1)\bigr\|_\infty
       = O_n\bigl(e^{-ct}\bigr)
  \]
  for all $n \ge 0$.
  
  For given~$V\!$, the assumption is satisfied whenever
  \[ D > \max_{k \ge 1} \frac{\rho' + k v_k}{k^2 \pi} , \]
  i.e., if diffusion is strong enough.
\end{theorem}

\begin{proof}
  Inequality~\eqref{E:bound} implies
  \begin{equation} \label{diffineq2}
    \frac{1}{2}\,\frac{d}{dt} \|f\|^2
      \le \Bigl(\sum_{l \ge 1} \frac{l^2(l^2-1)}{6} v_l^2 \frac{1}{l} |f_l|^2
                 + \max_{k \ge 1} k(v_k - k \delta)\Bigr) \|f\|^2
      \le \bigl(\rho' + \max_{k \ge 1} k(v_k - k \delta)\Bigr) \|f\|^2 .
  \end{equation}
  (with \( \|f\|^2 = \sum_{k \ge 1} \tfrac{1}{k} |f_k|^2 \) as before
  and using $|f_l| \le 1$).
  Now the proof proceeds as for Theorem~\ref{thm:stability},
  but using~\eqref{diffineq2} instead of~\eqref{diffineq1}.
\end{proof}

%%%%%%%%%%%%%%%%%%%%%%%%%%%%%%%%%%%%%%%%%%%%%%%%%%%%%%%%%%%%%%%%%%%%%%%%%%

\section{Convergence to peak solutions in case of small initial support
         and no diffusion}
\label{compact-support}

If there is {\em no random turning}, i.e., $D = 0$ in~\eqref{tp},
then sums of delta peaks can be stationary solutions. To make this precise,
we have to define the right hand side of equation~\eqref{tp} for suitable
distributions on the circle.

The kind of distribution we are mostly interested
in are (positive) \emph{measures}, but it turns out that it is advantageous
to use \emph{differentiable measures} instead. The main reason for this is
that the map $S^1 \to \D^0(S^1)$, $\psi \mapsto \delta_\psi$, is continuous,
but not differentiable (since the derivative at zero would have to be $-\delta'_0$).
As a map to $\D^1(S^1)$, it becomes differentiable, though.

Let $k \ge 0$.
The space $\D^k(S^1)$ is the dual space of the space $\C^k(S^1)$ of $k$~times
continuously differentiable functions on the circle~$S^1$.
The elements of~$\D^0(S^1)$ are called \emph{measures}, and the elements
of~$\D^1(S^1)$ are called \emph{differentiable measures} on~$S^1$.
By standard theory (see for example~\cite{Jantscher}), $\D^k(S^1)$ can be
identified with the subspace of $\D(S^1)$ (which is the dual of~$\C^\infty(S^1)$)
consisting of distributions of order $\le k$ --- $f \in D^k(S^1)$ if and only if
\[ \text{there is $C > 0$ such that} \qquad
    \bigl|\langle f, h \rangle\bigr| \le C \sum_{j=0}^k \|h^{(j)}\|_\infty
    \quad \text{for all $h \in \C^\infty(S^1)$.}
\]
In particular, we can consider $\D^0(S^1)$ as a subspace of~$\D^1(S^1)$.
$\D(S^1)$ can be
identified with the space of 1-periodic distributions on~$\R$, compare~\cite{Jantscher}.

A distribution \( f \in \D^k(S^1) \) is \emph{non-negative} if
\( \langle f, h \rangle \ge 0 \) for every non-negative function
\( h \in \C^{\infty}(S^1) \). We write
$\D^k_+(S^1)$ for the set of non-negative distributions in~$\D^k(S^1)$.
Note that for $f \in \D^k_+(S^1)$ and test functions $h_1 \le h_2$ we have
$\langle f, h_1 \rangle \le \langle f, h_2 \rangle$.

The \emph{support} \( \supp f \) of \( f \in \D^k(S^1) \) is the smallest closed
subset of \( S^1 \) outside of which \( f = 0 \).

Let \( \psi \in S^1 \).
The \emph{delta distribution} \( \delta_\psi \in \D^0_+(S^1) \) is defined by
\( \langle \delta_\psi, h \rangle = h(\psi) \)
where \( h \) is a test function.

The \emph{mass} of a distribution \( f \in \D(S^1) \) is defined as
\( \int_{S^1} f(\theta) \, d\theta = \langle f,1 \rangle \).

The \emph{convolution} of a distribution \( f \in \D^k(S^1) \) with a function
\( V \in \C^{k}(S^1) \) is defined as
\[ (V*f)(\theta) = \langle f, V(\cdot-\psi) \rangle \]
for \( \theta \in S^1 \).
It is known that \( V*f \in \C^k(S^1) \) (see \cite{Jantscher}).
For example, \( (V*\delta_{\psi})(\theta) = V(\theta-\psi) \)
for \( \theta, \psi \in S^1 \).

Convergence \( f_n \stackrel{\D}{\to} f \) in~$\D(S^1)$ (or~$\D^k(S^1)$)
as \( n \to \infty \) means that
\( \langle f_n, h \rangle \to \langle f, h \rangle \) as \( n \to \infty \)
for all test functions \( h \in \C^{\infty}(S^1) \).

%==========================================================================

\subsection{No time-periodic solutions} \strut

We now prove a lemma that is inspired by~\cite{Mogilner2}.
It will let us deduce that there are no
time-periodic solutions when there is no diffusion. We define
\[ \Phi(\theta) = \int_0^\theta V(\psi)\,d\psi \,; \]
this makes sense as a function on~$S^1\!$, since $\int_{S^1} V(\psi)\,d\psi = 0$.
Note that $\Phi$ is an even function.

\begin{lemma}  \label{LemmaPot}
  Let $f(t,\theta) \ge 0$ be a solution of equation~\eqref{tp} with $D = 0$, and define
  \[ \Psi(t) = \int_{S^1} \int_{S^1} \Phi(\theta-\psi) f(t,\psi) f(t,\theta)
                  \,d\psi\,d\theta \,.
  \]
  Then $\frac{d\Psi}{dt} \le 0$, with equality if and only if $f$ is
  a stationary solution.
\end{lemma}

\begin{proof}
  We have, by symmetry and integration by parts,
  \begin{align*}
    \frac{d}{dt} \Psi(t)
      &= \int_{S^1} \int_{S^1} \Phi(\theta-\psi)
                         f(t,\psi) \dfdt(t,\theta) \, d\psi\,d\theta
          + \int_{S^1} \int_{S^1} \Phi(\theta-\psi)
                         \dfdt(t, \psi) f(t, \theta) \, d\psi\,d\theta \\
      &= 2 \int_{S^1} \int_{S^1} \Phi(\theta-\psi)
                         f(t,\psi) \dfdt(t,\theta) \, d\psi\,d\theta \\
      &= 2 \int_{S^1} \Bigl(\int_{S^1} \Phi(\theta-\psi) f(t, \psi) \,d\psi\Bigr)
                         \frac{\partial f}{\partial\theta}
                           \bigl((V*f(t,\cdot))(\theta) f(t, \theta)\bigr)
                       d\theta \\
      &= -2 \int_{S^1} \frac{\partial}{\partial \theta}
                          \bigl(\Phi * f(t, \cdot)\bigr)(\theta)
                       \, \bigl(V*f(t,\cdot)\bigr)(\theta) f(t, \theta)\,d\theta \\
      &= -2 \int_{S^1}
          \bigl(\bigl(V * f(t, \cdot)\bigr)(\theta)\bigr)^2 f(t,\theta)\,d\theta \\
      &\le 0 \,.
  \end{align*}
  From this computation, we also see that $\frac{d\Psi}{dt}(t_0) = 0$ only if
  $(V*f(t_0,\cdot)) f(t_0,\cdot) = 0$, which implies
  $\dfdt(t_0,\theta) = 0$ for all~$\theta \in S^1$, so $f(t_0,\cdot)$
  must be a stationary solution.
\end{proof}

In~\cite{Mogilner2}, this `potential'~$\Psi$ is introduced in the case when
$f$ is a sum of point masses. In terms of the positions of these point
masses, equation~\eqref{tp} (with $D = 0$) then turns into a gradient
system, which shows that it will evolve towards an equilibrium. If $V$
is positive on $\left]0,\half\right[$, then the potential has its global
minimum when all the masses are concentrated at the same point. So
Mogilner et~al.\ conclude that the system of point masses will converge
to this equilibrium (which corresponds to a single delta peak) if sufficiently
perturbed from its initial state. We use the continuous analogue to show
that at least no time-periodic solutions can arise.

\begin{corollary} \label{CorTimePeriodic}
  Let $f(t,\theta) \ge 0$ be a time-periodic solution of~\eqref{tp} with $D = 0$.
  Then $f$ is in fact a stationary solution.
\end{corollary}

\begin{proof}
  We have $f(T,\cdot) = f(0,\cdot)$ for some $T > 0$. This implies that
  $\Psi(T) = \Psi(0)$. Since by Lemma~\ref{LemmaPot}, $\frac{d\Psi}{dt} \le 0$,
  we must have $\frac{d\Psi}{dt}(t) = 0$ for all $0 \le t \le T$.
  By Lemma~\ref{LemmaPot} again, this implies that $f$ is stationary.
\end{proof}

This argument does not carry over to the case with diffusion.
Given the equalizing effect of diffusion, it seems rather unlikely
that time-periodic solutions appear when $D$ increases from zero.
However, we cannot rule out the possibility of some kind of diffusion-driven
instability that might lead to a periodic solution.

%==========================================================================

\subsection{Convergence to peaks} \strut

Let \( V \in \C^{1}(S^1) \) be odd, and let \( f \in \D^1_+(S^1) \).
Let \( h \in \C^\infty(S^1) \) be a test function.
Then the transport term of~\eqref{tp} is defined as
\begin{equation} \label{tptest}
  \bigl\langle \frac{\partial}{\partial\theta} \left((V*f) \,f \right), h \bigr\rangle
   = - \langle f, (V*f) h' \rangle .
\end{equation}

\begin{proposition} \label{statpeaks}
  Let \( V \in \C^{1}(S^1) \) be odd and consider equation~\eqref{tp} with $D = 0$.
  \renewcommand{\theenumi}{\roman{enumi}}
  \renewcommand{\labelenumi}{\theenumi)}
  \vspace{-2ex}
  \begin{enumerate}
    \item A single peak \( f = \delta_\psi \in \D^1_+(S^1) \) with \( \psi \in S^1 \)
          is a stationary solution of~\eqref{tp} with mass~1.
    \item Let \( V(\theta_0) = 0 \) for fixed \( 0 < \theta_0 \le \half \).
          Two peaks with arbitrary masses and distance \( \theta_0 \) are a stationary
          solution,
          i.e.~\( f = m_1 \delta_{\psi} + m_2 \delta_{\psi+\theta_0} \in \D^1_+(S^1) \)
          is a stationary solution of~\eqref{tp} for any \( \psi \in S^1 \)
          and \( m_1, m_2 > 0 \). If \( m_1 + m_2 = 1 \) then \( f \) has mass 1.
    \item \( n \ge 3 \) peaks with equal masses and equal distances
          are a stationary solution:
          For \( 1 \le j \le n \) let \( \psi_j \in S^1 \)
          with \( \psi_{j+1} - \psi_j = \frac{1}{n} \)
          (where \( \psi_{n+1} = \psi_1 \)).
          Then \( f = \frac{1}{n} \sum_{j=1}^n \delta_{\psi_j} \in \D^+(S^1) \)
          is a stationary solution of~\eqref{tp} with mass~1.
  \end{enumerate}
\end{proposition}

Note that \( n \ge 3 \) peaks with {\em different} masses are in general no stationary
solution contrary to the ``degenerate'' case \( n = 2 \) (where \( V(\half) = 0 \)
holds always).
The situation changes e.g.~for \( n = 4 \) if \( V(\frac{1}{4}) = 0 \).
Then again stationary solutions consisting of four peaks with
distance~\( \frac{1}{4} \) and possibly different masses occur.

\begin{proof}
  Let \( h \in \C^\infty(S^1) \) be a test function.

  i) Using \eqref{tptest} and $V(0) = 0$, we get
  $\langle \delta_\psi, (V*\delta_\psi) h' \rangle = V(0) h'(0) = 0$.

  ii) Let $\psi_2 = \psi + \theta_0$.
  Using \eqref{tptest} and $V(0) = V(\theta_0) = V(-\theta_0) = 0$, we get
  \begin{align*}
    \langle m_1 \delta_\psi &+ m_2 \delta_{\psi_2},
        (V*(m_1 \delta_\psi + m_2 \delta_{\psi_2})) h' \rangle \\
    &= m_1 \langle \delta_\psi, (V*(m_1 \delta_\psi + m_2 \delta_{\psi_2})) h' \rangle
        + m_2 \langle \delta_{\psi_2},
                      (V*(m_1 \delta_\psi + m_2 \delta_{\psi_2})) h' \rangle \\
    &= m_1 (m_1 V(0) + m_2 V(-\theta_0)) h'(\psi)
        + m_2 (m_1 V(\theta_0) + m_2 V(0)) h'(\psi_2)
    = 0 .
  \end{align*}

  iii) This follows immediately from statement~i) and Proposition~\ref{1/n-periodic}.
\end{proof}

In the following theorem we show that solutions converge to single peaks if the support of the initial function is sufficiently small.

\begin{theorem} \label{smallsupport}
  Let $V \in \C^1(S^1)$ be odd with $V'(0) > 0$ and $V > 0$ on
  $\left]0,\theta_v\right[$, where $0 < \theta_v \le \half$.
  Assume that \( f \ge 0 \) is a solution of~\eqref{tp} with $D = 0$ such that
  $\supp f(0,\cdot) \subset I$ where $I \subset S^1$ is a closed interval
  with $\Vol(I) < \theta_v$ and such that $m(I, f(0,\cdot)) = 1$.

  Then $\supp f(t,\cdot) \subset I$ for all $t \ge 0$, and $f(t,\cdot)$ converges
  to the delta distribution $\delta_{M}$, where $M = M(I, f(0,\cdot))$ is
  the local barycenter of~$f(0,\cdot)$ on~$I$.
\end{theorem}

\begin{proof}
  We lift $I$ to an interval $I' = [a,b] \subset \R$ and let $\ell = b-a < \theta_v$.
  Then $f(0,\cdot) = p_*(g_0)$ for a function $g_0 : \R \to \R$ with
  $\supp g_0 \subset I'$.

  We consider equation~\eqref{tpR}, where we take $W = p^*(V)$ on~$[-\ell,\ell]$ and
  extend it to all of~$\R$ in such a way that it is odd and satisfies $W(x) > 0$
  for $x > 0$ (which is possible since $p^*(V) > 0$ on $\left]0,\ell\right]$).
  Let $g$ be the solution of equation~\eqref{tpR} with $D = 0$ such that
  $g(0,\cdot) = g_0$. By Proposition~\ref{P:peak-on-R},
  $\supp g(t,\cdot) \subset I'$ for all $t \ge 0$. So the function
  $(W*g(t,\cdot)) g(t,\cdot)$ appearing
  on the right hand side of equation~\eqref{tpR} will always be equal to
  $(p^*(V)*g(t,\cdot)) g(t,\cdot)$ (since $W(x-y) = p^*(V)(x-y)$ when $x,y \in I'$).
  This means that $g$ will also be the solution of equation~\eqref{tpR},
  if we use $p^*(V)$ instead of~$W$. By Proposition~\ref{P:2eqns}, we then have
  $f(t,\cdot) = p_*(g(t,\cdot))$ for all $t \ge 0$. In particular,
  $\supp f(t,\cdot) \subset p(\supp g(t,\cdot)) \subset p(I') = I$.
  By Proposition~\ref{P:peak-on-R}, we also know that $g(t,\cdot)$ converges
  to~$\delta_{M'}$, where $M' = \int_{\R} x g(0,x)\,dx$,
  so $f(t,\cdot) = p_*(g(t,\cdot))$ will converge to~$\delta_{M}$, since $M = p(M')$.
\end{proof}

\begin{example}
  Initial growth of an already sharp peak may be also seen if \( 0 < D \ll 1 \)
  even if no single peak solution is expected, see Figure \ref{initial_growth}.
  In Figure \ref{initial_growth} the interaction function is \( \half \)-periodic,
  \( V(\theta) = \sin(4 \pi \theta) \).
  Only the second eigenvalue becomes positive for decreasing diffusion,
  and Primi's et~al.~\cite{Primi} conditions for existence of a single peak are
  not satisfied.
  Therefore, a single peak solution is not expected for~\eqref{tp} with \( D > 0 \),
  but as shown in the figure a sharp peak grows initially.
  Indeed, we did not see the development of a second peak although we had the program
  run up to times larger than~140.

  Note that by Proposition~\ref{P:perstat} any stationary solution must be
  $\half$-periodic in this example (and such stationary solutions exist and are
  expected to be stable). Therefore the behavior described here must be an
  artifact of the numerics. We think that the explanation is that the time scale
  for the transition from one peak to two peaks should be roughly of the order
  of~$e^{1/D}$, so the rate of change would be of order~$e^{-1/D}$, which is
  numerically zero if $D$ is as small as in the example.
\end{example}

\begin{figure}
  \Gr{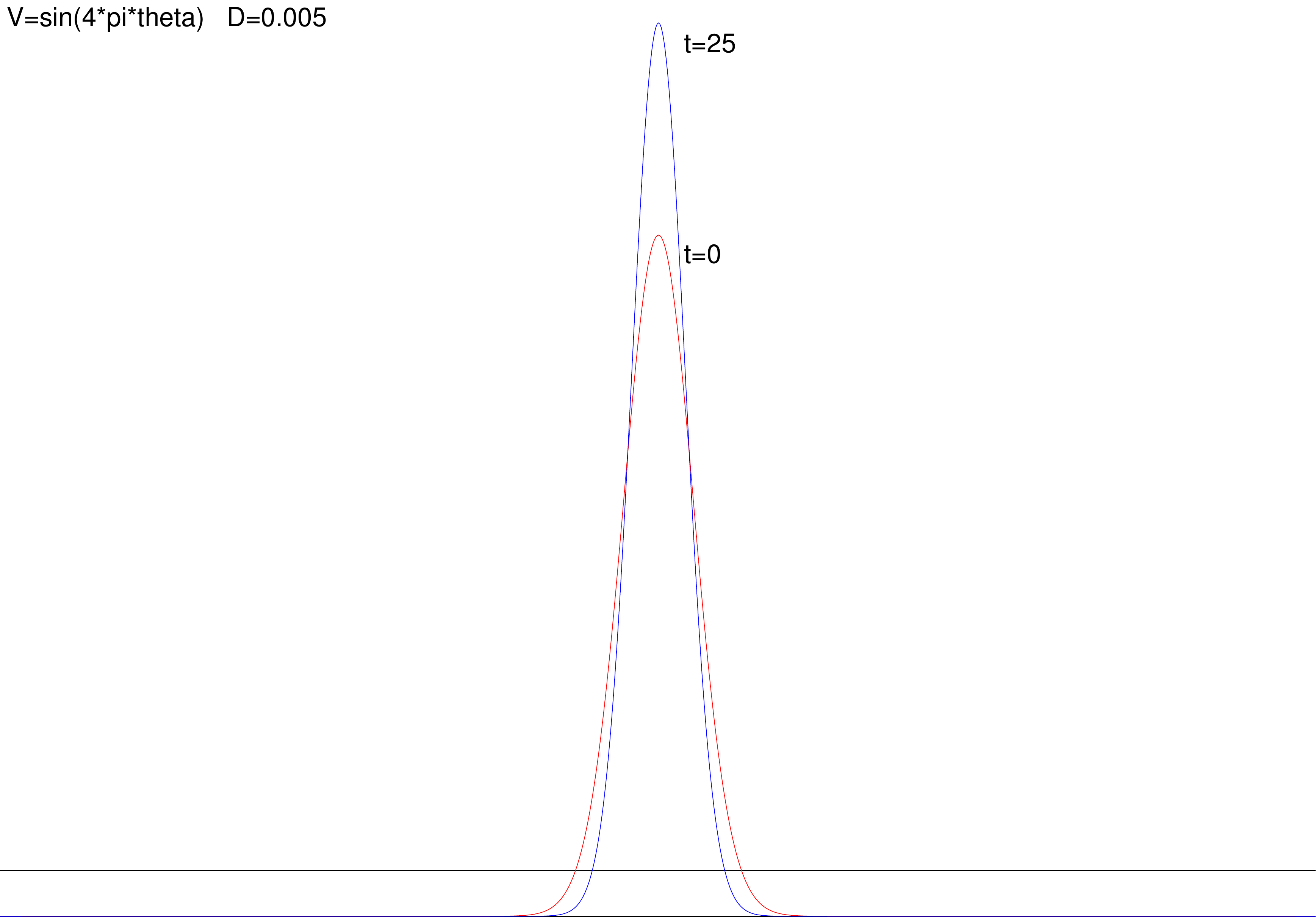}{0.5\textwidth}
  \caption{\label{initial_growth}
  \( V(\theta) = \sin(4 \pi \theta) \), \( D = 0.005 \), and 
  \( f(0,\cdot) \) is the stationary solution of~\eqref{tp}
  for \( V(\theta) = \sin(2 \pi \theta) \), \( D = 0.005 \).
  The numerical algorithm is described in Section~\ref{numerical-algorithms}
  (Fourier based method with 61 Fourier coefficients).
  }
\end{figure}

The next corollary follows directly from Theorem~\ref{smallsupport} and
Proposition~\ref{1/n-periodic}.
The assumptions on~\( V_n \) imply also that the \( n \)-th eigenvalue is positive
(Corollary~\ref{cor_stab}).

\begin{corollary}
  Let $n \ge 1$, \( 0 < \theta_v < \frac{1}{2n} \), and let \( V \in \C^{1}(S^1) \)
  be odd and such that \( V_n'(0) > 0 \) and \( V_n > 0 \) on
  \( \left]0,\theta_v\right[ \).
  Let $f \ge 0$ be a solution of equation~\eqref{tp} with $D = 0$ such that
  \( f(0,\cdot) \) is \( \frac{1}{n} \)-periodic with mass~1, and
  assume that there exists an interval \( I \subset S^1 \) with \( \Vol(I) < \theta_v \)
  such that \( \supp f(0,\cdot) \subset \bigcup_{j=0}^{n-1} (I+\frac{j}{n}) \).

  Then the solution \( f \) converges to \( n \) peaks of equal masses at equal
  distances:
  \[ f(t,\cdot) \stackrel{\D}{\to}
       \frac{1}{n} \sum_{j=0}^{n-1} \delta_{M+\tfrac{j}{n}}
  \qquad \text{as $t \to \infty$,} \qquad
    \text{where $M = M(I, f(0,\cdot))$.}
  \]
\end{corollary}

The following corollary states that several peaks at random distances may form if
\( V \) is zero in a neighborhood of \( \thalf \). There is, however, a minimal
distance between them.

\begin{corollary}
  Let \( V \in \C^1(S^1) \) be odd and \( 0 < \theta_v < \half \) such that
  \( V > 0 \) on \( \left]0,\theta_v\right[ \) and
  \( V(\theta) = 0 \) for \( \theta_v \le |\theta| \le \half \).
  Let \( n \ge 1 \) and \( f(0,\cdot) \in \C^+(S^1) \),
  \( \supp f(0,\cdot) \subset \bigcup_{j=1}^n I_j \)
  where \( \Vol(I_j) < \theta_v \) and \( \dist ( I_j, I_k ) > \theta_v \)
  for all \( 1 \le j \neq k \le n \).
  Define \( m_j(t) = m(I_j, f(t,\cdot)) \)
  and \( M_j(t) = M(I_j, f(t,\cdot)) \).

  Then \( \supp (f(t,\cdot)) \subset \bigcup_j I_j \) for all \( t \ge 0 \)
  and the masses \( m_j \) as well as the barycenters \( M_j \) are constant in~\( t \).
  The solution \( f \) converges to a sum of delta peaks:
  \[ f(t,\cdot) \stackrel{\D}{\to} \sum_{j=1}^{n} m_j \, \delta_{M_j}
      \qquad \text{as $t \to \infty$.}
  \]
\end{corollary}

\begin{proof}
  As long as the support of $f(t,\cdot)$ stays contained in the union of the~$I_j$,
  the evolution of~$f$ on each of the~$I_j$ proceeds independently, since the part
  of~$f$ contained in the other intervals does not contribute to the right hand
  side of equation~\eqref{tp}. But then Theorem~\ref{smallsupport} shows that the part
  that starts in~$I_j$ stays in~$I_j$ and converges to a delta peak as stated.
\end{proof}

In the following theorem we are interested in convergence to two peaks, but
Proposition~\ref{1/n-periodic} cannot be used since $f(0,\cdot)$ is not
necessarily $\thalf$-periodic. Neither is $V$ $\thalf$-periodic in general.

We first prove a lemma that allows us to show convergence to the
delta-distribution if mass is constant and second moments converge to zero.

\begin{lemma} \label{convergence-peak}
  Let $I \subset S^1$ be a closed interval.
  Let \( f_n \in \D^1_+(S^1) \) with \( \supp f_n \subset I \) and
  $\langle f_n, 1 \rangle = m > 0$ for all \( n \ge 1 \).
  Let $M \in I$, and let $q_M : S^1 \to \R$ be a $\C^\infty$ function
  satisfying $q_M(\theta) = (\theta-M)^2$ for all $\theta \in I$.
  If $\langle f_n, q_M \rangle \to 0$ as $n \to \infty$, then
  $f_n \stackrel{\D}{\to} m \delta_{M}$ as $n \to \infty$.
\end{lemma}

Note that the same conclusion is valid when we only assume that
$q_M(\theta) \ge c (\theta-M)^2$ for all $\theta \in I$ with some $c > 0$.

\begin{proof}
  Let $\ell_M : S^1 \to\R$ be a $\C^\infty$ function such that
  $\ell_M(\theta) = \theta - M$ for $\theta \in I$, and define
  $a_n = \langle f_n, \ell_M \rangle$ and $b_n = \langle f_n, q_M \rangle$
  Then by the Cauchy-Schwarz inequality
  (applied to the inner product $(g,h) \mapsto \langle f_n, gh \rangle$
  for functions $g, h : I \to \R$, concretely with $g = 1$ and $h = \ell_M$)
  we have $a_n^2 \le m b_n$. Since $b_n \to 0$,
  we must have $a_n \to 0$ as well.
  Let $h \in \C^{\infty}(S^1)$ be a test function. We can write
  \[ h(\theta) = h(M) + h'(M) \ell_M(\theta) + r(\theta) q_M(\theta) \]
  for $\theta \in I$, with $|r(\theta)| \le C = \frac{1}{2} \max_I |h''|$.
  We then have
  \begin{align*}
    \bigl| \langle f_n - m \delta_{M}, h \rangle \bigr|
    &= \bigl| \langle f_n, h(M) + h'(M) \ell_M + r q_M \rangle - m h(M) \bigr| \\
    &= \bigl| h'(M) \langle f_n(x), \ell_M \rangle
                +  \langle f_n(x), r q_M \rangle \bigr| \\
    &\le |h'(M)| a_n + C b_n \\
    &\to 0 \quad \text{as $n \to \infty$.} \qedhere
  \end{align*}
\end{proof}

The final positions $\bar{M}_0$ and~$\bar{M}_1$ of the two peaks in the
theorem below are obtained from the special case when $V(p(x)) = cx$
with $c > 0$ in an interval around zero and $V$ is $\half$-periodic.
In this case one gets equations $\frac{dM_j}{dt}(t) = - c M_j(t)$
(with $M_j(t)$ defined as in the proof below), so that $M_j(t) \to 0$,
justifying the choice of~$\bar{M}_j$.

\begin{theorem} \label{twopeak}
  Let \( V \in \C^{1}(S^1) \) be odd with \( V'(0) > 0 \) and \( V'(\thalf) > 0 \),
  and assume that there exist \( 0 < \theta_1 \le \theta_2 < \half \)
  such that \( V > 0 \) on \( \left]0,\theta_1\right[ \)
  and \( V < 0 \) on \( \left]\theta_2,\half\right[ \).

  Let $f \ge 0$ be a solution of equation~\eqref{tp} with $D = 0$ such that
  \( f(0,\cdot) \in \C^+(S^1) \) with \( \int_{S^1} f(0,\theta) \,d\theta = 1 \) and
  \( \supp f(0,\cdot) \subset I_0 \cup I_1 \)
  where \( I_0 \subset S^1 \) is a closed interval such that
  \( \Vol(I_0) < \min\{\theta_1,\thalf-\theta_2\} \)
  and \( I_1 = I_0 + \half \).
  Then \( \supp f(t,\cdot) \subset I_0 \cup I_1 \) for all \( t \ge 0 \).

  Let $I$ be a closed interval in~$S^1$ containing $I_0 \cup I_1$.
  Define the local masses \( m_j(t) = m(I_j,f(t,\cdot)) \),
  and let $M(t) = M(I, f(t,\cdot))$ be the local barycenter on~$I$.
  Then $m_0(t)$, $m_1(t)$ and~$M(t)$ are constant in time; we write $m_0$, $m_1$
  and~$M$ for their values. Define
  \[ \bar{M}_0 = M + \thalf m_1 \qquad \text{and} \qquad
    \bar{M}_1 = M - \thalf m_0 = \bar{M}_0 - \thalf .
  \]

  Then \( f(t,\cdot) \) converges to a sum of two opposite peaks:
  \[ f(t,\cdot) \to m_0 \delta_{\bar{M}_0} + m_1 \delta_{\bar{M}_1}
    \qquad \text{as \quad $t \to \infty$.}
  \]
\end{theorem}

\begin{proof}
  We first show that $\supp f(t,\cdot) \subset I_0 \cup I_1$ for all $t \ge 0$.
  Let $I_0 = [\alpha, \beta]$,
  then $I_1 = [\alpha',\beta'] = [\alpha+\half, \beta+\half]$;
  let $\eps = \Vol(I_0) = \beta-\alpha < \min\{\theta_1,\thalf-\theta_2\}$.
  Let $h : S^1 \to \R_+$ with $\supp h \subset I_0 \cup I_1$.
  Because $V > 0$
  on $\left]0,\theta_1\right[ \cup \left]\half, 1-\theta_2\right[$, we see that
  \begin{displaymath}
    (V*h)(\beta)
    = \int_{\alpha}^{\beta}
        V(\underbrace{\beta - \psi}_{0 < \bullet <\theta_1}) h(\psi) \,d\psi
      + \int_{\alpha'}^{\beta'}
        V(\underbrace{\beta - \psi}_{\half < \bullet < 1-\theta_2}) h(\psi)\,d\psi
    > 0 .
  \end{displaymath}
  Similarly,
  \begin{displaymath}
    (V*h)(\alpha)
    = \int_{\alpha}^{\beta}
        V(\underbrace{\alpha - \psi}_{-\theta_1< \bullet <0}) h(\psi) \,d\psi
      + \int_{\alpha'}^{\beta'}
        V(\underbrace{\alpha - \psi}_{\theta_2 <\bullet< \half }) h(\psi)\,d\psi
    < 0 ,
  \end{displaymath}
  because \( V \) is negative on both intervals.
  In the same way, we get $(V*h)(\alpha') > 0$ and $(V*h)(\beta') < 0$.
  By Lemma~\ref{L:invsupp} it follows that $\supp f(t,\cdot) \subset I_0 \cup I_1$
  for all \( t \ge 0 \).

  By Lemma~\ref{L:local}, the local masses $m_j(t)$ are then constant,
  and the same is true for $M(t)$ (since $f(t,\cdot) = 0$ on $S^1 \setminus I$
  for all $t \ge 0$).
  We now define local first and second moments by
  \[  M_j(t) = \int_{I_j} (\theta-\bar{M}_j)  f(t,\theta)\,d\theta
      \qquad \text{and} \qquad
      m_{2,j}(t) = \int_{I_j} (\theta-\bar{M}_j)^2 f(t,\theta)\,d\theta
      \quad \text{for $j \in \{0,1\}$.}
  \]
  Note that the expression $\theta-\bar{M}_j$ makes sense on~$I_j$ (even on~$I$ ---
  we lift to a suitable interval in~$\R$ and compute the difference there).
  The definitions imply that
  \( M_0(t) + M_1(t) = M - m_0 \bar{M}_0 - m_1 \bar{M}_1 = 0 \)
  for all \( t \ge 0 \).
  Let $m_2(t) = m_{2,0}(t) + m_{2,1}(t)$. We will
  show that $m_2(t) \to 0$ as $t \to \infty$. The time derivative of~$m_2$ is
  (after integration by parts)
  \begin{align*}
    \frac{dm_2}{dt}(t)
      &= -2 \Bigl(\int_{I_0} \int_{I_0} + \int_{I_0} \int_{I_1}\Bigr) \:
            V(\theta-\psi) (\theta-\bar{M}_0) f(t,\psi) f(t,\theta) \, d\psi\,d\theta
        \\
      & \qquad{}
        - 2 \Bigl(\int_{I_1} \int_{I_0} + \int_{I_1} \int_{I_1}\Bigr) \:
          V(\theta-\psi) (\theta-\bar{M}_1) f(t,\psi) f(t,\theta) \, d\psi\,d\theta .
  \end{align*}
  To estimate this, we observe that there is $b > 0$ such that
  \[ V(\phi) \phi \ge b \phi^2 \qquad\text{and}\qquad
    V(\phi+\thalf) \phi \ge b \phi^2
    \qquad \text{for all $\phi \in [-\eps,\eps]$.}
  \]
  This is because $V > 0$ on $\left]0,\eps\right]$ and on
  $\left]\thalf,\thalf+\eps\right]$ and because $V(0) = V(\thalf) =0$,
  $V'(0) > 0$ and $V'(\thalf) > 0$.
  We now bound the various integrals from below. For the first, we find
  \begin{align*}
    2 \int_{I_0} & \int_{I_0}
        V(\theta-\psi) (\theta-\bar{M}_0) f(t,\psi) f(t,\theta) \,d\psi\,d\theta \\
      &= \int_{I_0} \int_{I_0}
          V(\theta-\psi) (\theta-\bar{M}_0) f(t,\psi) f(t,\theta) \,d\psi\,d\theta
          + \int_{I_0} \int_{I_0}
              V(\psi-\theta) (\psi-\bar{M}_0) f(t,\psi) f(t,\theta) \,d\psi\,d\theta \\
      &= \int_{I_0} \int_{I_0}
          V(\theta-\psi) (\theta-\psi) f(t,\psi) f(t,\theta) \,d\psi\,d\theta \\
      &\ge b \int_{I_0} \int_{I_0}
              (\theta-\psi)^2 f(t,\psi) f(t,\theta) \,d\psi\,d\theta \\
      &= b \int_{I_0} \int_{I_0}
            \bigl((\theta-\bar{M}_0) - (\psi-\bar{M}_0)\bigr)^2
                f(t,\psi) f(t,\theta) \,d\psi\,d\theta \\
      &= 2 b \bigl(m_{2,0}(t) m_0 - M_0(t)^2\bigr) .
  \end{align*}
  In the same way, we find for the fourth integral that
  \[ \int_{I_1} \int_{I_1}
      V(\theta-\psi) (\theta-\bar{M}_1) f(t,\psi) f(t,\theta) \,d\psi\,d\theta
      \ge b \bigl(m_{2,1}(t) m_1 - M_1(t)^2\bigr) .
  \]
  The remaining two integrals are estimated together, as follows.
  \begin{align*}
    \int_{I_0} & \int_{I_1}
        V(\theta-\psi) (\theta-\bar{M}_0) f(t,\psi) f(t,\theta) \,d\psi\,d\theta
    + \int_{I_1} \int_{I_0}
          V(\theta-\psi) (\theta-\bar{M}_1) f(t,\psi) f(t,\theta) \,d\psi\,d\theta \\
      &= \int_{I_0} \int_{I_1}
             V(\theta-\psi) \bigl((\theta-\bar{M}_0)-(\psi-\bar{M}_1)\bigr)
                                  f(t,\psi) f(t,\theta) \,d\psi\,d\theta \\
      &\ge b \int_{I_0} \int_{I_1}
            \bigl((\theta-\bar{M}_0) - (\psi-\bar{M}_1)\bigr)^2
                f(t,\psi) f(t,\theta) \,d\psi\,d\theta \\
      &= b \bigl(m_{2,0}(t) m_1 - 2 M_0(t) M_1(t) + m_{2,1}(t) m_0\bigr) .
  \end{align*}
  Adding up, we find that (recalling that $M_0(t) + M_1(t) = 0$)
  \[ \frac{dm_2}{dt}(t)
       \le -2 b \bigl((m_{2,0}(t) + m_{2,1}(t))(m_0 + m_1)
                        - M_0(t)^2 - 2 M_0(t) M_1(t) - M_1(t)^2\bigr)
       = -2 b m_2(t) .
  \]
  This shows that $m_2(t) \le e^{-2b t} m_2(0)$, and since $m_2(t) \ge 0$, this
  implies $m_2(t) \to 0$ as $t \to \infty$. So the local second moments
  $m_{2,0}(t)$ and~$m_{2,1}(t)$ tend to zero as well.
  Using Lemma~\ref{convergence-peak} on the intervals \( I_0, I_1 \) separately,
  it follows that for \( t \to \infty \) the solution converges to two peaks,
  \[ f(t,\cdot) \stackrel{\D}{\to} m_0 \delta_{\bar{M}_0} + m_1 \delta_{\bar{M}_1} , \]
  where the distance between the peaks is $\bar{M}_0 - \bar{M}_1 = \thalf$.
\end{proof}

%%%%%%%%%%%%%%%%%%%%%%%%%%%%%%%%%%%%%%%%%%%%%%%%%%%%%%%%%%%%%%%%%%%%%%%%%%%%%%%%%%%

\section{Linear stability of peaks}

\subsection{Stability of position} \label{stability-of-position} \strut

We start this section by using the `peak ansatz' of Mogilner, Edelstein-Keshet
and Ermentrout~\cite{Mogilner2}.
The initial distribution is a sum of \( n \ge 2 \) peaks at positions
\( \theta_j(0) \in S^1 \)
and with masses \hbox{\( m_j > 0 \)} where \( \sum_{j=1}^n m_j = 1 \)
(different masses are a generalization of~\cite{Mogilner2}).
The solution keeps this form,
\( f(t,\theta) = \sum_{j=1}^n m_j \delta_{\theta_j(t)}(\theta) \),
and the positions \( \theta_j(t) \) satisfy the following system of ordinary differential equations.
\begin{equation} \label{ode}
  \frac{d \theta_j}{dt}(t)
     = - \sum_{k=0}^{n-1} m_k V\bigl(\theta_j(t) - \theta_k(t)\bigr)
  \quad \text{for $j=0,\ldots,n-1$.}
\end{equation}
To see this, we write
\( \delta_{\theta_j} = \delta_0(.-\theta_j) = \delta(.-\theta_j) \)
and plug \( f(t,\cdot) = \sum_j m_j \delta(.-\theta_j(t)) \) 
into the transport equation \( \partial_t f = \partial_\theta ((V*f) f) \).
For the left hand side we get
\begin{equation} \label{h1}
  \dfdt(t,\cdot)
     = \sum_{j=0}^{n-1}
           m_j \, \left(-\frac{d \theta_j}{dt}(t) \right) \, \delta'(.-\theta_j(t))\,,
\end{equation}
and for the right hand side
\begin{align}
  \partial_\theta \left((V*f(t,\cdot)) f(t,\cdot)\right)
    &= \partial_\theta \left(
         \sum_{j=0}^{n-1} \sum_{k=0}^{n-1} m_j m_k
              V\bigl(\theta_j(t)-\theta_k(t)\bigr) \, \delta(.-\theta_j(t)) \right)
                   \nonumber \\
    &= \sum_{j=0}^{n-1} \sum_{k=0}^{n-1} m_j m_k
         V\bigl(\theta_j(t)-\theta_k(t)\bigr) \, \delta'(.-\theta_j(t)).
       \label{h2}
\end{align}
Comparing \eqref{h1} and~\eqref{h2} we deduce~\eqref{ode}.

The case \( n=2 \) is interesting.
Since \( V(0) = 0 \) and \( V \) is odd, the system is
\begin{align*}
  \dot{\theta_0} &= - m_0 V(\theta_0-\theta_1) \\
  \dot{\theta_1} &= - m_1 V(\theta_1-\theta_0) = m_1 V(\theta_0-\theta_1)
\end{align*}
hence (recall that \( m_0 + m_1 = 1 \))
\begin{equation}
\label{ode2}
  \frac{d}{dt}(\theta_0 - \theta_1) = - V(\theta_0 - \theta_1).
\end{equation}
Because \( \theta_j(t) \to \bar{\theta}_j \) implies
\( \delta_{\theta_j(t)} \stackrel{\D}{\to} \delta_{\bar{\theta}_j} \),
we may conclude the following.

\begin{example} \label{two-peaks}
  Let \( 0 \le \theta_v \le \half \) and \( V \in \C^{1}(S^1) \) odd with
  \( V > 0 \) on \( \left]0,\theta_v\right[\ \)
  and \( V < 0 \) on \(\left]\theta_v,\thalf\right[ \).
  If \( \dist(\theta_0(0), \theta_1(0)) < \theta_v \),
  then \( \theta_0(t)-\theta_1(t) \to 0 \) for \( t \to \infty \),
  i.e., the solution of~\eqref{ode} converges to a single peak;
  if \( \dist(\theta_0(0), \theta_1(0)) > \theta_v \),
  then \(  \dist(\theta_0(t), \theta_1(t)) \to \thalf \),
  hence the solution of~\eqref{ode} converges to two opposite peaks.
  \( \dist(\theta_0, \theta_1) = \theta_v \) is an unstable stationary solution.
\end{example}

Now we are in a good position to show that one gets into trouble when defining
a `first moment' in the `obvious' naive way by
\( \int_{-\half}^{\half} p^*(f)(t,x) x \,dx \).
The point is that this is (in general) {\em not} time-invariant.

\begin{example} \label{Ex:moment}
  Let \( 0 < \eps < \tfrac{1}{8} \),
  \( V \) odd with \( V(\theta) = \theta \) on \( [0,4 \eps] \) and
  \[ f(0,\cdot) = \thalf \delta_{\theta_0(0)} + \thalf \delta_{\theta_1(0)}
    \quad \text{where} \quad
    \theta_0(0) = p(-\thalf+\eps) \quad\text{and}\quad
    \theta_1(0) = p(\thalf-3 \eps)  . \]
  Then
  \[ f(t,\cdot) \to \thalf \:
            \big( \delta(.-(-\thalf-\eps)) + \delta(.-(\thalf-\eps)) \big)
            \stackrel{S^1}{=} \delta(.-(\thalf-\eps))
    \quad \text{as $t \to \infty$.} \]

  To see this, note that \\
  i) \( \frac{d \theta_0}{dt}(0) = -\thalf V(\theta_0(0)-\theta_1(0))
          = -\thalf V(4 \eps) < 0 \)
  and \( \frac{d \theta_1}{dt}(0) > 0 \);
  hence, \( \dist(\theta_0(t),\theta_1(t)) \) is decreasing in \( t = 0 \);

  ii) \( \frac{d}{dt} (1+\theta_0(t) - \theta_1(t))
        \stackrel{\eqref{ode2}}{=} - V(1+\theta_0(t) - \theta_1(t))
        = - (1+\theta_0(t) - \theta_1(t)) \)
  as long as \( \dist(\theta_0(t),\theta_1(t)) \le 4 \eps \).\\
  Since \( \dist(\theta_0(0),\theta_1(0)) = 4 \eps \), i) and ii) imply that
  \( \dist(\theta_0(t),\theta_1(t)) = 1 + \theta_0(t) - \theta_1(t) \to 0 \)
  as \( t \to \infty \);

  iii) \( \frac{d}{dt} (\theta_0(t) + \theta_1(t)) \stackrel{\eqref{ode}}{=} 0 \),
  therefore, \( \theta_0(t) + \theta_1(t) = -2 \eps \pmod{1} \) for all \( t \ge 0 \).
  These facts imply that \( \theta_0(t) \to -\half - \eps \)
  and \( \theta_1(t) \to \half - \eps \).

  The `first moment' of the initial distribution is
  \( \int_{-\half}^{\half} p^*(f)(0,x) x \, dx = - \eps \).
  The `first moment' of the limit is
  \( \int_{-\half}^{\half} \delta(x-(\thalf - \eps)) x \,dx = \thalf - \eps \).
  Therefore, this `first moment' is not invariant.
  (In fact, it jumps by $\half$ when one of the two peaks moves through the
  point $p(\half) \in S^1$.)
\end{example}

We will now analyze the {\em local} stability of two selected stationary solutions,
namely peaks in {\em one} place, i.e., \( \theta_j = \theta_0 \) for all
\( 0 \le j < n \), and peaks with {\em equal} masses at {\em equal distances},
i.e., \(m_j = \tfrac{1}{n} \) and \( \theta_0 \in S^1 \),
\( \theta_j = \theta_{j-1} + \tfrac{1}{n} \) for \( 1 \le j < n \).
Obviously, both are stationary solutions of equation~\eqref{ode}.
The matrix of the linearization is
\begin{equation}
   A = \left(
   \begin{array}{cccc}
     -\sum\limits_{k=0,k \neq 0}^{n-1} m_k V'(\theta_0-\theta_k) &
          m_1 V'(\theta_0-\theta_1) & \ldots &
          m_n V'(\theta_0-\theta_{n-1}) \\
     m_0 V'(\theta_1-\theta_0) &
     -\sum\limits_{k=0,k\neq 1}^{n-1} m_k V'(\theta_1-\theta_k) &
          \ldots & m_{n-1} V'(\theta_1-\theta_{n-1}) \\
     \vdots && \ddots& \\
     m_0 V'(\theta_{n-1}-\theta_0) &
     m_1 V'(\theta_{n-1}-\theta_1) & \ldots &
     -\sum\limits_{k=0,k \neq n-1}^{n-1} m_k V'(\theta_{n-1}-\theta_k)
   \end{array} 
   \right).
\label{permutation-matrix}
\end{equation}

In both cases \( A \) has a clear structure such that the eigenvalues can be
calculated explicitly (remember \( \sum m_k = 1 \) for the first case;
if \( \theta_j - \theta_{j+1} = \tfrac{1}{n} \) and \( m_j = \tfrac{1}{n} \),
then \( A \) is a symmetric and cyclic matrix, because \( V' \) is even).
The eigenvalues are
\[ %\begin{equation} \label{eigenvalues}
  \lambda_j
  = \left\{
     \begin{array}{c@{\quad \text{if} \quad }l}
          0 & j = 0  \\
          -V'(0) &
          1 \le j < n, \text{\ assuming that $\theta_k = \theta_0$ for all~$k$} \\
       \frac{1}{n} \sum_{k=1}^{n-1} \, V'(\frac{k}{n}) (-1 + \cos(2 \pi \frac{j k}{n}))
          & 1 \le j < n, \text{\ assuming that $\theta_k = \theta_0 + \frac{k}{n}$
                                for all~$k$}
      \end{array}
      \right.
\] % \end{equation}
One eigenvalue is zero, because of the translational invariance of the system.
Note that \( \lambda_j = \lambda_{n-j} \).

\begin{theorem} \label{stab-n-peaks} \strut
  \renewcommand{\theenumi}{\alph{enumi}}
  \renewcommand{\labelenumi}{\emph{\theenumi)}}
  \vspace{-2ex}
  \begin{enumerate}\setlength{\parindent}{0mm}
  \item
    A single peak is stable {\em up to translation} in the space of peak solutions
    if and only if $V'(0)$ is positive.
  \item
    Let \( n \ge 2 \);
    \( n \) peaks with equal masses and equal distances are stable
    {\em up to translation}
    in the space of \(n \)-peak solutions with equal masses if \( \sum_{k=1}^{n-1} \,
    V'(\frac{k}{n}) (-1 + \cos(2 \pi \frac{j k}{n})) < 0 \) \((*)\)
    for all \( 1 \le j \le n-1 \).

    For \( n=2,3 \) a necessary and sufficient condition for \((*)\) to hold
    is \( V'(\frac{j}{n}) > 0 \) for all \( 1 \le j \le n-1 \); the condition is
    sufficient for all $n \ge 2$.
  \end{enumerate}
\end{theorem}

\begin{proof}
  If \( V'(0)> 0 \) and \( V'(\frac{j}{n}) > 0 \), respectively, then all
  eigenvalues except \( \lambda_0 \) are negative;
  this implies stability up to translation. We now assume that $(*)$ holds.

  If \( n = 2 \), then
  $0 > \lambda_1 = \thalf V'(\thalf) (-1 + \cos(\pi)) = -V'(\thalf)$,
  so $V'(\half) > 0$.

  If \( n = 3 \), then
  \( 0 > \lambda_1
      = \tfrac{1}{3} ( V'(\tfrac{1}{3}) (-1 + \cos(\tfrac{2 \pi}{3}))
          + V'(\tfrac{2}{3}) (-1 + \cos(\tfrac{4 \pi}{3}))
      = \tfrac{2}{3} V'(\tfrac{1}{3}) (-1 + \cos(\tfrac{2 \pi}{3}))
  \),
  because \( V'(\tfrac{1}{3}) = V'(\tfrac{1}{3}-1) = V'(\tfrac{2}{3}) \),
  so $V'(\tfrac{1}{3}) = V'(\tfrac{2}{3}) > 0$.
\end{proof}

\begin{example}
  Primi et~al.~\cite{Primi} consider examples with
  \( V(\theta) = \sign(\alpha) \, \sin(2 \pi \theta + \alpha \, \sin(2 \pi \theta)) \)
  and find that four-peak like solutions are not stable if \( \alpha = \pm 1.2 \).
  For \( D = 0 \) this is explained now, because
  \[ V'(\tfrac{1}{4})
    = \sign(\alpha) \, \cos(2 \pi \tfrac{1}{4} + \alpha \, \sin(2 \pi \tfrac{1}{4}))
      \, (2 \pi + \alpha 2 \pi \cos(2 \pi \tfrac{1}{4}))
    = -2 \pi \, \sign(\alpha) \, \sin(\alpha)
    < 0
  \]
  for all \( 0 < |\alpha| < \pi \). Note that $V'(\tfrac{1}{4}) > 0$ is still
  necessary for $(*)$ to hold when $n = 4$.
\end{example}

%===============================================================================

\subsection{Stability with respect to small perturbations}
\label{small-perturbations} \strut

We consider the linear stability of the stationary solution $f(t,\cdot) = \delta$
in the space of \emph{differentiable measures} on~$S^1$, $\D^1(S^1)$. Recall that
this is the dual space of~$\C^1(S^1)$ and can be identified with the subspace
of distributions in~$\D(S^1)$ of order at most~1. Note that $\delta_\theta$ is
close to~$\delta$ in~$\D^1(S^1)$ when $\theta$ is small (since
$\langle \delta_\theta-\delta, h \rangle = \theta h'(\tilde\theta)$ for
some $\tilde\theta$ between $0$ and~$\theta$, so that
$\|\delta_\theta-\delta\|_{\D^1} \le |\theta|$).

We formulate a lemma that we will need later.

\begin{lemma} \label{L:switch}
  Let $L$ be a (time-independent) differential operator on~$S^1$, and
  let $\hat{L}$ be another differential operator on~$S^1$ such that
  \[ \langle Lf, h \rangle = \langle f, \hat{L}h \rangle \]
  for $f \in \D(S^1)$ and $h \in \C^\infty(S^1)$. Let $f(t,\cdot) \in \D(S^1)$
  be a solution of the PDE $\partial_t f = L f$. Let $H(t,\cdot) \in \C^\infty(S^1)$
  be the solution of the PDE $\partial_t H = \hat{L} H$ such that $H(0,\cdot) = h$.
  Then $\langle f(t,\cdot), H(-t,\cdot) \rangle$ is constant. In particular,
  \[ \langle f(t,\cdot), h \rangle = \langle f(0,\cdot), H(t,\cdot) \rangle . \]
\end{lemma}

\begin{proof}
  We have
  \begin{align*}
    \frac{d}{dt} \langle f(t,\cdot), H(-t,\cdot) \rangle
      &= \langle \partial_t f(t,\cdot), H(-t,\cdot) \rangle
          + \langle f(t,\cdot), -(\partial_t H)(-t,\cdot) \rangle \\
      &= \langle Lf(t,\cdot), H(-t,\cdot) \rangle
          + \langle f(t,\cdot), -\hat{L}H(-t,\cdot) \rangle \\
      &= \langle f(t,\cdot), \hat{L}H(-t,\cdot) - \hat{L}H(-t,\cdot) \rangle
       = 0 .
  \end{align*}
  Applying this with $\tilde{h} = H(t, \cdot)$ instead of~$h$ to obtain
  $\tilde{H}$, we have $\tilde{H}(-t, \cdot) = h$ and
  \[ \langle f(t,\cdot), h \rangle
       = \langle f(t,\cdot), \tilde{H}(-t, \cdot) \rangle
       = \langle f(0,\cdot), \tilde{H}(0, \cdot) \rangle
       = \langle f(0,\cdot), H(t, \cdot) \rangle . \qedhere
  \]
\end{proof}

Since the solution space of our equation is invariant with respect to translations,
no stationary solution can be absolutely linearly stable. In order to deal
with this technical problem, we will consider perturbations that do not change
the barycenter.

In the following, we will always consider equation~\eqref{tp} with $D = 0$.
If we set $f(t,\cdot) = \delta + \tilde{f}(t,\cdot)$ and linearize, we obtain
the linear PDE
\begin{equation} \label{E:lin}
  \frac{\partial\!\tilde{f}}{\partial t}(t,\theta)
    = \frac{\partial}{\partial\theta}
        \bigl(V \tilde{f}(t,\cdot) 
          + (V * \tilde{f}(t,\cdot)) \delta\bigr)(\theta)
  \,.
\end{equation}

\begin{theorem}[Linear stability of a single peak] \label{theorem-linstab}
  Let $V \in \C^1(S^1)$ be odd and such that $V > 0$ on $\left]0,\half\right[$
  and $V'(0) > 0$. Assume that $\tilde{f}(t,\cdot) \in \D^1(S^1)$ is a solution
  of equation~\eqref{E:lin} such that $\supp \tilde{f}(0,\cdot) \subset I$
  with a closed interval $0 \in I \subset S^1$ with $p(\half) \notin I$.
  We can lift $I$ uniquely to an interval $I' \subset \R$ with $0 \in I'$
  and $p(I') = I$. We assume that
  $\langle \tilde{f}(0,\cdot), 1 \rangle = \langle \tilde{f}(0,\cdot), \ell \rangle = 0$
  where $\ell$ is a function on~$S^1$ that satisfies $\ell(p(x)) = x$ for $x \in I'$.
  Then $\tilde{f}(t,\cdot)$ converges to zero as $t \to \infty$ in~$\D^1(S^1)$.
\end{theorem}

\begin{proof}
  Let $h \in \C^1(S^1)$. We have to show that
  $\langle \tilde{f}(t,\cdot), h \rangle \to 0$ as $t \to \infty$. We have
  \begin{align*}
    \frac{d}{dt} \langle \tilde{f}(t,\cdot), h \rangle
      &= \Bigl\langle \frac{\partial\!\tilde{f}}{\partial t}(t,\cdot), h \Bigr\rangle \\
      &= \Bigl\langle \frac{\partial}{\partial\theta} \bigl(V \tilde{f}(t,\cdot)
                   + (V * \tilde{f}(t,\cdot)) \delta\bigr), h \Bigr\rangle \\
      &= \langle V \tilde{f}(t,\cdot) + (V * \tilde{f}(t,\cdot)) \delta, -h' \rangle \\
      &= -\langle \tilde{f}(t,\cdot), V h' \rangle
           - (V * \tilde{f}(t,\cdot))(0) h'(0) \\
      &= \langle \tilde{f}(t,\cdot), -V (h' - h'(0)) \rangle \,.
  \end{align*}
  (The last equality uses that $V$ is odd.)
  We note that if $h$ is constant, then $\langle \tilde{f}(t,\cdot), h \rangle$
  is constant in time and that if $h = c \ell$ on~$I$, then the same
  is true. Since
  $\langle\tilde{f}(0,\cdot), 1\rangle = \langle\tilde{f}(0,\cdot), \ell\rangle = 0$,
  $\langle \tilde{f}(t,\cdot), h \rangle = 0$ for such~$h$. We can therefore
  restrict to functions~$h$ satisfying $h(0) = h'(0) = 0$. Let $H(t,\cdot)$
  denote the (unique) solution of the initial value problem
  \[ \frac{\partial H}{\partial t} = - V \frac{\partial H}{\partial \theta}\,,
     \qquad H(0, \cdot) = h \,.
  \]
  Then we see by Lemma~\ref{L:switch} that
  $\langle \tilde{f}(t,\cdot), h \rangle
     = \langle \tilde{f}(0,\cdot), H(t,\cdot) \rangle$.
  (In particular, this shows that equation~\eqref{E:lin} has a unique solution
  in~$\D^1(S^1)$ under the given assumptions.) Let $\Phi : \R \times S^1 \to S^1$
  denote the flow associated to $V\!$, i.e.,
  \[ \frac{\partial\Phi}{\partial t}(t,\theta) = V\bigl(\Phi(t,\theta)\bigr) \,,
     \qquad \Phi(0,\theta) = \theta \,.
  \]
  Then $H(t,\Phi(t,\theta)) = h(\theta)$, as can be readily checked. Equivalently,
  $H(t,\theta) = h(\Phi(-t,\theta))$. Now we claim that $H(t,\cdot)|_I$ converges
  to zero in~$\C^1(I)$. For this, note first that for $\theta \in I$, we have
  $\Phi(-t,\theta) \to 0$ as $t \to \infty$ uniformly in~$\theta$ (this is because
  $p(\half)$ is the unique attracting and $p(0)$ the unique repelling fixed
  point of the flow~$\Phi$). So $\|H(t,\cdot)|_I\|_\infty \to |h(0)| = 0$.
  Next, we observe that
  \[ \frac{\partial}{\partial t} \frac{\partial}{\partial\theta} \Phi(-t,\theta)
       = -\frac{\partial}{\partial\theta} \frac{\partial\Phi}{\partial t}(-t,\theta)
       = -\frac{\partial}{\partial\theta} V\bigl(\Phi(-t,\theta)\bigr)
       = -V'\bigl(\Phi(-t,\theta)\bigr) \frac{\partial}{\partial\theta} \Phi(-t,\theta)
     \,.
  \]
  For large~$t$, $\Phi(-t,\theta)$ will be uniformly close to zero, so
  $-V'\bigl(\Phi(-t,\theta)\bigr)$ will be uniformly negative (recall
  that $V'(0) > 0$). This shows that $\frac{\partial}{\partial\theta} \Phi(-t,\theta)$
  tends to zero as $t \to \infty$, uniformly for $\theta \in I$. This in turn
  implies that
  \[ \Bigl|\frac{\partial H}{\partial\theta}(t, \theta)\Bigr|
       = \Bigl|h'\bigl(\Phi(-t,\theta)\bigr)
                \frac{\partial}{\partial\theta} \Phi(-t,\theta)\Bigr|
  \]
  also tends to zero uniformly on~$I$ as $t \to \infty$. So
  \[ \langle \tilde{f}(t,\cdot), h \rangle
       = \langle \tilde{f}(0,\cdot), H(t,\cdot) \rangle \to 0
       \qquad \text{as $t \to \infty$,}
  \]
  and this means that $\tilde{f}(t,\cdot) \to 0$ in~$\D^1(S^1)$.
  More precisely, it follows that
  $\supp \tilde{f}(t,\cdot) \subset \Phi(-t, I)$, so that the support is
  contracted to $\{0\}$, whereas mass and first moment are always zero.
\end{proof}

It is certainly natural to consider perturbations that do not change the total
mass (thinking of redistributing the mass on the circle). What about perturbations
that do not preserve the barycenter? Consider a small perturbation $g$
in~$\D^1(S^1)$ with mass zero and $\langle g, \ell \rangle = M$
with $|M| \ll 1$. Then $\delta + g = \delta_M + (\delta-\delta_M + g)$,
and $\delta-\delta_M + g$ is still a small perturbation, but now of~$\delta_M$.
Assuming that $p(\half) \notin I - M$, the theorem above then predicts convergence
to the shifted peak~$\delta_M$.

In a way, we can see this from the proof. If we do not assume that
$M = \langle \tilde{f}(0,\cdot), \ell \rangle = 0$, then (using test functions~$h$
with $h(0) = 0$, but not assuming $h'(0) = 0$) we find that
\[ \tilde{f}(t,\cdot) \stackrel{\D}{\to} -M \delta' . \]
This is in accordance with $\delta_M - \delta \approx -M \delta'$.

It is perhaps also interesting to compare Theorem~\ref{smallsupport} with
Theorem~\ref{theorem-linstab}. The former shows that an initial distribution
that is contained in an interval covering less than half of the circle will
converge to a delta peak under equation~\eqref{tp} without diffusion. The latter
shows that this peak is stable with respect to small perturbations that avoid an
arbitrarily small neighborhood of the point opposite to the location of the peak.

Proposition \ref{1/n-periodic} yields the following generalization to \( n \) equally distanced peaks with equal masses.

\begin{corollary}[Stability of \( n \) peaks with respect to
                  \( \tfrac{1}{n} \)-periodic perturbations]
  Let $n \ge 1$, let $V \in \C^1(S^1)$ be odd and such that $V_n > 0$ on
  $\left]0,\tfrac{1}{2n}\right[$
  and $V_n'(0) > 0$. Assume that $\tilde{f}(t,\cdot) \in \D^1(S^1)$ is an
  $\tfrac{1}{n}$-periodic solution
  of equation~\eqref{E:lin} such that
  $\supp \tilde{f}(0,\cdot) \subset \bigcup_{j=0}^{n-1} (I+\tfrac{j}{n})$
  with a closed interval $0 \in I \subset S^1$ with $p(\pm\tfrac{1}{2n}) \notin I$.
  We can lift $I$ uniquely to an interval $I' \subset \R$ with $0 \in I'$
  and $p(I') = I$. We assume that
  $\langle \tilde{f}(0,\cdot), 1 \rangle = \langle \tilde{f}(0,\cdot), \ell \rangle = 0$
  where $\ell$ is a function on~$S^1$ that satisfies $\ell(p(x)) = x$ for $x \in I'$
  and $\ell(\theta) = 0$ for $\theta \in \bigcup_{j=1}^{n-1} (I+\tfrac{j}{n})$.
  Then $\tilde{f}(t,\cdot)$ converges to zero as $t \to \infty$ in~$\D^1(S^1)$.
\end{corollary}

We now want to derive a result similar to Theorem~\ref{theorem-linstab}, but
for two opposite peaks of not necessarily equal mass. We take this stationary
solution to be $f_0 = m_- \delta_{-1/4} + m_+ \delta_{1/4}$. If we
take $f = f_0 + \tilde{f}$ in equation~\eqref{tp} with $D = 0$ and linearize,
we obtain
\begin{equation} \label{E:lintwo}
  \begin{array}{r@{}l}
    \dfrac{\partial\!\tilde{f}}{\partial t}(t,\theta)
      = \dfrac{\partial}{\partial\theta}
            \big(&(m_- V(\theta+\tfrac{1}{4}) + m_+ V(\theta-\tfrac{1}{4}))
                \tilde{f}(t,\theta) \\[2mm]
        &  \quad{} + m_- (V*\tilde{f}(t,\cdot))(-\tfrac{1}{4}) \delta_{-1/4}(\theta)
                   + m_+ (V*\tilde{f}(t,\cdot))(\tfrac{1}{4}) \delta_{1/4}(\theta)\big)
            .
  \end{array}
\end{equation}
We will define
\[ \tilde{V}(\theta) = m_- V(\theta+\tfrac{1}{4}) + m_+ V(\theta-\tfrac{1}{4}) . \]

\begin{theorem}[Linear stability of two peaks] \label{stab-2peaks}
  Let $V \in \C^1(S^1)$ be odd and such that $V'(0) > 0$ and $V'(\half) > 0$.
  With the notations $m_-$, $m_+$ and~$\tilde{V}$ from above, suppose that
  $\tilde{V}$ has exactly four zeros on~$S^1\!\!$, namely $-\tfrac{1}{4}$, $\theta_0$,
  $\tfrac{1}{4}$ and~$\theta_1$ (in counter-clockwise order).
  Assume that $\tilde{f}(t,\cdot) \in \D^1(S^1)$ is a solution
  of equation~\eqref{E:lintwo} such that $\supp \tilde{f}(0,\cdot) \subset I_- \cup I_+$
  with closed intervals $\pm\tfrac{1}{4} \in I_{\pm} \subset S^1$
  such that $I_- \subset p\left(\left]\theta_1-1,\theta_0\right[\right)$
  and $I_+ \subset p\left(\left]\theta_0,\theta_1\right[\right)$.
  Let $\ell \in \C^\infty(S^1)$ be such that $\ell(\theta) = \theta-(\pm\tfrac{1}{4})$
  for $\theta \in I_{\pm}$. We assume that
  $m(I_+, \tilde{f}(0,\cdot)) = m(I_-, \tilde{f}(0,\cdot))
    = \langle \tilde{f}(0,\cdot), \ell \rangle = 0$.
  Then $\tilde{f}(t,\cdot)$ converges to zero as $t \to \infty$ in~$\D^1(S^1)$.
\end{theorem}

Note that $\tilde{V}$ has to have at least four zeros, since $\tilde{V}'$ is
positive at the two zeros at $\pm\tfrac{1}{4}$.

Note also that because of the translational invariance, any two peaks with
distance \( \half \) are stable under the assumptions of Theorem~\ref{stab-2peaks}.

\begin{proof}
  We proceed in a similar way as in the proof of Theorem~\ref{theorem-linstab}.
  We find that
  \[ \frac{\partial}{\partial t} \langle \tilde{f}(t,\cdot), h \rangle
       = -\bigl\langle \tilde{f}(t,\cdot),
                     m_- V(\cdot+\tfrac{1}{4})(h'-h'(-\tfrac{1}{4}))
                      + m_+ V(\cdot-\tfrac{1}{4})(h'-h'(\tfrac{1}{4}))\bigr\rangle \,.
  \]
  So we let $H(t,\theta)$ be the solution of
  \[ \frac{\partial H}{\partial t}(t,\theta)
       = -m_- V(\theta+\tfrac{1}{4})
           \Bigl(\frac{\partial H}{\partial\theta}(t,\theta)
                  - \frac{\partial H}{\partial\theta}(t,-\tfrac{1}{4})\Bigr)
         - m_+ V(\theta-\tfrac{1}{4})
           \Bigl(\frac{\partial H}{\partial\theta}(t,\theta)
                  - \frac{\partial H}{\partial\theta}(t,\tfrac{1}{4})\Bigr)
  \]
  with $H(0,\cdot) = h$; then
  $\langle \tilde{f}(t,\cdot), h \rangle
      = \langle \tilde{f}(0,\cdot), H(t,\cdot) \rangle$
  by Lemma~\ref{L:switch}, and we have to figure out the longterm behavior
  of~$H$. We see that a function~$h$ that is constant separately on $I_-$
  and on~$I_+$ is a stationary solution on $I_- \cup I_+$ and that the same is true
  when $h$ is a multiple of~$\ell$. So we can assume that
  $h(\tfrac{1}{4}) = h(-\tfrac{1}{4})
     = m_+ h'(\tfrac{1}{4}) + m_- h'(-\tfrac{1}{4}) = 0$.
  We write $H'$ for $\frac{\partial H}{\partial\theta}$. Then we have that
  \begin{align*}
    \frac{d}{dt} H'(t, \tfrac{1}{4})
       &= -m_- V'(\thalf) \bigl(H'(t, \tfrac{1}{4}) - H'(t, -\tfrac{1}{4})\bigr)
       \intertext{and}
    \frac{d}{dt} H'(t, -\tfrac{1}{4})
       &= -m_+ V'(\thalf) \bigl(H'(t, -\tfrac{1}{4}) - H'(t, \tfrac{1}{4})\bigr) .
  \end{align*}
  This shows that
  $m_+ H'(t,\tfrac{1}{4}) + m_- H'(t,-\tfrac{1}{4}) = 0$ for all~$t$ and that
  $H'(t,\tfrac{1}{4}) - H'(t,-\tfrac{1}{4}) \to 0$ as $t \to \infty$ (recall
  that $V'(\thalf) > 0$). So $H'(t,\tfrac{1}{4}) \to 0$ and $H'(t,-\tfrac{1}{4}) \to 0$.
  By arguments similar to those in the proof of Theorem~\ref{theorem-linstab}
  (note that in the present situation, the flow associated to~$\tilde{V}$
  moves the values of~$h$ away
  from $\tfrac{1}{4}$ and~$-\tfrac{1}{4}$ and toward $\theta_0$ and~$\theta_1$),
  we then see that $H(t, \cdot) \to 0$ in~$\C^1(I_- \cup I_+)$ and therefore
  $\langle \tilde{f}(t,\cdot), h \rangle \to 0$ as $t \to \infty$. For a general
  test function~$h$, we then find that
  \[ H(t,\cdot) \to h(\tfrac{1}{4}) \chi_+ + h(-\tfrac{1}{4}) \chi_-
                    + (m_+ h'(\tfrac{1}{4}) + m_- h'(-\tfrac{1}{4})) \ell
         \quad\text{on $I_- \cup I_+$,}
  \]
  where $\chi_{\pm}$ is a function in~$\C^\infty(S^1)$ that takes the value~1
  on~$I_{\pm}$ and the value~0 on~$I_{\mp}$. This translates into
  \[ \tilde{f}(t,\cdot) \stackrel{\D}{\to}
       m(I_+,\tilde{f}(0,\cdot)) \delta_{1/4} + m(I_-,\tilde{f}(0,\cdot)) \delta_{-1/4}
         - \langle \tilde{f}(0,\cdot), \ell \rangle
              (m_+ \delta'_{1/4} + m_- \delta'_{-1/4})
       = 0 . \qedhere
  \]
\end{proof}

The need for the three assumptions
$m(I_+, \tilde{f}(0,\cdot)) = m(I_-, \tilde{f}(0,\cdot))
    = \langle \tilde{f}(0,\cdot), \ell \rangle = 0$
arises because two opposite peaks of arbitrary masses and arbitrary orientation
form a stationary solution. If we have a perturbation that violates these
assumptions (but does not change the total mass), say
\[ m(I_+, \tilde{f}(0,\cdot)) = \mu, \quad
   m(I_-, \tilde{f}(0,\cdot)) = -\mu \quad\text{and}\quad
   \langle \tilde{f}(0,\cdot), \ell \rangle = M ,
\]
then we can proceed as in the one-peak case. We adjust masses and orientation
to obtain
\[ (m_+ + \mu) \delta_{1/4+M} + (m_- - \mu) \delta_{-1/4+M} \]
as a stationary solution such that the resulting perturbation of this solution
satisfies the assumptions.

If there is some $0 < \theta_v < \thalf$ such that $V(\theta_v) = 0$ and
$V'(\theta_v) > 0$ (and $V'(0) > 0$, of course), then we expect two peaks
at a distance of~$\theta_v$ also to be a stable stationary solution,
up to a redistribution of mass between the two peaks
and reorientation that preserves the distance. This is indeed the case.

\begin{corollary} \label{twopeaks2}
  Let $V \in \C^1(S^1)$ be odd and such that $V'(0) > 0$, and assume that
  there is $0 < \theta_v < \thalf$ such that $V(\theta_v) = 0$ and
  $V'(\theta_v) > 0$. Let $m_{\pm} > 0$ with $m_+ + m_- = 1$, and consider
  the stationary solution $f_0 = m_+ \delta_{\theta_v/2} + m_- \delta_{-\theta_v/2}$
  of equation~\eqref{tp} with~$D = 0$. Let
  $\tilde{V}(\theta) = m_- V(\theta + \tfrac{\theta_v}{2})
                         + m_+ V(\theta - \tfrac{\theta_v}{2})$
  and suppose that $\tilde{V}$ has exactly four zeros on~$S^1\!\!$, namely
  $-\tfrac{\theta_v}{2}$, $\theta_0$,
  $\tfrac{\theta_v}{2}$ and~$\theta_1$ (in counter-clockwise order).
  Then $f_0$ is linearly stable with respect to perturbations satisfying the
  conditions in Theorem~\ref{stab-2peaks} (with $\pm\tfrac{1}{4}$ replaced
  by $\pm\tfrac{\theta_v}{2}$).
\end{corollary}

\begin{proof}
  The proof is virtually identical to the proof of Theorem~\ref{stab-2peaks},
  after replacing $\pm\tfrac{1}{4}$ by $\pm\tfrac{\theta_v}{2}$.
\end{proof}

We saw that single peaks are stable up to reorientation if \( V'(0) > 0 \).
Two opposite peaks are stable up to redistribution of mass and reorientation
preserving the distance under the assumptions of Theorem~\ref{stab-2peaks},
which include \( V'(0) > 0 \) and \( V'(\thalf) > 0 \).
Now we prove that \( n \ge 3 \) equal peaks at equal distances are stable
if \( V'(\tfrac{j}{n}) > 0 \) for \( 0 \le j \le n-1 \), up to ?.

\begin{theorem} \label{theorem-stab-n-peaks}
  Let $n \ge 3$. Let $V \in \C^1(S^1)$ be odd and such that $V'(\tfrac{j}{n}) > 0$
  for all $0 \le j < n$ and $V_n > 0$ on~$\left]0,\tfrac{1}{2n}\right[$.
  Let, for $0 \le j < n$, $I_j$ be a closed interval in~$S^1$ contained in
  $\left]\tfrac{j}{n}-\tfrac{1}{2n},\tfrac{j}{n}+\tfrac{1}{2n}\right[$,
  and let $\ell \in \C^\infty(S^1)$ be such that $\ell(\theta) = \theta-\tfrac{j}{n}$
  for $\theta \in I_j$, for all~$j$.
  Then
  \[ \text{the stationary solution} \qquad
     f_0 = \frac{1}{n} \sum_{j=0}^{n-1} \delta_{j/n}
     \qquad \text{of equation~\eqref{tp} with $D = 0$}
  \]
  is linearly stable with respect to perturbations $\tilde{f} \in \D^1(S^1)$
  such that
  \[ \supp \tilde{f} \subset \bigcup_{j=0}^{n-1} I_j, \qquad
     m(I_j, \tilde{f}) = 0 \quad \text{for all $0 \le j < n$, \quad and \quad}
     \langle \tilde{f}, \ell \rangle = 0 .
  \]
\end{theorem}

\begin{proof}
  The proof proceeds in a way analogous to the proofs of Theorems
  \ref{theorem-linstab} and~\ref{stab-2peaks}. The equation governing the
  development of~$H(t,\cdot)$ is
  (writing again $H'$ for $\frac{\partial H}{\partial\theta}$)
  \begin{equation} \label{E:Heqmany}
     \frac{\partial H}{\partial t}(t,\theta)
       = -\frac{1}{n} \sum_{j=0}^{n-1}
                  V(\theta-\tfrac{j}{n})\bigl(H'(t,\theta) - H'(t,\tfrac{j}{n})\bigr) .
  \end{equation}
  The flow associated to~$V_n = \sum_j V(\cdot - \tfrac{j}{n})$ moves away
  from the points $\tfrac{j}{n}$ toward the points $\tfrac{j}{n} \pm \tfrac{1}{2n}$.
  So for any test function~$h$ satisfying $h(\tfrac{j}{n}) = h'(\tfrac{j}{n}) = 0$
  for all~$j$, we find that $\langle \tilde{f}(t,\cdot), h \rangle \to 0$
  as $t \to \infty$ in the same way as before. For the derivatives $H'(t,\tfrac{j}{n})$
  we obtain the equation (using $V_n(\tfrac{j}{n}) = 0$)
  \[ \frac{d}{dt} H'\Bigl(t,\frac{j}{n}\Bigr)
      = -\frac{1}{n} \sum_{k=0}^{n-1} V'\Bigl(\frac{j-k}{n}\Bigr)
                \Bigl(H'\Bigl(t,\frac{j}{n}\Bigr) - H'\Bigl(t,\frac{k}{n}\Bigr)\Bigr) .
  \]
  This leads to
  \[ \frac{d}{dt} \sum_{j=0}^{n-1} H'\Bigl(t,\frac{j}{n}\Bigr)^2
       = -\frac{1}{2n} \sum_{j,k=0}^{n-1} V'\Bigl(\frac{j-k}{n}\Bigr)
                \Bigl(H'\Bigl(t,\frac{j}{n}\Bigr) - H'\Bigl(t,\frac{k}{n}\Bigr)\Bigr)^2
       \le 0 ,
  \]
  with equality only if all $H'(t,\tfrac{j}{n})$ are equal. On the other hand,
  one sees easily that $\sum_j H'(t,\tfrac{j}{n})$ is constant. Together, this
  implies that all $H'(t,\tfrac{j}{n})$ converge to the same value as $t \to \infty$.
  Since functions that are constant on each~$I_j$ and also~$\ell$ are stationary
  under equation~\eqref{E:Heqmany}, we get that
  \[ H(t, \cdot) \to \sum_{j=0}^{n-1} h\Bigl(\frac{j}{n}\Bigr) \chi_j
                       + \frac{1}{n} \sum_{j=0}^{n-1} h'\Bigl(\frac{j}{n}\Bigr) \ell
        \qquad \text{on $\displaystyle\bigcup_{j=0}^{n-1} I_j$,}
  \]
  where $\chi_j \in \C^\infty(S^1)$ is a function that takes the value~1 on~$I_j$
  and the value~0 on all~$I_k$ with $k \neq j$. In terms of~$\tilde{f}$, this reads
  \[ \tilde{f}(t,\cdot) \stackrel{\D}{\to}
       \sum_{j=0}^{n-1} m(I_j, \tilde{f}(0,\cdot)) \delta_{j/n}
         - \langle \tilde{f}(0,\cdot), \ell \rangle
             \frac{1}{n} \sum_{j=0}^{n-1} \delta'_{j/n}
         = 0 . \qedhere
  \]
\end{proof}

As before, if $M = \langle \tilde{f}, \ell \rangle \neq 0$, then we expect a
reorientation by~$M$ in the positive direction. It is less clear what happens
when mass is redistributed between the domains of attraction of the various peaks.
The proof above would suggest that we simply end up with equidistant peaks of
different masses, but this will in general no longer be a stationary solution.
If we consider the system of ODEs~\eqref{ode} for moving peaks, then we see
by the theorem on implicit functions
that there is a unique stationary solution up to translation
near~$f_0$ with peaks of prescribed slightly different masses if the matrix
in~\eqref{permutation-matrix} with $\theta_j = \tfrac{j}{n}$ and $m_j = \tfrac{1}{n}$
only has one vanishing eigenvalue --- it is the matrix obtained as the Jacobian
with respect to the~$\theta_j$ of the map
\[ (m_0,\ldots,m_{n-1};\theta_0,\ldots,\theta_{n-1})
     \longmapsto \Bigl(-\sum_{k=0}^{n-1} m_k V(\theta_j-\theta_k)\Bigr)_{0\le j<n} ,
\]
and the zero eigenvalue corresponds to an overall translation. This condition will be
satisfied when the positions of $n$~equidistant peaks of the same
mass are stable up to translation, since then all the relevant eigenvalues
are negative. Note that the condition on~$V$ in Theorem~\ref{theorem-stab-n-peaks}
is sufficient to ensure this is the case, compare Theorem~\ref{stab-n-peaks}.

In the special case that we have $V(\tfrac{j}{n}) = 0$ for all $0 \le j < n$
the stationary solution of the system~\eqref{ode} will consist of equidistant
peaks even when the masses are not equal. In this case, one can formulate
a variant of Theorem~\ref{theorem-stab-n-peaks} in analogy to Theorem~\ref{stab-2peaks}
that shows that $n$~equidistant peaks with different masses are linearly
stable with respect to perturbations respecting the mass distribution and
the overall orientation.

%%%%%%%%%%%%%%%%%%%%%%%%%%%%%%%%%%%%%%%%%%%%%%%%%%%%%%%%%%%%%%%%%%%%%%%%%%%%%%%%%%

\section{Numerical algorithms and simulations}
\label{numerical-algorithms}

\subsection{Solving the transport-diffusion equation via the Fourier transformed system}
\strut

In Section~\ref{examples} we will calculate numerically solutions of the
transport-diffusion equation~\eqref{tp} for randomly chosen as well as pre-structured
initial distributions.

By using the Fourier transform we convert the partial differential equation into an
infinite (but discrete) system of ordinary differential equations; since large Fourier
coefficients of  a smooth function are small we can then restrict to a finite system
which can be solved very efficiently.

In the Section~\ref{S:eqncircle} we found that the Fourier transform of the
transport-diffusion equation~\eqref{tp} is given by (compare~\eqref{fou_k})
\begin{align}
  \dot{f}_k &= - (2 \pi)^2 k^2 D f_k + 2 \pi i k \sum_{l \in Z} V_l f_l f_{k-l}
                  \nonumber\\
            &= c_k f_k + 4 \pi k \sum_{l \in \Z, l \neq 0,k} v_l f_l f_{k-l}
                \quad \text{for $k \in \Z_{>0}$,}
               \label{fou-k2}
\end{align}
where the eigenvalues \( c_k \) of the system (see~\eqref{ev}) and \( v_k \in \R \)
are
\[ c_k = - (2 \pi)^2 k^2 D + 4 \pi k v_k
   \qquad \text{and} \qquad
   v_k = \int_{0}^{\half} V(\theta) \, \sin(2 \pi k \theta) \, d\theta .
\]
Mass conservation is reflected by \( \dot{f}_0 = 0 \); we put \( f_0 = 1 \).
Note that the number of positive eigenvalues is usually small (see the remarks
after Theorem~\ref{thm:stability}).

In order to avoid the necessity to use very small timesteps (\( k^2 \) is large
for higher modes) we multiply~ \eqref{fou-k2} by \( \exp(-c_k t) \)
and define \( g_k(t) = f_k(t) \exp(-c_k t) \).
Then we get
\begin{equation} \label{numerical}
  \dot{g_k}(t) = (\dot{f_k}(t) - c_k f_k(t)) e^{-c_k t}
               = 4 \pi k e^{-c_k t} \sum_{l\neq 0,k} v_l f_l(t) f_{k-l}(t) ,
\end{equation}
which we solve by a second-order scheme.

The number \( n \) of equations is adapted dynamically, in the following way.
We start with \( n \) Fourier coefficients of \( f \), assuming that higher modes
are zero;
we calculate the right hand side of~\eqref{numerical} for \( 1 \le k \le 2 n \)
and accept for \( 1 \le k \le n \) the resulting \( f_k(t + \Delta t) \) as the
new value for \( f_k \).
If for some \( \tilde{k} > n \) the slope of \( f_{\tilde{k}} \) is larger than
some (small) error bound,
then the number~\( n \) of equations is increased to \( \tilde{k} + 1 \).
The additionally needed Fourier coefficients \( f_k \) with
\( n < k \le \tilde{k} + 1 \) are initialized as zero.

A further advantage of this scheme is that higher
periodicity of an initial function is preserved.

%===============================================================================

\subsection{Solving the stationary equation via iteration} \label{SubS:iteration}
\strut

We also programmed the iteration scheme which Primi et~al.~\cite{Primi} used to prove existence of peak-like solutions.

We start with an arbitrary function \( f^{(0)} \) on~$S^1$ with given mass (usually~1).
E.g., \( f^{(0)} \) may be the solution of
\( D \frac{d f^{(0)}}{d \theta}(\theta) = -V(\theta) f^{(0)}(\theta) \),
which is expected to lie near the one-peak solution, if it exists
(see Primi et~al.~\cite{Primi};
\( V = V*\delta \approx V*f \) if \( f \) is one-peak like).

Then we iterate
\[ D \frac{d f^{(n+1)}}{d \theta}(\theta) = -(V*f^{(n)})(\theta)\,f^{(n+1)}(\theta)
   \quad \text{and require} \quad
     \int_{S^1} f^{(n+1)}(\theta) \,d\theta = \int_{S^1} f^{(0)}(\theta) \,d\theta ,
\]
so that \( f^{(n+1)} \) is a function on~$S^1$ with the same mass as \( f^{(0)} \).

{\em If} this sequence converges, then the limit is obviously a solution
of~\eqref{stateq}, i.e., it is a stationary solution of~\eqref{tp}.
Primi et~al.~\cite{Primi} give criteria for convergence; e.g., the assumptions
\( V'(0) > 0 \) and \( \int_0^{\theta} V(\psi) \, d\psi > 0 \) for
\( \theta \in \left]0,\half\right] \) imply the existence of one-peak like solutions
if \( D \) is small.

Our iteration program reliably finds stationary solutions with
$\frac{1}{n}$-periodicity if no stationary solutions with lower periodicity are present.
Otherwise, it is better to use \( \tilde{V}_n \) and $n^2 D$
instead of \( V \) and~$D$, compare Proposition~\ref{1/n-periodic}.
(Unfortunately, in our program numerical instabilities accumulate, so that
$\frac{1}{n}$-periodicity of \( f^{(0)} \) for \( n > 1 \) is not preserved
numerically, in contrast to the theoretical prediction.)

%=============================================================================

\subsection{Examples} \label{examples}
\strut

In the following examples the stationary solutions were calculated with both
algorithms (exceptions will be mentioned); their stability was checked with the
Fourier based system.

The first example is interesting because it shows a backward bifurcation and mixed
mode solutions.

\begin{example} \label{sin+sin2}
  Let \( D > 0 \), \( \gamma \ge 0 \) and define
  \[ V(\theta) = \sin(2 \pi \theta) + \gamma \, \sin(4 \pi \theta) . \]
  We use formula~\eqref{ev} for the eigenvalues and get
  \[ c_1 = -4 \pi^2 D
          + 4 \pi \int_0^{\half} ( \sin^2(2 \pi \theta)
          + \gamma \sin(4 \pi \theta) \sin(2 \pi \theta)) \, d\theta
        = \pi \, ( - 4 \pi D + 1 ) > 0
        \iff D < \frac{1}{4 \pi}
  \]
  and
  \[ c_2 = - 16 \pi^2 D + 8 \pi \int_0^{\half}
                                        (\sin(2 \pi \theta) \sin(4 \pi \theta)
          + \gamma \sin^2(4 \pi \theta))\, d\theta
        = 2 \pi (-8 \pi D + \gamma) > 0
        \iff D < \frac{\gamma}{8 \pi} .
  \]
  For \( (D,\gamma) \in \R^+ \times \R^+ \) all four combinations of \( c_1 <> 0 \),
  \( c_2 <> 0 \) are possible, see Figure~\ref{ev_diagramm}.
  All other eigenvalues are negative.

  \begin{figure}
  \unitlength4mm
  \begin{center}
  \begin{picture}(10,11)
      \thicklines
      \put(0,0){\vector(1,0){10}}    % D-Achse
      \put(10.1,-0.1){\(D\)}
      \put(0,0){\vector(0,1){10}}    % gamma-Achse
      \put(-0.5,10){\(\gamma\)}
      \linethickness{0.05mm}
      \put(0,2){\line(1,0){10}}  % V'(1/2) = 0
  %%    \put(7,2.3){\( V'(\half) = 0 \)}
      \put(10,1.7){\(\; \gamma = \half, V'(\half) = 0 \)}
      \linethickness{0.5mm}
      \put(0,0){\line(1,6){1.7}}    % gamma = 4 pi D
      \put(0.5,-0.7){\( 1/(4\pi) \)}
      \put(1.7,9.2){\( c_2 = 0 \)}
      \put(1,0){\line(0,6){10}}  % c_1=0
      \put(1.1,0.7){\( c_1 = 0 \)}
      \put(1,6){\( \gamma = 2 \)}
      \put(4,5){\shortstack[c]{\( c_1 < 0 \)\\\( c_2 < 0 \)}}
      \put(-3,0.4){\shortstack[c]{\( c_1 > 0 \)\\\( c_2 < 0 \)}}
      \thinlines
      \put(-0.9,0.5){\vector(1,0){1.4}}
      \put(-3,5){\shortstack[c]{\( c_1 > 0 \)\\\( c_2 > 0 \)}}
      \put(-0.9,5){\vector(1,0){1.3}}
      \put(1,10.8){\( c_1 < 0,\, c_2 > 0 \)}
      \put(1.3,10.5){\vector(0,-1){1}}
      \put(6,2.8){\( V \) has single zero}
      \put(6,0.8){\( V > 0 \)}
  \end{picture}
  \end{center}
  \caption{\label{ev_diagramm}
  Regions of in/stability of first and second eigenvalue
  %for \( V = \sin(2 \pi \theta) + \gamma \sin(4 \pi \theta) \)
  }
  \end{figure}
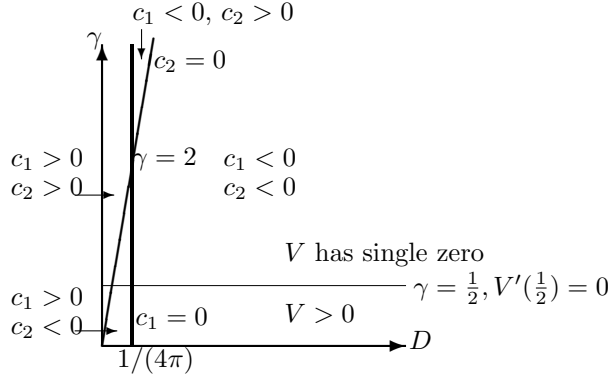

  For all parameter values we have \( V'(0) > 0 \) and
  \( \int_0^{\theta} V(\psi) \, d\psi > 0 \) for \( 0 < \theta < \thalf \);
  therefore for very small diffusion coefficient one-peak like solutions exist
  (Primi et~al.~\cite{Primi}) and at least for \( D = 0 \) they are stable
  by Theorem~\ref{theorem-linstab}.
  We find that \( V_2(\theta) = \gamma \sin(4 \pi \theta) \),
  therefore \( V_2'(0) > 0 \)
  and \( \int_0^{\theta} V_2(\psi) d\psi > 0 \) on \( \left]0,\frac{1}{4}\right[ \),
  thus two-peaks like solutions exist for small enough diffusion
  (Primi et~al.~\cite{Primi}). For \( \gamma < \half \) we have \( V'(0) > 0 \) and
  \( V'(\thalf) < 0 \), therefore two peaks are not stable, see
  Theorem~\ref{stab-n-peaks}; two-peaks like solutions are only stable in the
  space of \( \thalf \)-periodic solutions (if $D$ is sufficiently small),
  see Proposition~\ref{1/n-periodic} and
  Theorem~\ref{theorem-linstab} for \( D = 0 \); note that \( c_2 > 0 \).
  For \( \gamma > \half \), \( V \) has a single simple zero and \( V'(\half) > 0 \),
  so for \( D = 0 \) two-peaks solutions are stable  by Theorem~\ref{stab-2peaks}.

  For \( D = \frac{1}{4 \pi} \approx 0.0796 \) and \( \gamma = 2 \),
  first and second eigenvalue of~\eqref{tp} are zero simultaneously.
  Therefore stationary solutions with two maxima of different height can be
  expected to exist, so-called `mixed mode solutions' (Golubitsky and
  Schaeffer~\cite{GolubitskySchaeffer}); Figure~\ref{sinsin2} (top left figure)
  shows such solutions.

  Figure~\ref{sinsin2} shows typical stationary solutions and their stability
  for \( \gamma = 2 \) and \( \gamma = 4 \).
  The $\thalf$-periodic stationary solutions are unstable (\( \gamma = 2 \)),
  or they become unstable with decreasing \( D \) (\( \gamma = 4 \)).
  This suggests that in general, the stability result for two peaks at \( D = 0 \)
  cannot be carried over to (very small) \( D > 0 \).
  However, solutions need much longer times at smaller \( D \) to move beyond states
  with two peaks of different height.
  E.g.,~for \( D \) of size of the order of~0.05, \( \gamma = 2 \), and starting with
  a small perturbation of \( f = 1 \), we get two~peaks of different heights in the
  first two time units, while convergence to the mixed-mode solution needs about
  30~time units;
  for \( D \approx 0.01 \) and \( \gamma = 2 \) as well as \( \gamma = 4 \), these time
  scales change to one unit for initial pattern formation and several hundred units
  for convergence to one peak.

  In Figure~\ref{sinsin2} the stationary solutions were generated with the
  iteration method, and their stability was tested with the Fourier algorithm.
  The unstable \( \half \)-periodic steady states in Figure~\ref{sinsin2} are stable
  in the space of \( \half \)-periodic functions;
  they are also found with the Fourier-based algorithm if the simulation is started
  with a \( \half \)-periodic function.

  \begin{figure}
  { \Gr{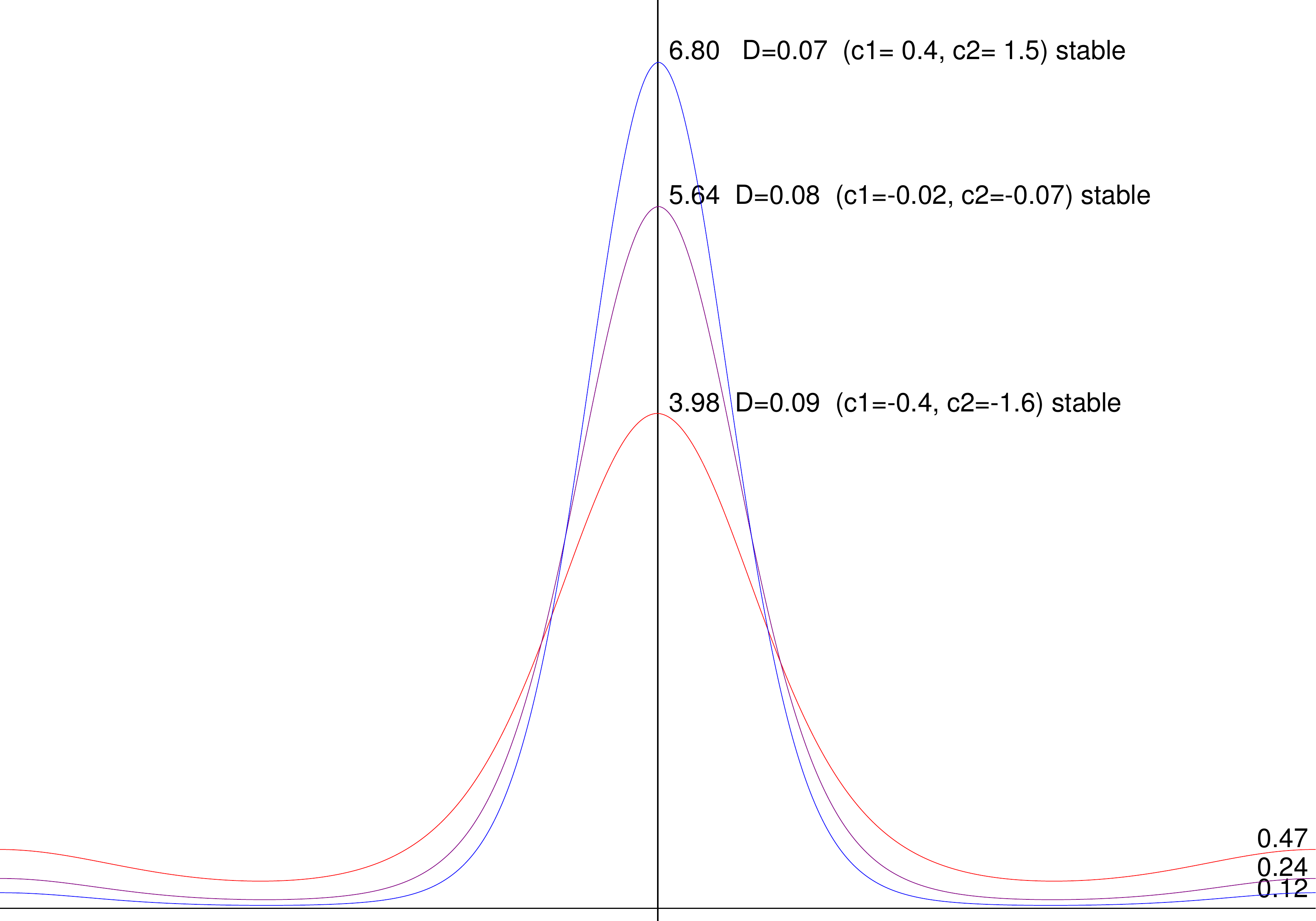}{0.48\textwidth} \hfill
    \Gr{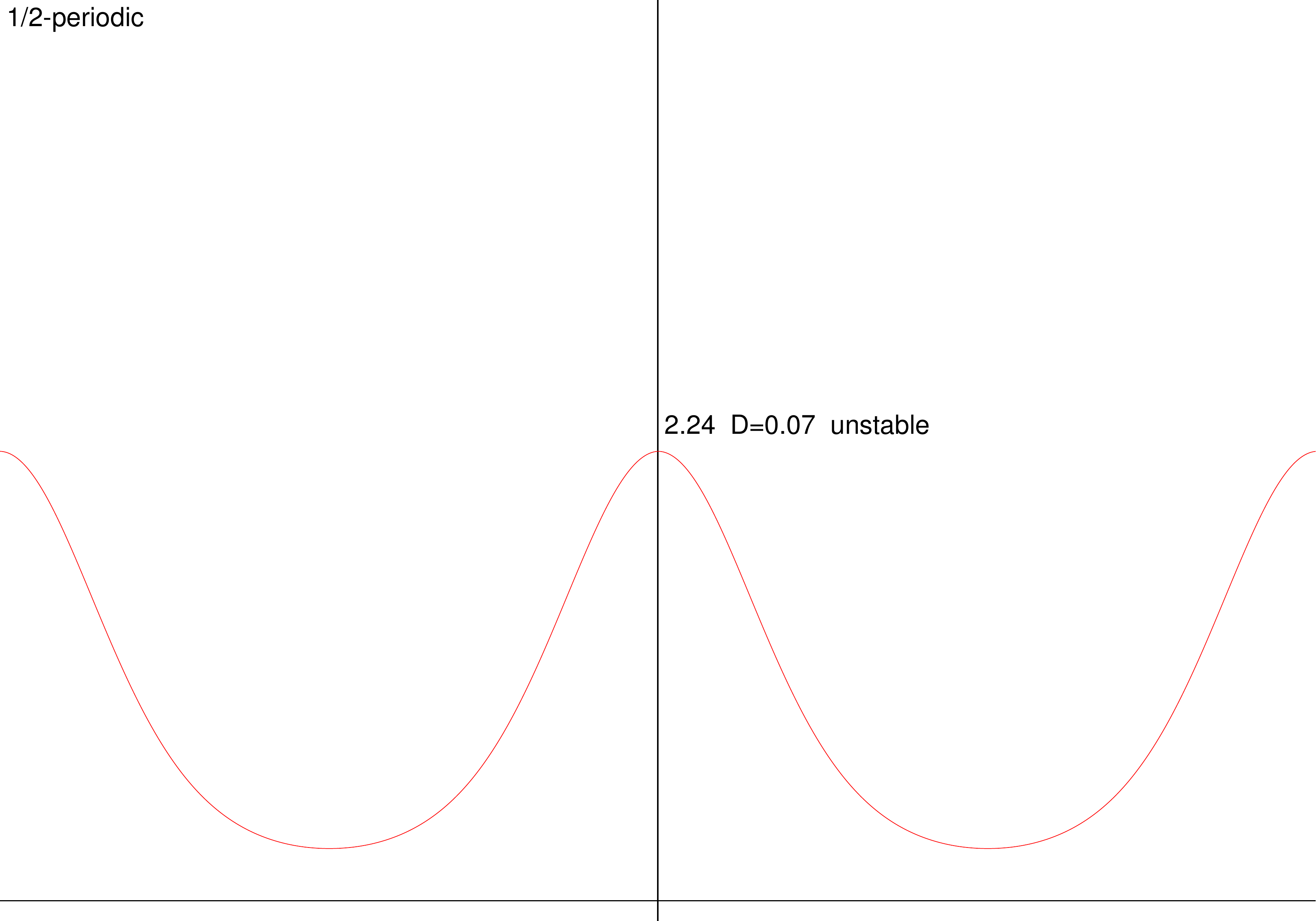}{0.48\textwidth}\\[2ex]
    \Gr{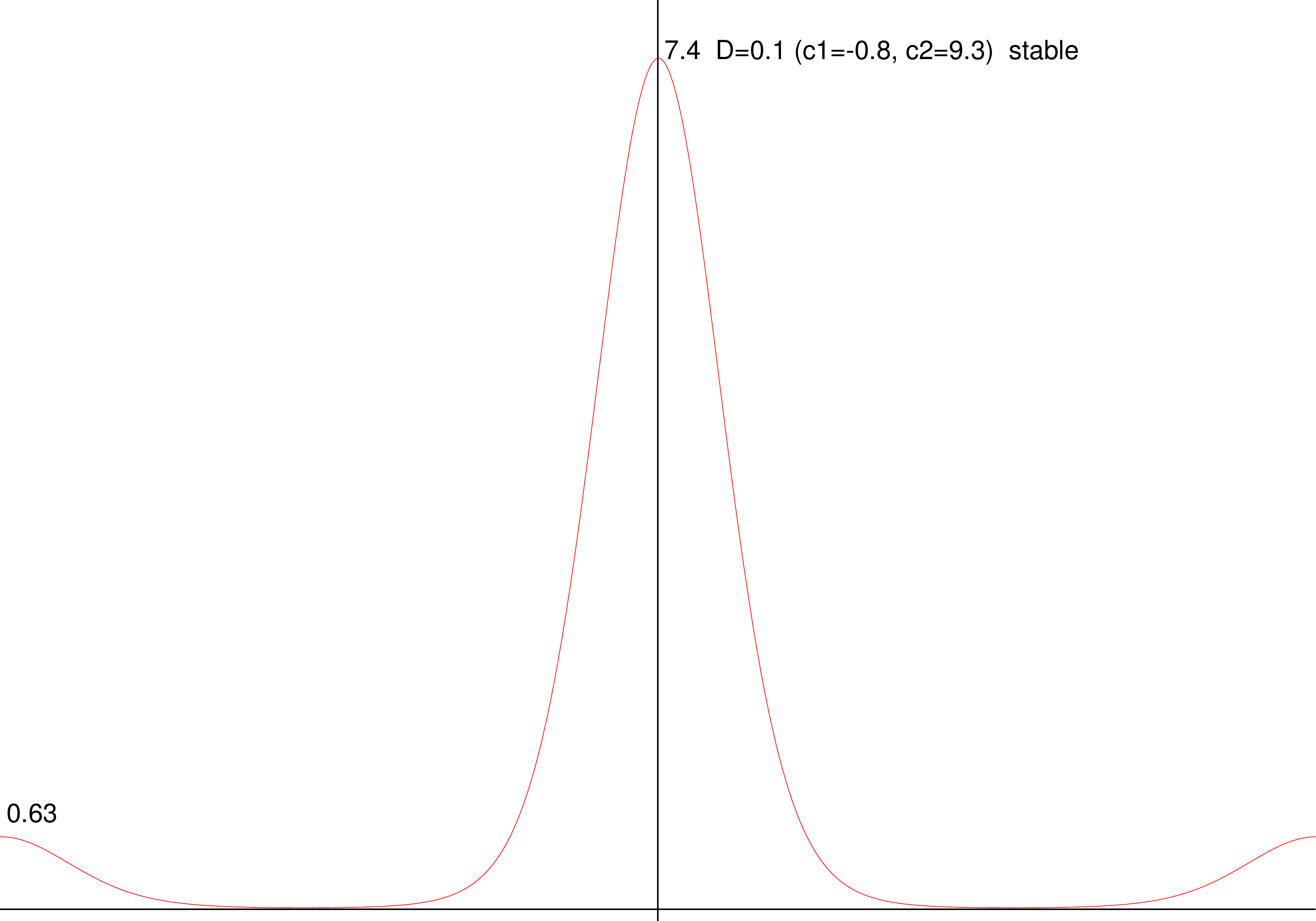}{0.48\textwidth} \hfill
    \Gr{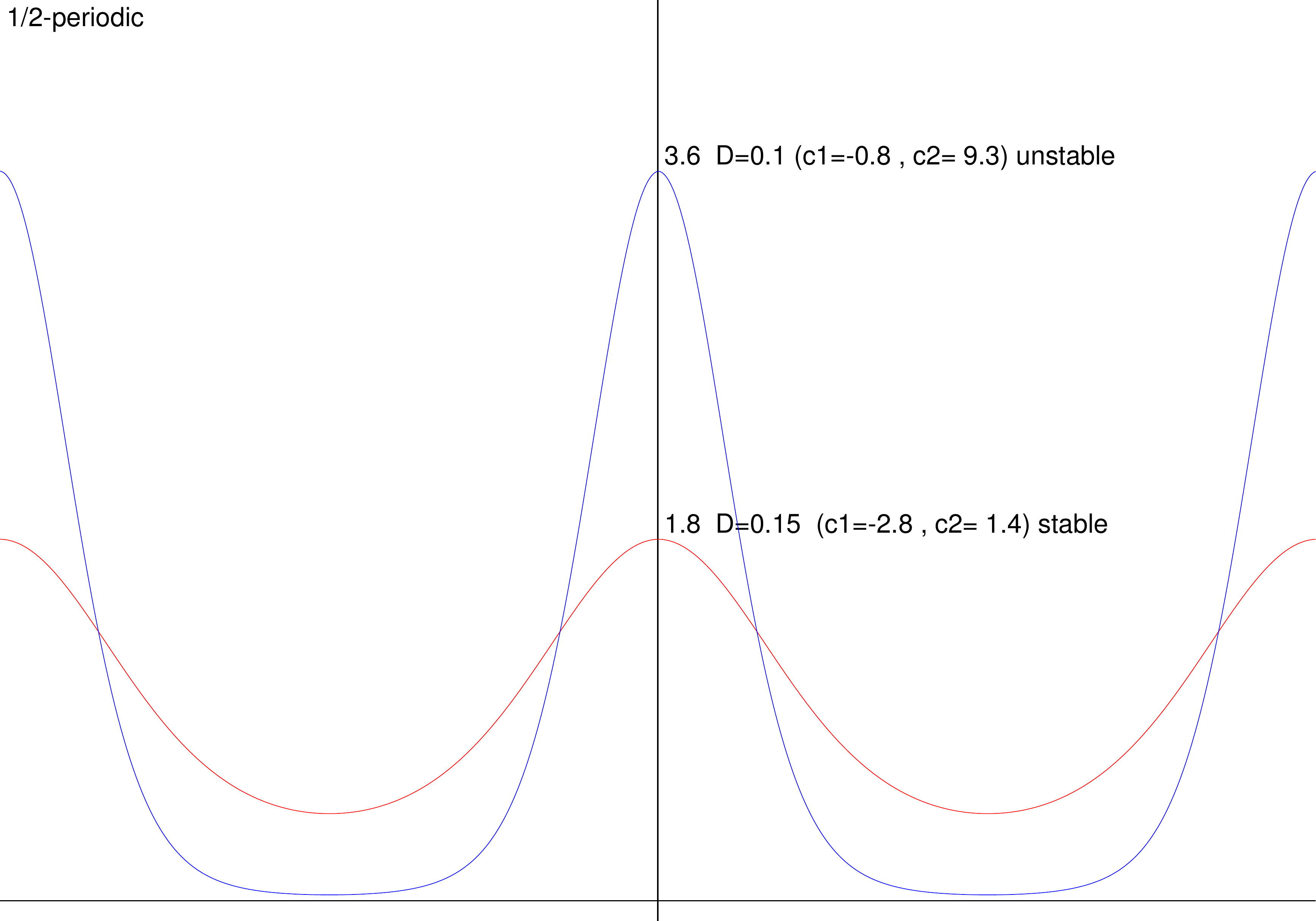}{0.48\textwidth}
  }
  \caption{\label{sinsin2}
  Stationary solutions for
  \( V(\theta) = \sin(2 \pi \theta) + \gamma \sin(4 \pi \theta) \)
  with \( \gamma = 2 \) (top) and \( \gamma = 4 \) (bottom)
  (\( V'(0) > 0, V'(\thalf) > 0, V(0.29) = 0 \) and \( V(0.27) = 0 \), respectively).
  (top left) Stable mixed mode solutions; the distance between the maxima is
  \( \half \). Note the backward bifurcation --- there is a stable stationary solution
  even though both eigenvalues $c_1$ and $c_2$ are negative!
  For \( D \ge 0.093 \) and \( \gamma = 2 \) we found no non-constant stationary
  solution.
  (top right) A \( \half \)-periodic stationary solution; it is unstable,
  but stable in the space of \( \half \)-periodic solutions.
  For \( D \ge \frac{1}{4 \pi} \) there are no non-trivial \( \thalf \)-periodic
  stationary solutions.
  (bottom left) Mixed mode solution with distance \( \half \) between the two maxima;
  for \( D = 0.15 \) we found no stationary mixed mode solution.
  (bottom right) \( \half \)-periodic solutions.
  In mixed mode solutions the larger maximum grows with decreasing \( D \) while the
  second maximum vanishes. In \( \thalf \)-periodic solutions the maxima grow with
  decreasing \( D \). In all cases the peaks become narrower.
  }
  \end{figure}
\end{example}

The second example shows non-trivial solutions when \( V'(0) \) and \( V'(\half) \)
are negative, and hence for zero diffusion coefficient neither one peak nor two peaks
at distance \( \thalf \) are stable by Theorem \ref{stab-n-peaks}.

\begin{example}
  Let
  \[ V(\theta) = \sin(2 \pi \theta) - \sin(4 \pi \theta) . \]
  Only the first eigenvalue is positive for small enough diffusion coefficient.
  Figure~\ref{negsin} shows how the stationary solution is approached.
  As one expects according to Corollary~\ref{twopeaks2} (which holds for \( D = 0 \)),
  for small diffusion coefficient it consists of two peaks with distance \( \theta_v \)
  (where \( V(\theta_v) = 0 \)). We checked numerically that indeed
  \( \tilde{V}'(0) = \half \, (V'(0) + V'(\theta_v)) > 0 \).

  This solution could be calculated only with the Fourier transformed system,
  since the iteration method does not converge --- it runs into a two-cycle.

  \begin{figure}
    \Gr{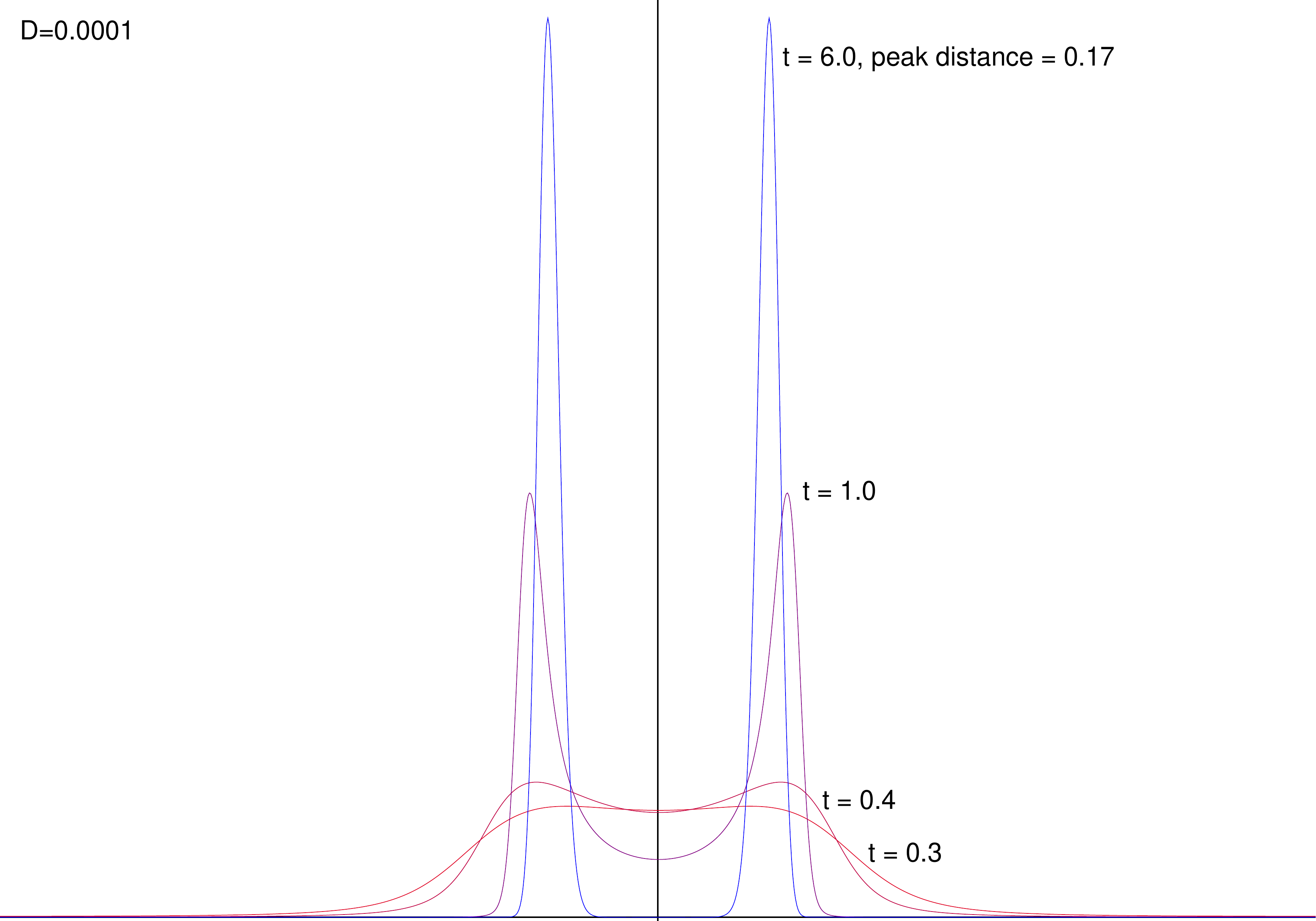}{0.5\textwidth}
  \caption{\label{negsin} Development of the stationary solution for
  \( V(\theta) = \sin(2 \pi \theta) - \sin(4 \pi \theta) \)
  started with \( f(\theta,0) = 1 + 0.8 \, \cos(2 \pi \theta) \).
  \( \theta_v = 0.17 \) is the only non-trivial zero of \( V \).
  The numerical computation started with 20 Fourier coefficients and ended with 140. }
  \end{figure}
\end{example}

The next example shows that one-peak and two-peaks like solutions are possible
in the same model at the same parameter values.
It is interesting that there exists a one-peak like solution although the first
eigenvalue (of the linearization near the homogeneous solution) is negative for all
parameter values.

\begin{example}
  Let
  \[ V(\theta) = \sin(4 \pi \theta) + \gamma \sin(6 \pi \theta) . \]
  For \( -\tfrac{2}{3} < \gamma < \tfrac{2}{3} \) the turning rate \( V \) has
  a single zero on \( \left]0,\half\right[ \), \( V'(0) > 0 \) and \( V'(\thalf) > 0 \);
  also, \( V'(\tfrac{1}{3}) > 0 \) for \( \gamma > \tfrac{1}{3} \).
  If \( \gamma > 0 \), then for all \( \delta > 0 \) there is an \( \eps > 0 \)
  such that
  \( \int_0^{\theta} V(\psi) \, d\psi > \eps \) for \( \delta < \theta \le \half \).
  Theorem 7.3.~in Primi et~al.~\cite{Primi} shows that a one-peak like solution exists
  for small enough \( D \).
  The second eigenvalue \( c_2 \) is positive for \( D < \tfrac{1}{8 \pi} \),
  the third eigenvalue \( c_3 \) is positive for \( D < \tfrac{\gamma}{12 \pi} \),
  which is \( \approx 0.013 \) for \( \gamma = 0.5 \).

  Figure~\ref{sin2sin3} shows stationary solutions (left side; calculated with the
  iteration scheme) and how they are approached in time (right side; Fourier based
  program).
  We see that one-peak and two-peaks like solutions are locally stable for
  small enough \( D \) and \( \gamma = 0.5 \).
  The one-peak like solution develops when the initial distribution is sufficiently
  centered (compare Theorem~\ref{smallsupport}), but in the simulations \( f(0,\cdot) \)
  did not have compact support.
  The three-peaks solution is stable in the subspace of \( \tfrac{1}{3} \)-periodic
  functions; it is unstable for \( D = 0.01 \) (data not shown);
  For \( D = 0.001 \) a solution with three peaks of different height, of which only
  two had distance \( \tfrac{1}{3} \), developed when the simulation was started with
  a perturbed three-peaks like function.
  Note that \( V_{(3)} \) actually satisfies the 3-peaks stability conditions of
  Theorem~\ref{theorem-stab-n-peaks}.

  With small diffusion coefficient the `typical' outcome of a simulation that
  is started with small deviations from \( f = 1 \) are one large and one small peak
  that are opposite.
  We suppose that for \( D > 0 \) these become equally high for large times;
  the smaller \( D \) is, the more time will be needed for that.
\end{example}

\begin{figure}
 \Gr{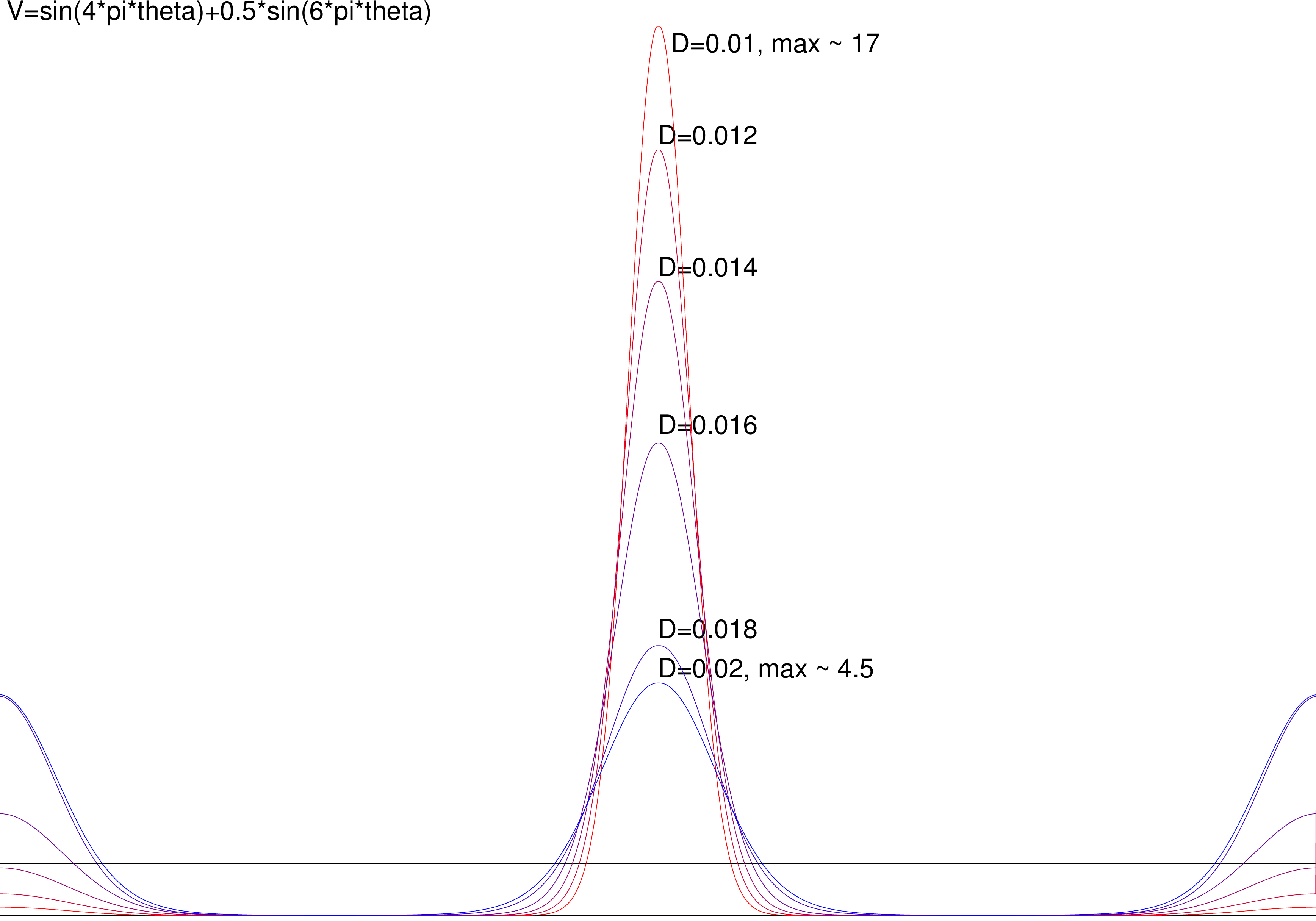}{0.48\textwidth} \hfill
 \Gr{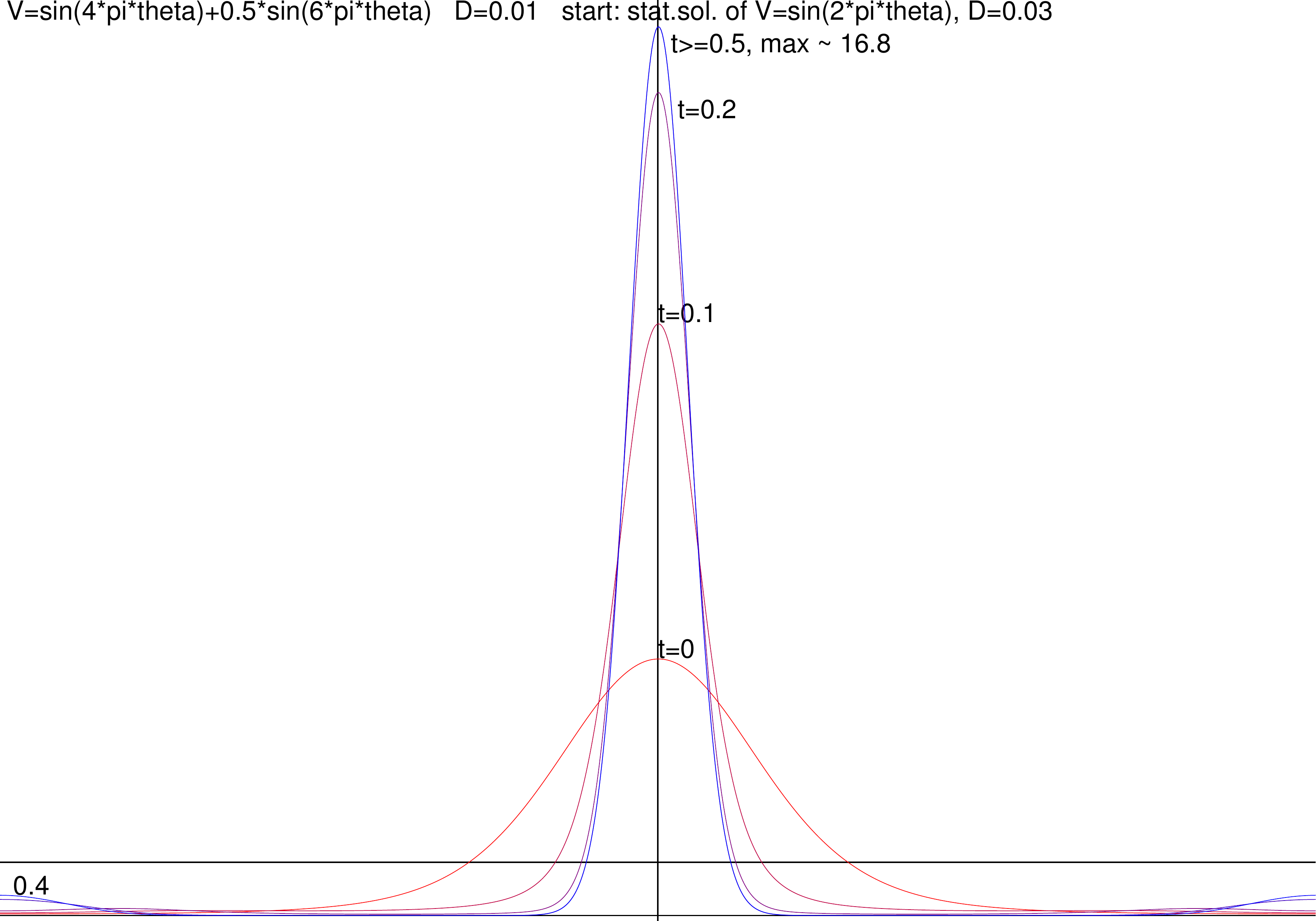}{0.48\textwidth}\\[2ex]
 \Gr{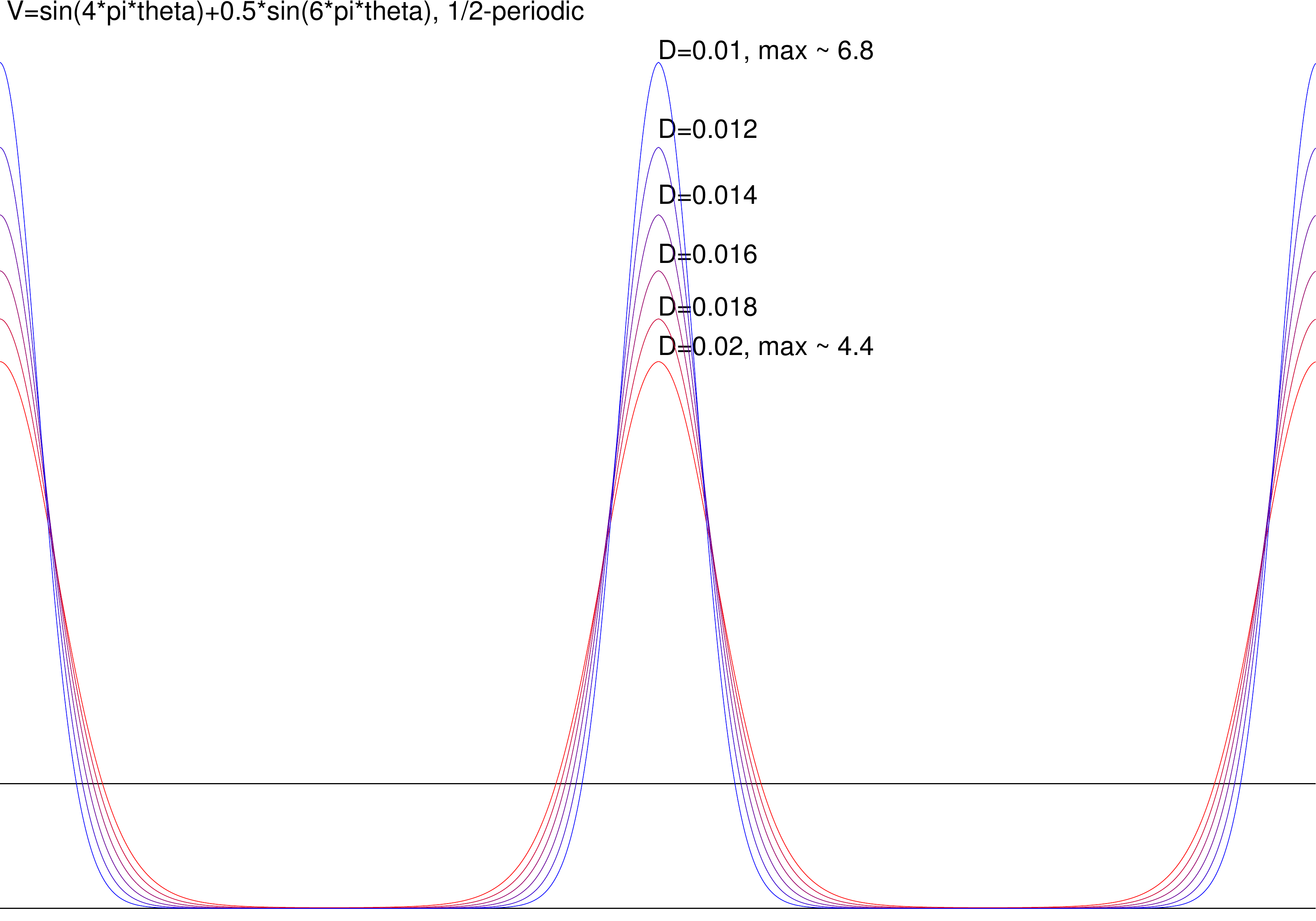}{0.48\textwidth} \hfill
 \Gr{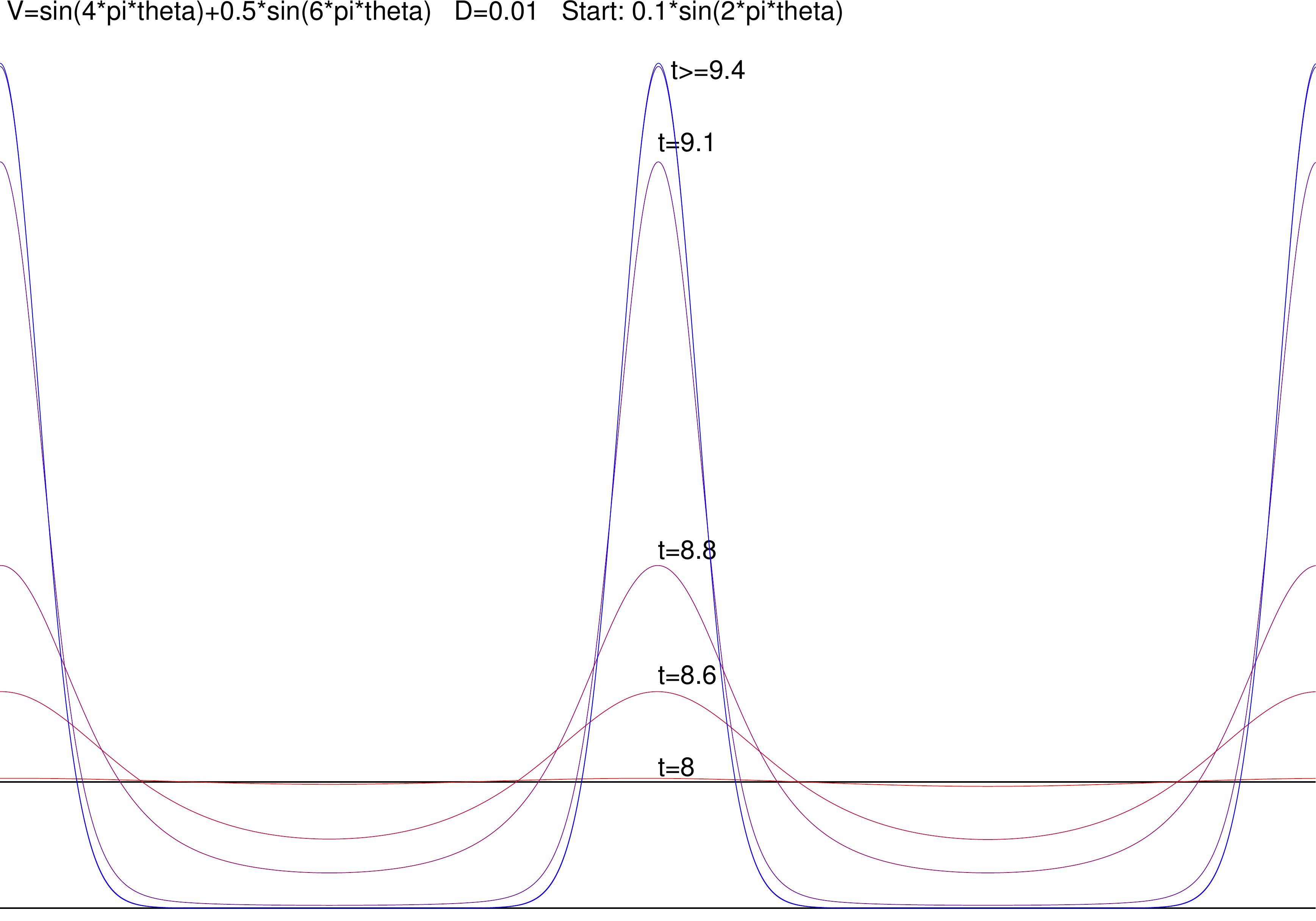}{0.48\textwidth}\\[2ex]
 \Gr{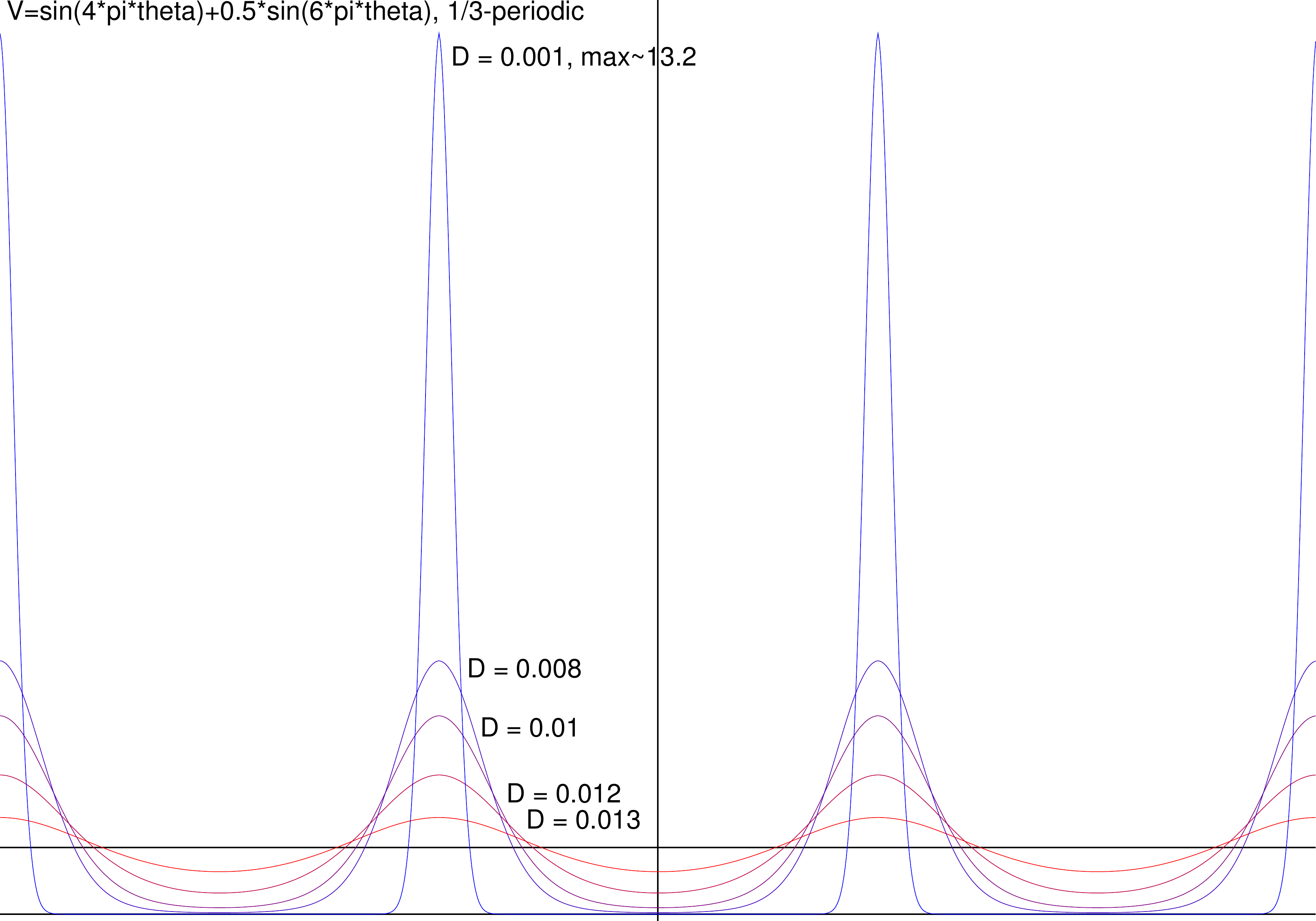}{0.48\textwidth}  \hfill
 \Gr{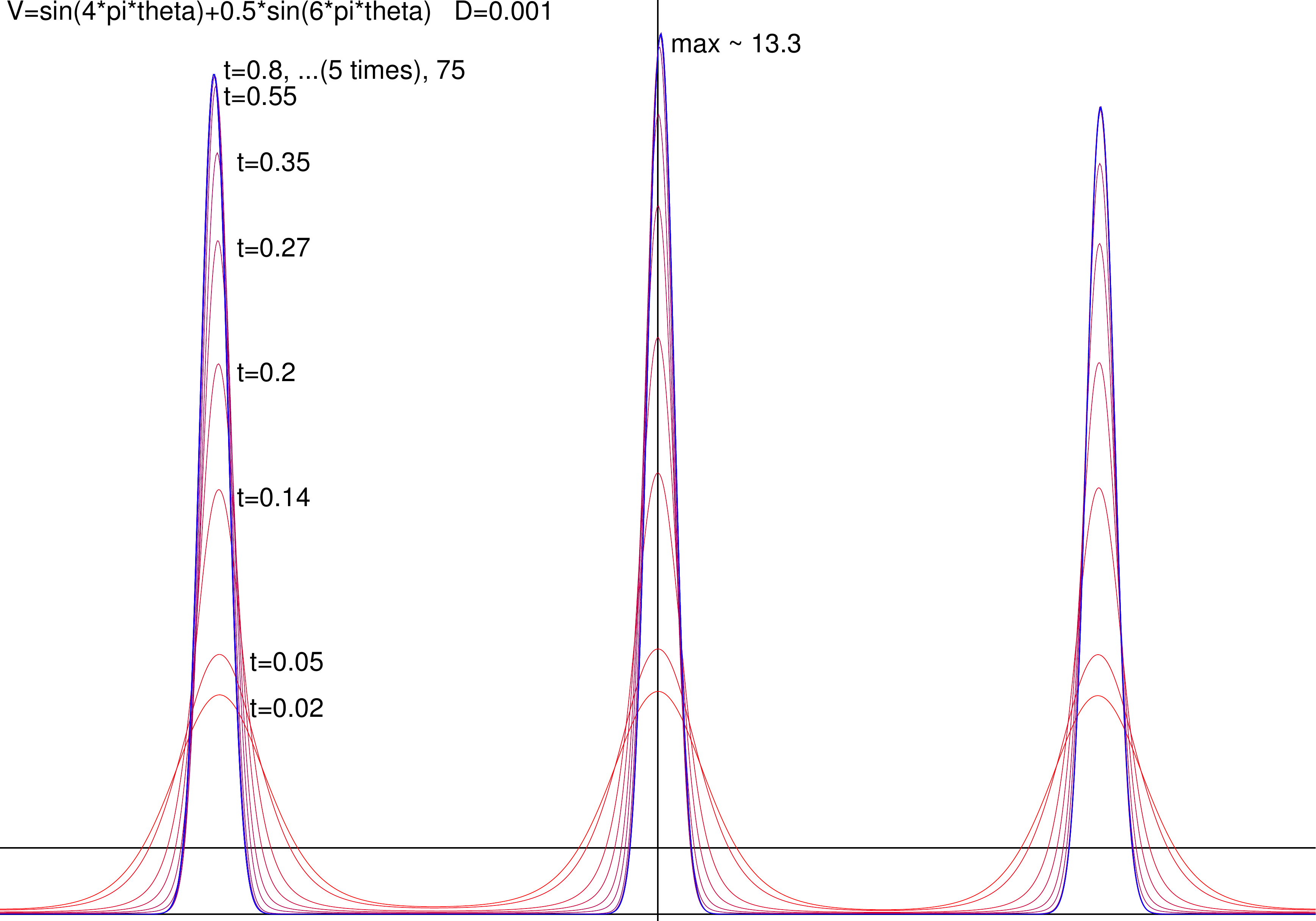}{0.48\textwidth} \\[2ex]
 \Gr{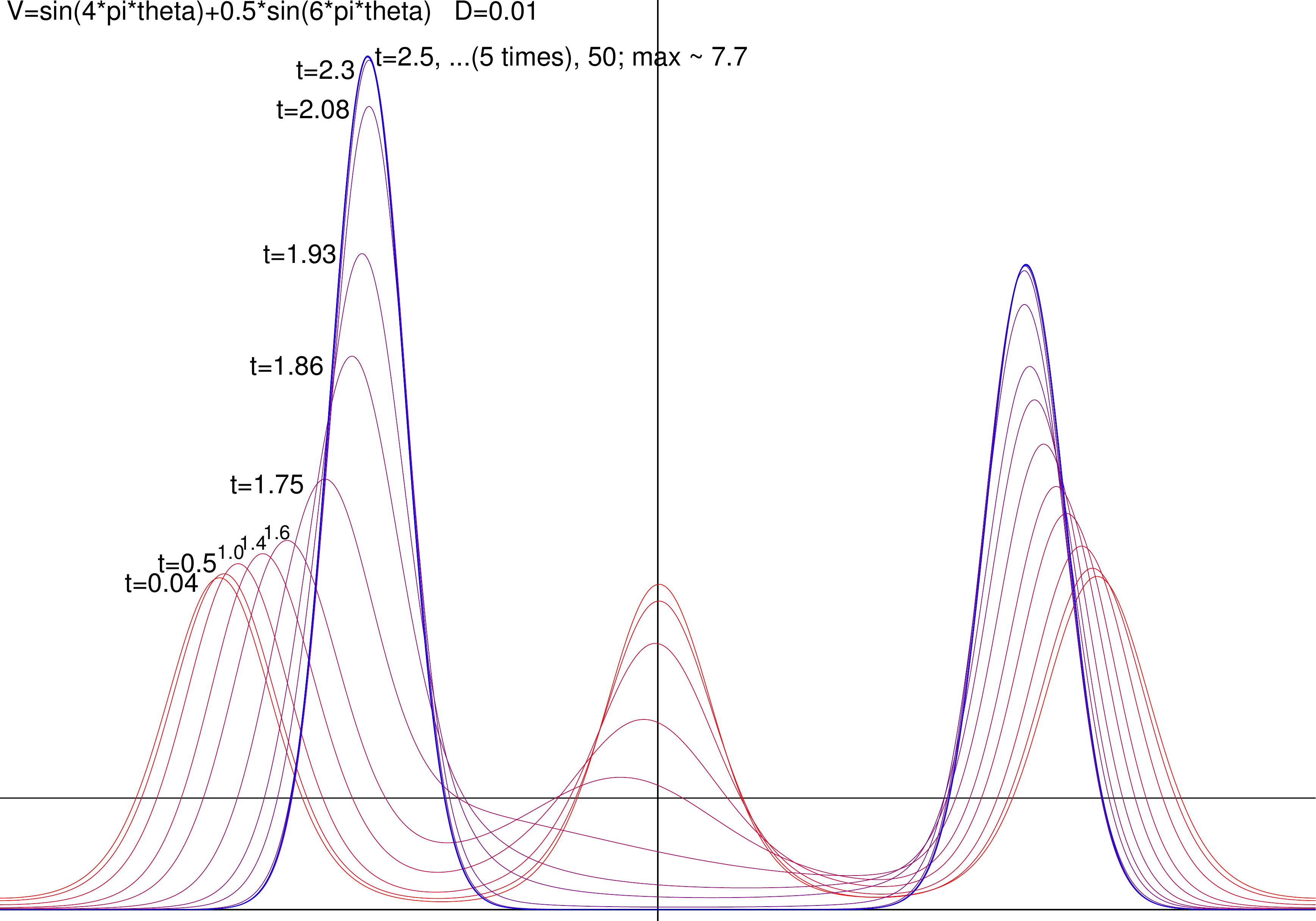}{0.48\textwidth}  \hfill
 \caption{\label{sin2sin3}
 \( V(\theta) = \sin(4 \pi \theta) + 0.5 \sin(6 \pi \theta) \).
 }
\end{figure}

{\bf Caption for Figure~\ref{sin2sin3}.} \quad
 Top row (left):
 These stationary solutions for various \( D \)-values
 were calculated with the iteration algorithm;
 for \( D >\approx 0.02\) the solutions look \( \thalf \)-periodic;
 (right): A one-peak like solution with a small second maximum develops fast
 when the simulation is started with a centered distribution;
 here we started with the stationary solution for
 \( V(\theta) = \sin(2 \pi \theta) \),
 \( D = 0.03 \).
 \newline
 Second row (left):
 These \( \thalf \)-periodic stationary solutions for various \( D \)-values
 were calculated with the iteration algorithm with forced \( \thalf \)-periodicity.
 (right) Approximation in time of a \( \thalf \)-periodic stationary solution.
 The simulation was started with \( f(0,\theta) = \phi \sin(2 \pi \theta) \),
 \( \phi = 0.1 \) (start not shown).
 Until \( t \approx 8 \) the distribution converges towards the
 homogeneous distribution (this is what it has to do --- any initial distribution
 that is orthogonal to the modes occurring in~$V$ dies out),
 then triggered by some numerical noise,
 instability of the constant solution takes over,
 and the second mode begins to grow.
% The same is true for larger values of \( \phi \).
 \newline
 Third row (left): 
 \( \tfrac{1}{3} \)-periodic solutions were calculated with the iteration scheme
 with forced \( \tfrac{1}{3} \)-periodicity.
 (right) The simulation was started with a small perturbation
 (\( \Re f_1(0) = 0.01 \), \( \Re f_2(0) = 0.005 \)) of the
 \( \tfrac{1}{3} \)-periodic solution
 for \( D = 0.01 \). A distribution with three slightly different peaks develops
 very fast, where the distances between first/second and first/third peak are
 {\em not} \( \tfrac{1}{3} \); then no further changes are discernible.
 \newline
 Bottom row: `Typical' result of a simulation,
 here started with the \( \tfrac{1}{3} \)-periodic solution for \( D = 0.01 \)
 which was perturbed by \( \Re f_1(0) = -0.01 \), \( \Re f_2(0) = -0.005 \).
 A distribution with two different maxima at distance \( \thalf \) develops fast,
 then no further changes are discernible.\\
{\em End of caption for Figure~\ref{sin2sin3}}

The last example shows that a variety of behaviors is possible if \( V < 0 \)
on \( \left]0,\half\right[ \).

\begin{example} \label{Ex:last}
  We compare
  \[ V_{(2)}(\theta) = \sin(4 \pi \theta) - \gamma \, \sin(2 \pi \theta)
      \quad (\gamma > 2 )
      \qquad \text{and} \qquad
    V_{(3)}(\theta) = \sin(6 \pi \theta) - \gamma \, \sin(2 \pi \theta)
      \quad (\gamma > 3) .
  \]
  In both cases \( V < 0 \) on \( \left]0,\half\right[ \), \( V'(0) < 0 \)
  and \( V'(\half) > 0 \).
  For \( V_{(2)} \) only the second eigenvalue is positive for
  \( D < \frac{1}{8 \pi} \);
  for \( V_{(3)} \) only the third eigenvalue is positive for
  \( D < \frac{1}{12 \pi} \);
  all other eigenvalues are negative.
  We have
  \( V_{(2),2} = 2 \, \sin(4 \pi \theta) > 0 \) on \( \left]0,\frac{1}{4}\right[ \)
  and
  \( V_{(3),3} = 3 \, \sin(6 \pi \theta) > 0 \) on \( \left]0,\frac{1}{6}\right[ \);
  all other \( V_{(j),n} \) are zero.

  Therefore we expect (at small enough diffusion coefficient \( D \)) for \( V_{(2)} \)
  stationary solutions with two equal maxima at distance \( \half \),
  and for \( V_{(3)} \) three equal maxima with distance \( \frac{1}{3} \).
  These develop indeed, but the time scales are interesting, see Figure~\ref{sinsin3}.
  Two, resp.~three different maxima develop very quickly but at unexpected distances;
  development toward equal distances and heights can be a very slow process.
  The explanation is that for both \( V \) there are orbits of other
  stationary solutions when \( D = 0 \) ---
  for \( V_{(2)} \) two peaks whose masses add to~1;
  for \( V_{(3)} \) three peaks whose positions and masses satisfy~\eqref{ode}, namely
  \(   0 = \sum_{k=1}^3 m_k V(\theta_j(t) - \theta_k(t)) \)
  for \( j=1,2,3 \) and \( m_1 + m_2 + m_3 = 1 \)
  (\( \theta_j \in S^1 \), \( m_j > 0 \)).
\end{example}

\begin{figure}
 \Gr{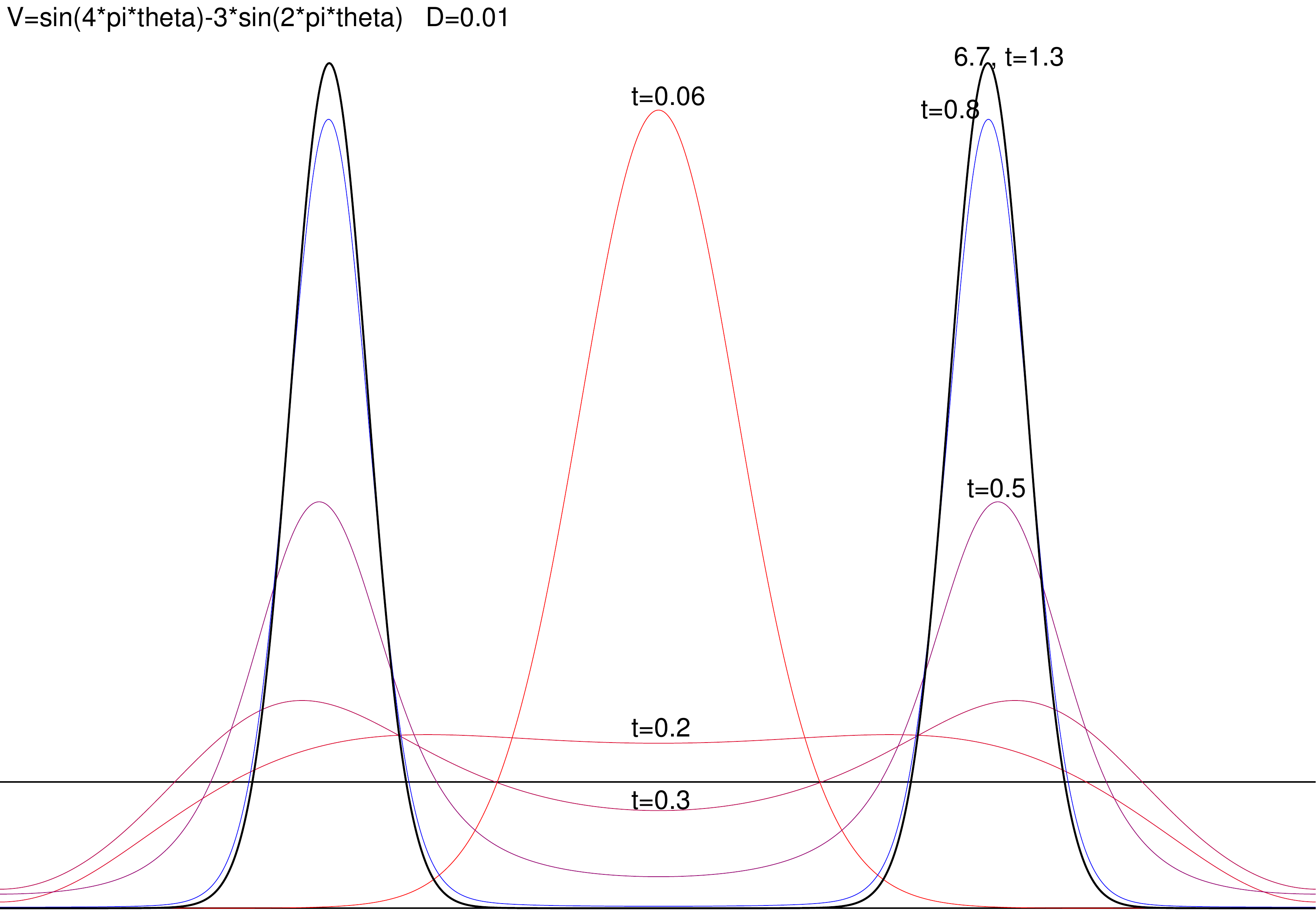}{0.48 \textwidth} \hfill
 \Gr{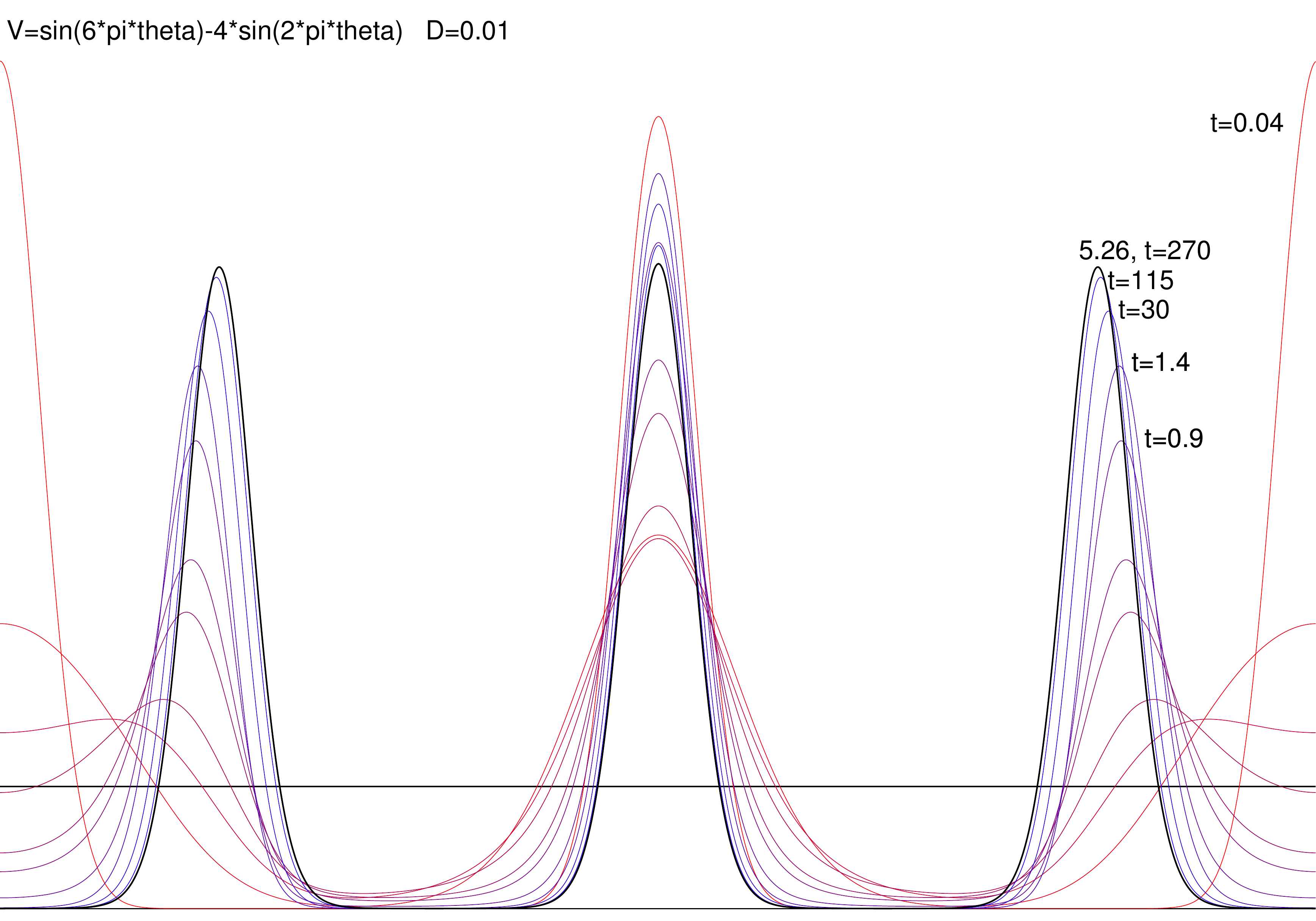}{0.48 \textwidth}
 \caption{\label{sinsin3}
 \( V(\theta) = \sin(4 \pi \theta) - 3 \, \sin(2 \pi \theta) \) (left),
 \( V(\theta) = \sin(6 \pi \theta) - 4 \, \sin(2 \pi \theta) \) (right)
 \( D = 0.01 \).
 Initial condition were one (left) and two (right) sharp peaks.
 Solutions were calculated with the Fourier-based algorithm.
 The iteration algorithm does not converge for these \( V \)'s
 (it runs into two-cycles);
 however, it converges to the shown stationary solutions if \( V_{(2),2} \)
 and \( V_{(3),3} \)
 and \( \half \)- and \( \frac{1}{3} \)-periodicity are used, respectively.
 }
\end{figure}

%%%%%%%%%%%%%%%%%%%%%%%%%%%%%%%%%%%%%%%%%%%%%%%%%%%%%%%%%%%%%%%%%%%%%%%%%%%%%%%%%%

\section{Discussion}
\label{discussion}

For the transport-diffusion equation~\eqref{tp} a wide variety of different patterns
has been observed. Indeed, only a limited number could be shown in the last section.
It emerges that it is nearly impossible to predict pattern formation only by knowing
the shape of \( V \); however, if one compiles information like the sign of the
eigenvalues \( c_k \), the zeros of \( V\!\), the signs of \( V'(\frac{j}{n}) \)
and the shapes of the \( V_n \) and of \( \int_0^\theta V_n(\psi) \, d\psi \),
then the picture becomes clearer.

If there is no diffusion, then we know quite something about the stability
or otherwise of peak solutions.
Stability of \( n \) peaks shows itself often also in the `short-time' behavior
of solutions of the diffusion-transport equation at small diffusion.
Therefore it is hard to clarify numerically whether a given stationary solution
is stable for small diffusion.
It is an open and interesting problem how to clarify the stability of stationary
solutions if diffusion is present and if several eigenvalues are positive.

A possible interest in the transport-diffusion equation (TDE) comes from its relation
to the following integro-differential equation (IDE) for a function
\( f : [0,\infty[ \times S^1 \to \R^+\):
\begin{equation}
\label{IDE}
 \frac{\partial\!f}{\partial t}(t,\theta)
   = - M f(t,\theta)
         + \int_{S^1} \int_{S^1} G_{\sigma}(\theta-\theta_o - V(\theta_i-\theta_o)) 
              \, f(t,\theta_i) \, f(t,\theta_o) \, d\theta_o \, d\theta_i ,
\end{equation}
where \( M = \int_{S^1} f(0,\theta) \, d\theta\), \( \sigma > 0 \),
\( G_\sigma : S^1 \to \R^+ \) is the periodic Gaussian
with \( \int_{S^1} G_\sigma(\theta) \, d\theta = 1 \),
\[ G_\sigma(\theta) =
    \frac{1}{\sqrt{2 \pi} \sigma} \sum_{n \in \Z} 
      e^{-\frac{1}{2} \left( \frac{\theta + n}{\sigma} \right)^2 } ,
\]
and the turning function \( V : S^1 \to S^1 \) is odd.

This IDE describes a jump process in which particles at an old orientation
\( \theta_o \) interact over \( S^1 \) with particles in \( \theta_i \)
and jump to a new position \( \theta = \theta_o + V(\theta_i-\theta_o) \).
The precision of the jump is measured by \( \sigma \).
Note that for the IDE \( V \) is a turning (therefore it maps to~$S^1$),
while the function \( V \) for the TDE is a velocity and takes real values
that can be arbitrarily large.
If, e.g.,~\( V(\psi) = \psi \) in the IDE, then all solutions
converge for \( t \to \infty \) to the constant solution (Geigant~\cite{Geigant99}),
while the TDE has non-constant stable stationary solutions.

Both equations preserve mass, positivity, axial symmetry and any periodicity,
and both are invariant under translations and reflections.
The \( SO(2) \)-invariance makes linearization and calculation of eigenvalues
near the stationary homogeneous solution possible, as well as the fast numerical
calculation of solutions by Fourier transforming the equation into a system of ODEs
(see Geigant and Stoll~\cite{GeigantStoll} for the IDE).

Let \( M = 1 \).
If \( V = 0 \), then the solutions of both systems converge to the constant~1
as \( t \to \infty \) (the TDE is the linear diffusion equation,
the IDE a linear jump process).
If \( D \) or \( \sigma \) are large compared to \( V\!\), solutions also converge to~1
(Theorem~\ref{thm:stability} for the TDE; Geigant~\cite{Geigant99} for the IDE).
Therefore, if \( V \) is small, then $D$ and~\( \sigma \), resp., must be very small
for pattern formation.
On the other hand, if \( V = 0 \) and \( D = 0 \) or \( \sigma = 0 \), resp.,
then \( \frac{\partial\!f}{\partial t} = 0 \), therefore nothing happens.
Hence, if \( V \) as well as \( D \) or \( \sigma \) are very small,
then pattern formation occurs very slowly (if at all).
Last but not least, if \( D = 0 \) or \( \sigma = 0 \) but \( V \neq 0 \),
the limiting equations of both equations have delta distributions as solutions.

This said, we assume that \( \sigma \) and \( V \) are very small,
and we use Taylor expansion in \( \sigma, V \) to get
\[ G_{\sigma}(\theta-\theta_o - V(\theta_i-\theta_o))
   = \delta(\theta-\theta_o) - V(\theta_i-\theta_o) \delta'(\theta-\theta_o)
     + \frac{\sigma^2}{2} \delta''(\theta-\theta_o) + O\bigl((\sigma^2+|V|)^2\bigr) .
\]
Plugging this right hand side into~\eqref{IDE} yields the transport-diffusion
equation~\eqref{tp} with \( D = \sigma^2/2 \), because
\begin{align*}
  \int_{S^1} \delta(\theta-\theta_o) f(\theta_i) f(\theta_o) \,d\theta_i\,d\theta_o
   &= f(\theta) , \\
  \int_{S^1} \int_{S^1} \delta'(\theta-\theta_o) (V(\theta_i-\theta_o)
           f(\theta_o)) \,d\theta_o \: f(\theta_i)\,d\theta_i
   &= \int_{S^1} \frac{d}{d \theta_o} \bigl(V(\theta_i-\theta_o)
                                       f(\theta_o)\bigr)\Bigm|_{\theta_o=\theta}
                 f(\theta_i) \,d\theta_i  \\
   &= -\int_{S^1} \bigl(V'(\theta_i-\theta) f(\theta)
                        - V(\theta_i-\theta) f'(\theta)\bigr)
                  f(\theta_i) \,d\theta_i \\
   &= -\frac{d}{d \theta} \bigl((V*f)(\theta) \, f(\theta)\bigr) ,
   \intertext{and}
 \int_{S^1} \delta''(\theta-\theta_o) f(\theta_o) f(\theta_i) \,d\theta_o\,d\theta_i
   &= \int_{S^1} f(\theta_i) \, d\theta_i \: f''(\theta)
    = f''(\theta) .
\end{align*}
Different arguments for this derivation are given in Mogilner and
Edelstein-Keshet~\cite{Mogilner1} and in Primi et al.~\cite{Primi}.

Similarly, for the eigenvalues \( \tilde{c}_k \) of~\eqref{IDE}
(see Geigant and Stoll~\cite{GeigantStoll}), Taylor expansion with
small \( V \) yields
\begin{align*}
 \tilde{c}_k
   &= -1 + 4 G_{\sigma,k} \int_{0}^{\half} \cos(\pi k \psi)
                            \cos\bigl(2 \pi k (\thalf \psi - V(\psi))\bigr) \,d\psi \\
   &\approx -1 + 4 G_{\sigma,k} \int_0^{\half} \cos(\pi k \psi) \cos(\pi k \psi)\,d\psi
            + 8 \pi^2 k^2 G_{\sigma,k}
            \int_0^\half \cos(\pi k \psi) \sin(\pi k \psi) V(\psi) \,d\psi \\
   &= (-1 + G_{\sigma,k})
       + 4 \pi^2 k^2 G_{\sigma,k} \int_{0}^{\half} V(\psi) \sin(2 \pi k \psi) \,d\psi \\
   &\stackrel{\sigma \to 0}{\longrightarrow}
      4 \pi^2 k^2 \int_{0}^{\half} V(\psi) \sin(2 \pi k \psi) \,d\psi ,
\end{align*}
because the \( k \)-th Fourier coefficient $G_{\sigma,k}$ tends to~1 as
\( \sigma \to 0 \).
Because \( c_k = 4 \pi k \int_0^{\half} V(\theta) \sin(2 \pi k \theta) \, d\theta \)
for \( D = 0 \) (see~\eqref{fou_k}), the signs of the eigenvalues of both models agree
for small enough \( V\!\), \( D \) and \( \sigma \) (similar arguments were given by
I.~Primi, personal communication).
We stress again that for larger \( V \) or \( D, \sigma \),
the signs of the eigenvalues may differ.

But we see for example that for both models there are turning functions \( V \)
that are negative on \( \left]0,\thalf\right[ \) but lead to non-trivial patterns,
see the Example~\ref{Ex:last} in Section~\ref{examples}.
Especially for the IDE this was a surprise to us.
Only the eigenvalue of the first mode in the IDE is always negative if \( V \)
is negative, which may correspond to the result of Primi et~al.~\cite{Primi} that
there are no one-peak like solutions for small diffusion
if \( \int_0^\theta V(\psi) \,d\psi \) is negative somewhere.

The formulas for the eigenvalues show also that higher modes have larger eigenvalues
for the IDE (\( k^2 \) versus \( k \) in TDE).
This explains perhaps why in simulations of the IDE at small \( \sigma \) we see
the initial formation of several peak-like maxima much more often than in simulations
of the TDE with small diffusion \( D \).

Both equations have limiting equations for \( D \to 0 \) and \( \sigma \to 0 \),
respectively.
For \( \sigma = 0 \) we have \( G_\sigma = \delta_0 \), and the limiting equation is
\begin{equation} \label{sigma=0}
  \frac{\partial\!f}{\partial t}(t,\theta)
    = - f(t,\theta) + \int_{S^1} f(t,\theta-V(\psi)) f(t,\theta+\psi-V(\theta))\,d\psi.
\end{equation}
In Geigant~\cite{Geigant2000} it is shown that for \( \sigma \to 0 \) the solutions
of the IDE converge to those of the limiting equation on fixed finite time intervals.

For both limiting equations a single peak is a stationary solution,
which is linearly stable if \( V \) is attracting.
`Attraction' in the case of the IDE means \( 0 < V(\theta) < \theta \)
for \( 0 < \theta < \thalf \) (see Geigant~\cite{Geigant2000}%
\footnote{In Theorem 3.1.\ of~\cite{Geigant2000} there is a typing mistake:
`attracting' must be defined as given here and in the definition on page~1211 in~\cite{Geigant2000}.}),
and in the case of the TDE \( V(\theta) > 0 \), \( V'(0) > 0 \)
for \( 0 < \theta < \thalf \), see Theorem~\ref{theorem-linstab}.
In both equations the perturbation may not extend to the opposite side of the peak
since particles located there cannot turn back (because \( V(\half) = 0 \)).

Two peaks with distance \( \thalf \) whose masses add up to 1 are also
a stationary solution for both limiting equations because \( V(\thalf) = 0 \).
For the IDE with \( \sigma = 0 \) Kang et~al.\ formulate theorems
on convergence of solutions to two opposite peaks if the initial distribution
is sufficiently localized, see Theorems 15 and~19 in~\cite{Kang}.
However, there is an estimate in both proofs, namely equations (24) and~(43),
where we cannot follow the argument --- they seem to bound a delta distribution
by a constant. Unfortunately, this estimate is very important for the proofs.
We have been informed by the authors that an erratum is in preparation.

For both equations the assumptions for convergence to two opposite peaks are
essentially an attracting shape of \( V \) near 0 (\( V > 0 \) to the right of~0,
\( V'(0) > 0 \), and for the IDE additionally \( V'(0) < 1 \) near 0) and near
\( \thalf \) (\( V < 0 \) to the left of~\( \thalf \), \( V'(\thalf) > 0 \), and
for the IDE additionally \( V'(\thalf) < 1 \)).

It is important to see that in both limiting equations there is {\em no}
`mass selection' toward equal masses of the peaks.
The open question is then on what time scales equalization of the peaks occurs
when \( \sigma \) or \( D \), resp., are positive.

The central {\em differences} between the two limiting equations for the TDE and~IDE
are as follows.
\begin{itemize}
  \item \( n \ge 2 \) initial peaks, i.e.,
        \( f(0,\cdot) = \sum_{k=1}^n m_k \delta(.-\theta_k) \) with $m_k > 0$,
        do not keep that form for the IDE
        (e.g., starting with two peaks in \( \theta_1, \theta_2 \),
        particles jump also to positions \( \theta_1 + V(\theta_2-\theta_1) \)
        and \( \theta_2 + V(\theta_1-\theta_2) \)).
  \item \( n \ge 3 \) peaks --- even if equidistant and with equal masses ---
        are in general {\em not} a stationary solution for the IDE.
\end{itemize}
Therefore, the IDE does not allow the `peak game' (see
Section~\ref{stability-of-position}; terminology by Mogilner et~al.~\cite{Mogilner2}).
Only if \( V(\tfrac{j}{n}) = 0 \) for \( j = 1,\ldots,n-1 \), then \( n \)
equidistant peaks with {\em arbitrary} masses are a stationary solution
of equation~\eqref{sigma=0}.
It is an educated guess that they are locally stable {\em up to redistribution of mass
and reorientation} if \( 0 < V'(\tfrac{j}{n}) < 1 \) holds for \( 0 \le j \le n-1 \).

%========================================================================

\subsection*{Acknowledgments.}
We thank Dr.~Ivano Primi for many helpful explanations and discussions.
We also thank our son Robin Stoll for programming all interactive input,
output and plotting routines for the numerical schemes.

%%%%%%%%%%%%%%%%%%%%%%%%%%%%%%%%%%%%%%%%%%%%%%%%%%%%%%%%%%%%%%%%%%%%%%%%%%%%%%

\end{document}